\documentclass[11pt]{amsart}

\usepackage{amsmath,amssymb,amsthm,amsfonts,amsrefs}
\usepackage[margin=1in,marginparwidth=0.8in, marginparsep=0.1in]{geometry}
\usepackage{graphicx}
\usepackage[all]{xy}

\usepackage{enumerate}
\usepackage{tikz}
\usepackage{tikz-cd}
    \usetikzlibrary{arrows}

\usepackage{xcolor,float}
\usepackage{mathrsfs}
\usepackage{caption}

\usepackage{verbatim}

\usepackage{mathtools}
\usepackage[makeroom]{cancel}
\usepackage{stmaryrd}
\usepackage[bottom]{footmisc}
\usepackage{enumitem}
\usepackage{todonotes}
\usepackage[hidelinks]{hyperref}
\usepackage[title]{appendix}
\usepackage{float} 

\usepackage{fancyhdr}
\usepackage{mathrsfs}
\newtheorem{definition}{Definition}[section]
\newtheorem{example}[definition]{Example}

\newtheorem{definition_theorem}[definition]{Definition-Theorem}

\theoremstyle{definition}
\newtheorem{conjecture}[definition]{Conjecture}
\newtheorem{remark}[definition]{Remark}

\theoremstyle{plain}
\newtheorem{theorem}[definition]{Theorem}
\newtheorem{lemma}[definition]{Lemma}
\newtheorem{proposition}[definition]{Proposition}
\newtheorem{corollary}[definition]{Corollary}

\makeatletter
\newcommand*\bigcdot{ {\mathpalette\bigcdot@{.5}} }
\newcommand*\bigcdot@[2]{\mathbin{\vcenter{\hbox{\scalebox{#2}{$\m@th#1\bullet$}}}}}
\makeatother

\newcommand\NN{ \mathbb{N} }

\newcommand\RR{ \mathbb{R} }

\newcommand\ZZ{ \mathbb{Z} }

\newcommand\cL{ \mathcal{L} }

\newcommand\cS{ \mathcal{S} }

\newcommand\cU{ \mathcal{U} }
\newcommand\cV{ \mathcal{V} }

\newcommand\sA{ \mathscr{A} }
\newcommand\sB{ \mathscr{B} }
\newcommand\sC{ \mathscr{C} }

\newcommand\pre{ {\operatorname{pre}} }
\newcommand\Rp{ { \mathbb{R}_{>0} } }

\newcommand\Fun{ {\operatorname{Fun}} }
\newcommand\id{\operatorname{id}}

\newcommand\Ind{ {\operatorname{Ind} } }
\newcommand\Hom{\operatorname{Hom}}

\newcommand\PrLcs{  {  \operatorname{Pr}^{\operatorname{L} }_{\omega,st} } }
\newcommand\PrLst{  {  \operatorname{Pr}^{\operatorname{L} }_{st} } }

\newcommand\PrRst{  {  \operatorname{Pr}^{\operatorname{R} }_{st} } }

\newcommand\clmi[1]{  \underset{ {#1} }{\operatorname{colim}} }

\newcommand\lmi[1]{  \underset{ {#1} }{\lim} }

\newcommand\Sp{ {\operatorname{Sp} } }
\newcommand\st{   {\operatorname{st}} }

\newcommand\Loc{ {\operatorname{Loc}} }
\newcommand\mhom{\operatorname{\mu hom}}
\newcommand\msh{\operatorname{\mu sh}}
\newcommand\ms{\operatorname{SS}}
\newcommand\msif{\operatorname{SS}^{\infty}}
\newcommand\msnz{ \dot{\operatorname{SS} }}
\newcommand\Op{\operatorname{Op}}

\newcommand\Sh{\operatorname{Sh}}

\newcommand\sHom{\mathscr{H}om}

\newcommand\supp{ {\operatorname{supp}} }
\newcommand\wsh{\operatorname{\mathfrak{w} sh}}

\newcommand\wrap{\mathfrak{W}}

\newcommand\Coh{\operatorname{Coh}}

\newcommand\PrLV[1][\cV]{  {  \operatorname{Pr}^{\operatorname{L} }_{\cV,st} } }

\newcommand\VD[1]{ {\operatorname{D}_{#1}  } }
\newcommand\ND[1]{ {\operatorname{D}_{#1}^\prime  } }
\newcommand\dT{ {\dot{T} } }

\newcommand{\ol}{\overline}

\newcommand{\bR}{\mathbb{R}}

\begin{document}

\title[Spherical adjunction and Serre functor from microlocalization]{\textbf{Spherical Adjunction and Serre Functor from Microlocalization}\\\vspace{3mm} {\textbf{\footnotesize{-- An approach by contact isotopies --}}}} 
\date{}
\author{Christopher Kuo}
\address{Department of Mathematics, University of Southern California}
\email{chrislpkuo@berkeley.edu} 
\author{Wenyuan Li}
\address{Department of Mathematics, Northwestern University.}
\email{wenyuanli2023@u.northwestern.edu}
\maketitle

\begin{abstract}
    For a subanalytic Legendrian $\Lambda \subseteq S^{*}M$, we prove that when $\Lambda$ is either swappable or a full Legendrian stop, the microlocalization at infinity $m_\Lambda: \Sh_\Lambda(M) \rightarrow \msh_\Lambda(\Lambda)$ is a spherical functor, and the spherical cotwist is the Serr{e} functor on the subcategory $\Sh_\Lambda^b(M)_0$ of  compactly supported sheaves with perfect stalks. This is a sheaf theory counterpart (with weaker assumptions) of the results on the cap functor and cup functor between Fukaya categories. When proving spherical adjunction, we deduce the Sato-Sabloff fiber sequence and construct the Guillermou doubling functor for microsheaves on any open subsets of the Legendrian and with respect to any Reeb flow.  
\end{abstract}

\setcounter{tocdepth}{2}
\tableofcontents

\section{Introduction}

\subsection{Context and background}
    Our goal in the series of paper \cite{Kuo-Li-duality,Kuo-Li-Calabi-Yau} is to investigate non-commutative geometry structure of the category of sheaves arising from the symplectic geometry structure on the Lagrangian skeleton of the Weinstein pair $(T^*M, \Lambda)$, where $\Lambda \subseteq S^*M$ is a subanalytic Legendrian subset in the ideal contact boundary $S^*M$ of the exact symplectic manifold $T^*M$.

    Following Kashiwara-Schapira \cite{KS}, given a real analytic manifold $M$, one can define a stable $\infty$-category $\Sh_\Lambda(M)$ of constructible sheaves on $M$ with subanalytic Legendrian singular support $\Lambda \subseteq S^*M$, which is invariant under Hamiltonian isotopies \cite{Guillermou-Kashiwara-Schapira}. On the other hand, one can also define another stable $\infty$-category $\msh_\Lambda(\Lambda)$ of microlocal sheaves on the Legendrian $\Lambda \subseteq S^*M$ \cite{Gui,Nadler-pants}, and there is a microlocalization functor
    $$m_\Lambda: \Sh_\Lambda(M) \rightarrow \msh_\Lambda(\Lambda).$$

    From the perspective of non-commutative geometry, the categories of (microlocal) sheaves, as sheaves on the Lagrangian skeleton of the Weinstein pair $(T^*M, \Lambda)$, should be understood as a non-commutative manifold with boundary, abiding Poincar\'{e}-Lefschetz duality and fiber sequence with the presence of a relative Calabi-Yau structure \cite{relativeCY}, whose existence is proved by Shende-Takeda using arborealization \cite{ShenTake}. 
    In this paper, we will study some non-commutative geometry structures, which are related but different from the duality fiber sequence in Calabi-Yau structures.

    The category of (micro)sheaves is closely related to a number of central topics in symplectic geometry and mathematical physics \cite{NadZas,Nad}. 
    Recently Ganatra-Pardon-Shende \cite{Ganatra-Pardon-Shende3} showed that the partially wrapped Fukaya category is equivalent to the category of compact objects in the unbounded dg category of sheaves. In particular, for contangent bundles with Legendrian stops,
    $$\mathrm{Perf}\,\mathcal{W}(T^*M, \Lambda)^\text{op} \simeq \Sh^c_\Lambda(M).$$
    The homological mirror symmetry conjecture \cite{KonHMS,AurouxAnti} predicts an equivalence of the Fukaya categories and categories of coherent sheaves on the mirror complex variety. 

    We will now explain the non-commutative geometry structures that will be investigated, namely Serre duality and spherical adjunctions, which arise from various predictions from symplectic geometry, mirror symmetry and mathematical physics.

\subsubsection{Context of spherical adjunction}
    Spherical adjunctions are introduced by Anno-Logvinenko \cite{Spherical} in the dg setting and then generalized \cite{SphericalInfty} in the stable $\infty$ setting, as a generalization of the notion of spherical objects \cite{SeidelThom}. Like spherical objects, spherical adjunctions provide interesting fiber sequences and autoequivalences of the categories called spherical twists and cotwists.

    In algebraic geometry, when we have a smooth variety $X$ with a divisor $i: D \hookrightarrow X$, the push forward functor and pull back functor
    $$i_*: \Coh(D) \leftrightharpoons \Coh(X) : i^*$$
    form a spherical adjunction between the dg categories of coherent sheaves, where the spherical twist is $-\otimes \mathcal{O}_X(D)$.

    In symplectic geometry, as is suggested by Kontsevich-Katzarkov-Pantev \cite{KonKarPanHodge} and Seidel \cite{SeidelSH=HH}, we have another interesting class of spherical adjunctions inspired by long exact sequences in Floer theory \cite{SeidelLES,SeidelFukI,Sabduality,EESduality}.
    For a symplectic Lefschetz fibration $\pi: X \rightarrow \mathbb{C}$ with regular fiber $F = \pi^{-1}(\infty)$, let $\mathcal{FS}(X, \pi)$ be the Fukaya-Seidel category associated to Lagrangian thimbles in $X$ and $\mathcal{F}(F)$ the Fukaya category of closed exact Lagrangians in $F$ \cite{Seidelbook}. The cap functor
    $$\cap_F: \mathcal{FS}(X, \pi) \rightarrow \mathcal{F}(F)$$
    defined by intersection of the Lagrangians with $F$ admits a left adjoint $\cup$ called the (left) cup functor \cite{AbGan} (see also \cite{AbSmithKhov}*{Appendix A}). In an unpublished work, Abouzaid-Ganatra proved that $\cap$ and $\cup$ form a spherical adjunction for general symplectic Landau-Ginzburg models \cite{AbGan}. 
    On the other hand, using the formalism of partially wrapped Fukaya categories \cite{Sylvan,Ganatra-Pardon-Shende1}, Sylvan considered the Orlov cup functor
    $$\cup_F: \mathcal{W}(F) \rightarrow \mathcal{W}(X, F)$$
    associated to any Weinstein pair $(X, F)$ and showed that the $\cup$ is a spherical functor\footnote{The data of a spherical functor is equivalent to the data of a spherical adjunction, as will be explained in Section \ref{sec:spherical}. Here we use spherical functors because the adjoint functor is not explicitly constructed Sylvan's work.} as long as the Weinstein stop $F \subset \partial_\infty X$ is a so called swappable stop \cite{SylvanOrlov}. In this case, the spherical twists/cotwists are the monodromy functors defined by wrapping around the contact boundary.  
    
    In microlocal sheaf theory, Nadler has also shown that functors between the pair of microsheaf categories over the symplectic Landau-Ginzburg model $(\mathbb{C}^n, \pi = z_1 \dots z_n)$, aftering (heuristically speaking) adding additional fiberwise stops, form a spherical adjunction. Then, by removing the fiberwise stops, the spherical adjunction for the original pair is also obtained \cite{Nadspherical}, but it is unclear how general this argument is in sheaf theory.

    The structure of spherical adjunctions has appeared in a number of previous works and leads to interesting applications in homological mirror symmetry \cite{AbAurouxHMS,Nadspherical,Gammagespherical,Jeffs}. However, there are still important problems that remain. Namely, the cap functor $\cap$ is only defined in the setting of a symplectic Landau-Ginzburg model instead of a general Weinstein pair, since it is completely not clear whether there are enough Lagrangian submanifolds asymptotic to a general Legendrian stop \cite{Ganatra-Pardon-Shende3}. This means that even though the cup functor $\cup$ is proved to be spherical as long as the stop is swappable \cite{SylvanOrlov}, it is difficult to characterize the adjoint functors geometrically.

\subsubsection{Context of Serre functor}
    On the other hand, in algebraic geometry, the dualizing sheaf induces Serre duality. When $X$ is smooth, the Serre functor is given by the canonical bundle $- \otimes \mathcal{O}_X(K_X)$.
    
    From the perspective of Fukaya categories, following a proposal of Kontsevich, Seidel has conjectured \cite{SeidelSH=HH} that for a symplectic Lefschetz fibration, the spherical cotwist is the inverse Serr{e} functor
    $$\sigma: \mathcal{FS}(X, \pi) \rightarrow \mathcal{FS}(X, \pi),$$
    and proved partial results \cite{SeidelFukI,SeidelFukII,SeidelFukIV1/2}, while from the perspective of Legendrian contact homology, Ekholm-Etnyre-Sabloff have proved Sabloff duality \cite{Sabduality,EESduality} between linearized homology and cohomology. These results predict an inverse Serr{e} functor, which should be the Poincar\'{e}-Lefschetz duality on the category of constructible sheaves with perfect stalks
    $$S_\Lambda^+: \Sh_\Lambda^b(M) \rightarrow \Sh_\Lambda^b(M).$$
    Little is known for either the sheaf categories or Fukaya categories of general Weinstein pairs.
.

\subsection{Main results and corollaries}
    We state our main result which provides a general criterion for the microlocalization functor $m_\Lambda: \Sh_\Lambda(M) \rightarrow \msh_\Lambda(\Lambda)$ to be spherical. Under the equivalence of Ganatra-Pardon-Shende \cite{Ganatra-Pardon-Shende3}, the left adjoint of microlocalization $m_\Lambda^l$ is equivalent to the Orlov cup functor on wrapped Fukaya categories, while we expect that the microlocalization $m_\Lambda$ is the cap functor on Fukaya-Seidel categories (see Remark \ref{rem:comicro-cup} and \ref{rem:micro-cap}).

\subsubsection{Spherical adjunction and Serre functor}
    Let $M$ be a real analytic manifold. Consider a fixed Reeb flow $T_t: S^{*}M \rightarrow S^{*}M$. Recall that a (time-dependent) contact isotopy $\varphi_t: S^{*}M \times \mathbb{R} \rightarrow S^{*}M$ is called a positive isotopy if $\alpha(\partial_t \varphi_t) \geq 0$.
    In the definition, we use the word stop for any compact subanalytic Legendrians (following \cite{Sylvan,Ganatra-Pardon-Shende1}), meaning that Hamiltonian flows are stopped by the Legendrian.

    The geometric notion of a swappable subanalytic Legendrian originates from positive Legendrian loops that avoid the Legendrian at the base point \cite{BS-LegIsotopy}, and is explicitly introduced by Sylvan \cite{SylvanOrlov}. Here our definition is slightly different from \cite{SylvanOrlov}.

\begin{definition}
    A compact subanalytic Legendrian $\Lambda \subseteq S^{*}M$ is called a swappable stop if there exists a compactly supported positive Hamiltonian on $S^{*}M \backslash \Lambda$ such that the flow sends $T_\epsilon(\Lambda)$ to an arbitrary small neighbourhood of $T_{-\epsilon}(\Lambda)$, and the backward flow sends $T_{-\epsilon}(\Lambda)$ to an arbitrary small neighbourhood of $T_\epsilon(\Lambda)$.
\end{definition}

    We also introduce the notion of geometric and algebraic full stops, both called full stops for simplicity. We will see in Proposition \ref{prop:full-geo=>alg} that a geometric full stop is always an algebraic full stop.

\begin{definition}
    Let $M$ be compact. A compact subanalytic Legendrian $\Lambda \subseteq S^{*}M$ is called an algebraic full stop if $\Sh_\Lambda(M)$ is proper. $\Lambda \subseteq S^{*}M$ is called a geometric full stop if for a collection of generalized linking spheres at infinity $\Sigma \subseteq S^{*}M$ of $\Lambda$, there exists a compactly supported positive Hamiltonian on $S^{*}M \backslash \Lambda$ such that the flow sends $\Sigma$ to an arbitrary small neighbourhood of $T_{-\epsilon}(\Lambda)$.
\end{definition}

\begin{example}
    There is a large class of examples of swappable stops and full stops in Section \ref{sec:sphere-crit}. Here are two cimple classes of examples. (1)~For a subanalytic triangulation $\mathcal{S} = \{X_\alpha\}_{\alpha \in I}$, the union of unit conormal bundles $\bigcup_{\alpha \in I}N^{*}_\infty X_\alpha$ is a algebraic full stop (we suspect that it is also a geometric full stop and a swappable stop, but we cannot prove that). 
    (2)~For an exact symplectic Landau-Ginzburg model $\pi: T^*M \rightarrow \mathbb{C}$, the Lagrangian skeleton $\mathfrak{c}_F$ of a regular fiber at infinity $F = \pi^{-1}(\infty)$ is a swappable stop and when $\pi$ is a Lefschetz fibration it is a geometric full stop.
\end{example}

    We are able to state our main result, which provides a general criterion for the microlocalization functor $m_\Lambda$ to be spherical.

\begin{theorem}[Theorem \ref{thm:main-fun}]\label{thm:main}
    Let $\Lambda \subseteq S^{*}M$ be a compact subanalytic Legendrian. Suppose $\Lambda$ is a full stop or a swappable stop. Then the microlocalization functor along $\Lambda$ and its left adjoint
    $$m_\Lambda: \Sh_\Lambda(M) \leftrightharpoons \msh_\Lambda(\Lambda) : m_\Lambda^l$$
    form a spherical adjunction.
\end{theorem}

    Restricting attention to the pair of sheaf categories of compact objects, and the corresponding pair of sheaf categories of proper objects when the manifold is compact, which are the sheaf theoretic models of suitable versions of Fukaya categories, we can show the following corollary.

\begin{corollary}\label{cor:main}
    Let $\Lambda \subseteq S^{*}M$ be a closed subanalytic Legendrian. Suppose $\Lambda$ is either a swappable stop or a geometric full stop. Then the microlocalization functor along $\Lambda$ on the sheaf category of objects with perfect stalks
    $$m_\Lambda: \Sh_\Lambda^b(M) \rightarrow \msh^b_\Lambda(\Lambda)$$
    is a spherical functor. Respectively, the left adjoint of the microlocalization functor on the sheaf category of compact objects
    $$m_\Lambda^l: \msh^c_\Lambda(\Lambda) \rightarrow \Sh^c_\Lambda(M)$$
    is also a spherical functor.
\end{corollary}
\begin{remark}\label{rem:comicro-cup}
    According to \cite{Ganatra-Pardon-Shende3}*{Proposition 7.24} there is a commutative diagram between microlocal sheaf categories and wrapped Fukaya categories
    \[\xymatrix{
    \mathcal{W}(F) \ar[r]^{\sim\hspace{10pt}} \ar[d]_{\cup_F} & \msh^c_{\mathfrak{c}_F}(\mathfrak{c}_F) \ar[d]^{m_{\mathfrak{c}_F}^*}\\
    \mathcal{W}(T^*M, F) \ar[r]^{\sim} & \Sh^c_{\mathfrak{c}_F}(M).
    }\]
    Therefore the second part of our theorem recovers the result by Sylvan \cite{SylvanOrlov} that
    $$\cup_F: \mathcal{W}(F) \rightarrow \mathcal{W}(X, F)$$
    is spherical in the case $X = T^*M$. However, different from \cite{SylvanOrlov}, we are able to explicitly construct the left and right adjoint functors in the proof.
\end{remark}
\begin{remark}\label{rem:micro-cap}
    For a Lefschetz fibration $\pi: T^*M \rightarrow \mathbb{C}$ 
    with regular fiber at infinity $F = \pi^{-1}(\infty)$, $\mathcal{W}(T^*M, F)$ is generated by Lagrangian thimbles \cite{Ganatra-Pardon-Shende2}*{Corollary 1.14} and is a proper category \cite{Ganatra-Pardon-Shende3}*{Proposition 6.7} (when $M = T^n$ and $F$ is the Weinstein thickening of the FLTZ skeleton \cite{FLTZCCC,RSTZSkel}, this can also be proved using mirror symmetry \cite{KuwaCCC}), and hence
    $$\Sh^b_{\mathfrak{c}_F}(M) \simeq \Sh^c_{\mathfrak{c}_F}(M) \simeq \mathcal{W}(T^*M, F).$$
    Since it is expected that $\mathcal{W}(T^*M, F) \simeq \mathcal{FS}(T^*M, \pi)$\footnote{As pointed out in \cite{Ganatra-Pardon-Shende3}*{Footnote~2}, if one takes $\mathcal{W}(T^*M, F)$ as the definition of the Fukaya-Seidel category then this is tautological. However a comparison result between $\mathcal{W}(T^*M, F)$ and the Fukaya-Seidel category defined in \cite{Seidelbook}*{Part~3} is not yet in the literature.}, there should be an equivalence $\Sh^b_{\mathfrak{c}_F}(M) \simeq \mathcal{FS}(T^*M, \pi)$ (when $M = T^n$, Zhou has sketched a proof in his thesis \cite{ZhouThesis}). Therefore, our theorem should be viewed as a sheaf theory version of the result \cite{AbGan} that
    $$\cap_F: \mathcal{FS}(T^*M, \pi) \rightarrow \mathcal{F}(F)$$
    is spherical. However, since we do not know a commutative diagram, that result \cite{AbGan} does not directly follow from ours.
\end{remark}

    We can write down the spherical twists and cotwists as follows. 
    Previous work of the first author \cite{Kuo-wrapped-sheaves} has defined the positive wrapping functor $\mathfrak{W}_\Lambda^+$ (resp.~negative wrapping functor $\mathfrak{W}_\Lambda^-$) sending an arbitrary sheaf in $\Sh(M)$ to $\Sh_\Lambda(M)$ by a colimit (resp. a limit) of positive (resp.~negative) wrappings into $\Lambda$. The spherical cotwist (resp.~the dual cotwist) $\Sh_\Lambda(M) \rightarrow \Sh_\Lambda(M)$ for $m_\Lambda$ is the functor defined by wrapping positively (resp.~negatively) around $S^{*}M$ once along the Reeb flow.

\begin{proposition}
    Let $\Lambda \subseteq S^{*}M$ and $T_t: S^{*}M \rightarrow S^{*}M$ be a Reeb flow. Then the spherical cotwist and dual cotwist are the negative and positive wrap-once functor (where $\epsilon > 0$ is sufficiently small)
    $${S}_\Lambda^- = \mathfrak{W}_\Lambda^- \circ T_{-\epsilon}, \,\,\, {S}_\Lambda^+ = \mathfrak{W}_\Lambda^+ \circ T_{\epsilon}.$$
\end{proposition}
\begin{remark}
    By \cite[Proposition 7.24]{Ganatra-Pardon-Shende3}, we know that the cup functor $\cap_F$ on partially wrapped Fukaya categories is isomorphic to the left adjoint of microlocalization $m_\Lambda^l$ on sheaf categories. Therefore, we have a commutative diagram
    \[\xymatrix{
    \mathcal{W}(T^*M, F) \ar[r]^{\sim} \ar[d]_{\mathcal{S}_F^\pm} & \Sh^c_{\mathfrak{c}_F}(M) \ar[d]^{S_{\mathfrak{c}_F}^\pm}\\
    \mathcal{W}(T^*M, F) \ar[r]^{\sim} & \Sh^c_{\mathfrak{c}_F}(M),
    }\]
    such that our wrap-once functor $S_{\mathfrak{c}_F}^\pm$ is isomorphic to the wrap-once functor of Sylvan \cite{SylvanOrlov}.
\end{remark}

    We will also write down a formula for spherical twists and dual twists in Section \ref{sec:semi-ortho} Corollary \ref{cor:twist}, which can be interpreted as the monodromy functors.

    Finally, the following statement shows that the negative wrap-once functor $S_\Lambda^-$ is the Serre functor on the subcategory of proper objects up to a twist by the dualizing sheaf $\omega_M$.

\begin{proposition}[Proposition \ref{prop:serre}]
    Let $\Lambda \subseteq S^{*}M$ be a full or swappable subanalytic compact Legendrian stop. Then $S_\Lambda^- \otimes \omega_M$ is the Serr{e} functor on $\Sh^b_\Lambda(M)_0$ of sheaves microsupported on $\Lambda$ with perfect stalks and compact supports. In particular, when $M$ is orientable, $S_\Lambda^-[-n]$ is the Serr{e} functor on $\Sh^b_\Lambda(M)_0$.
\end{proposition}
\begin{remark}
    Since our wrap-once functor is isomorphic to the wrap-once functor of Sylvan \cite{SylvanOrlov}, we can prove that the negative wrap-once functor of Sylvan
    $$\mathcal{S}_\Lambda^-: \mathrm{Prop}\,\mathcal{W}(T^*M, F) \rightarrow \mathrm{Prop}\,\mathcal{W}(T^*M, F)$$
    is the Serr{e} functor when $M$ is orientable. In particular, when $F = \pi^{-1}(\infty)$ is the fiber of a symplectic Lefschetz fibration, then $\mathcal{S}_\Lambda^-$ is the Serr\'{e} functor on $\mathcal{W}(T^*M, F)$.
\end{remark}
\begin{remark}
    Spherical adjunctions together with a compatible Serr{e} functor, in the smooth and proper setting, implies existence of a weak relative right Calabi-Yau structure \cite{KPSsphericalCY}, but we do not expect the relative Calabi-Yau structures for general Weinstein pairs to be proved this way. See the discussion in Remark \ref{rem:relativeCY}.
\end{remark}

    Note that the Serre functor result has appeared in earlier literature in the setting of constructible sheaves on certain smooth algebraic varieties. For instance, Beilinson-Bezruvkavnikov-Mirkovic \cite{TiltingExercise} found the Serre functor on constructible sheaves on flag varieties with Schubert stratifications (which is a proper category). We provide another description of the Serre functor in that setting.

\subsubsection{Spherical pairs and variation of skeleta}
    We know that the data of a spherical adjunction
    $$F: \sA \leftrightharpoons \sB: F^l$$
    is equivalent to a perverse sheaf of categories on a disk with one singularity \cite{KapraSchober}, where the stalk at $0 \in \mathbb{D}^2$ is the nearby category $\sA$ while the stalks at $(0, 1]$ are the vanishing category $\sB$. By considering a double cut of the real interval $[-1, 1] \subset \mathbb{D}^2$, we can also define a spherical pair \cite{KapraSchober}
    $$\sB_- \xrightleftharpoons{F_-} \sC \xleftrightharpoons{F_+} \sB_+$$
    where the two vanishing categories $\sB_\pm$ are related by a pair of equivalences, and the nearby category $\sC$ admits semi-orthogonal decompositions by $\sA$ and $\sB_\pm$.

    Given our formalism of spherical adjunctions, we will prove as a corollary spherical pairs which give rise to non-trivial equivalences of microlocal sheaf categories over different Lagrangian skeleta that do not a priori require non-characteristic deformations (while in known examples of such equivalences, \cite{DonovanKuwa,ZhouVarI} the Lagrangian skelata are related by non-characteristic deformations).

\begin{definition}
    Let $\Lambda_\pm \subseteq S^{*}M$ be two disjoint closed subanalytic Legendrian stops. Suppose there exist both a positive and a negative compactly supported Hamiltonian flow that sends $\Lambda_+$ to an arbitrary small neighbourhood of $\Lambda_-$, whose backward flows send $\Lambda_-$ to an arbitrary small neighbourhood of $\Lambda_+$. Then $(\Lambda_-, \Lambda_+)$ is called a swappable pair.
\end{definition}
\begin{remark}
    When $\Lambda_\pm \subseteq S^{*}M$ are Lagrangian skeleta of Weinstein hypersurfaces $F_\pm \subseteq S^{*}M$, we do not know whether $(X, F_\pm)$ are in fact Weinstein homotopic, though in some examples we will mention we suspect that they are. Moreover, it is in general a hard question when a singular Lagrangian will arise as the skeleton of a Weinstein manifold \cite{WeinRevisit}*{Problem 1.1~\&~Remark 1.2}.
\end{remark}

    We will show that a swappable pair of Legendrian stops produces a spherical pair.

\begin{theorem}[Theorem \ref{thm:sphere-pair-var}]
    Let $\Lambda_\pm \subseteq S^{*}M$ be a swappable pair of closed Legendrian stops. Then $\Sh_{\Lambda_-}(M) \simeq \Sh_{\Lambda_+}(M), \msh_{\Lambda_-}(\Lambda_-) \simeq \msh_{\Lambda_+}(\Lambda_+)$, and there is a spherical pair
    $$\Sh_{\Lambda_-}(M) \rightleftharpoons \Sh_{\Lambda_+ \cup \Lambda_-}(M) \leftrightharpoons \Sh_{\Lambda_+}(M).$$
\end{theorem}
\begin{remark}
    As in the previous result, we can show that the spherical pair can be restricted to the subcategories of compact objects of sheaf categories, which therefore leads to a result on partially wrapped Fukaya categories.
\end{remark}
\begin{remark}
    For symplectic topologists, this result may seem boring since one may suspect that the corresponding Weinstein pairs turn out to be Weinstein homotopic. However, when considering Fukaya-Seidel categories given a Landau-Ginzburg potential, the Weinstein hypersurfaces $F_\pm$ can be fibers of different potential functions. Therefore, studying the behaviour of their Lagrangian skeleta provides a way to compare the categories directly.
\end{remark}

\subsection{Sheaf theoretic wrapping, small and large}
We mention one technical point behind the various arguments of this paper.
The Guillermou-Kashiwara-Schapira sheaf quantization \cite{Guillermou-Kashiwara-Schapira} allows us to define the notion of isotopies of sheaves, or simply wrappings, which we will recall in Section \ref{sec:isotopies_of_sheaves}.
For a sheaf $F \in \Sh(M)$ and a positive contact isotopy $T_t$ on $S^* M$, there is a family of sheaves $F_t$ and morphism $F_s \rightarrow F_t$, $s \leq t$, 
such that $F_0 = F$ and $\msif(F_t) = T_t(\msif(F))$.
Such families of sheaves allow us to do two things.

First, one can displace microsupport at infinite with small wrappings, which is a sheaf-theoretical version of demanding generic intersection of Lagrangians in the Floer setting.    
For two sheaves $F, G \in \Sh(M)$, although the object $\Hom(F,G)$ is always defined, it is usually easier to understand it when $\msif(F) \cap \msif(G) = \varnothing$.
Furthermore, when $\msif(F)$ and $\msif(G)$ are both Legendrians, and assume $T_t (\msif(G))$ intersects $\msif(F)$ only at $t = 0$, 
the jump of $\Hom(F,G_t)$ from $t < 0$ to $t > 0$ can be measured, in a way tautologically by $\msh_\Lambda(\Lambda)$ where $\Lambda \coloneqq \msif(F) \cap \msif(G)$,
through microlocalization.

Secondly, since the isotopy $T_t$ pushes away microsupport, one can in fact cut off the microsupport by consider large wrappings.
That is, the left (resp.~right) adjoint of the inclusion $\Sh_X(M) \subseteq \Sh(M)$, whose existence is gauranteed by the existence of microlocal cut-off lemma locally \cite[Section 5.2]{KS}, 
can be described by the global geometry through the functor $\wrap_\Lambda^+$ (resp.~$\wrap_\Lambda^-$) defined by taking the colimit (resp.~limit) over larger and larger positive (resp. negative) wrappings. (See Theorem \ref{w=ad} for details.)

Combining the two facts in Section \ref{sec:doubling}, we can give a geometric description to the left (resp.~right) adjoint of the microlocalization functor 
$m_\Lambda: \Sh_\Lambda(M) \rightarrow \msh_\Lambda(\Lambda)$ by decomposing it to a inclusion called doubling functor $w_\Lambda$ followed by a quotient by the positive (resp.~negative) wrapping functor $\wrap_\Lambda^+$ (resp.~$\wrap_\Lambda^-$)
$$ \msh_\Lambda(\Lambda) \hookrightarrow \Sh_{T_\epsilon(\Lambda) \cup T_{-\epsilon}(\Lambda)}(M) \twoheadrightarrow \Sh_\Lambda(M).$$
The doubling functor will induce a fiber sequence of functors $T_{-\epsilon} \rightarrow T_\epsilon \rightarrow w_\Lambda \circ m_\Lambda$, from where the spherical dual cotwist by wrapping around positively once $S_\Lambda^+$ (resp.~the spherical cotwist by wrapping around negatively once $S_\Lambda^-$) will naturally arise as we apply the positive (resp.~negative) wrapping functors. This turns out to be the key ingredient in our proof.

\begin{remark}
    Though the doubling functor $w_\Lambda$ in our current setting has appeared already in works of Nadler--Shende \cite[Section 7]{NadShen}, we do not know how to prove further adjunction properties. Our main result crucially relies on the fact that the doubling functor will induce both the left adjoint and the right adjoint of the microlocalization functor, which does not seem obvious at the first place. For this reason, we will provide two different approaches generalizing the work of Guillermou \cite[Section 13--15]{Gui}, and in particular providing an explicit local formula of the doubling functor as opposed to \cite{NadShen}.
\end{remark}

\subsection*{Acknowledgement}
We would like to thank Mohammed Abouzaid, Pramod Achar, Shaoyun Bai, Roger Casals, Laurent C\^ot\'e, Yuichi Ike, Emmy Murphy, Nick Rozenblyum, Germ\'an Stefanich, Vivek Shende, Pyongwon Suh, Alex Takeda, Dima Tamarkin and Eric Zaslow  
for helpful discussions. CK was partially supported by NSF CAREER DMS-1654545 and VILLUM FONDEN grant 37814.

\section{Preliminary} 
We give a quick review of the microlocal sheaf theory developed by Kashiwara and Schapira in \cite{KS} within the modern categorical setting \cite{Volpe-six-operations, JinTreu} as well as results from more recent work such as \cite{Nadler-pants, Ganatra-Pardon-Shende3}.
The purpose of this section is to fix notations and collect previous results which we will use in the main body of the text. 

\subsection{Microlocal sheaf theory}
Through out this paper, we fix once and for all a rigid stable symmetric monoidal category $\cV_0$ in the sense of Hoyois, Scherotzke, and Sibilla \cite{Hoyois-Scherotzke-Sibilla}, and we announce that our sheaves will take coefficient in $\cV \coloneqq \Ind(\cV_0)$, the Ind-completion of $\cV_0$. As discussed in \cite{Hoyois-Scherotzke-Sibilla}, stable categories over $\cV$ enjoy many formal properties enjoyed by ordinary stable categories studied in \cite{Lurie-HA} and one can develop sheaf theory in this setting. Special cases of $\cV$ include dg categories over a field $\Bbbk$, those over the integer $\ZZ$, and the category of spectra $\Sp$.

Now a sheaf $F$ on a topological space $X$ is a presheaf $F: \Op_X^{op} \rightarrow \cV$ satisfying the standard local-to-global condition:
For an open set $U \subseteq X$ and an open covering $\cU \coloneqq \{U_i\}$ of $U$, the canonical map
$$ F(U) \rightarrow \lmi{U_I \in C_\cU} F(U_I)$$
is an isomorphism where $C_\cU$ is the \v{C}ech nerve formed by finite intersections of open sets in $\cU$.
We use $\Sh(X)$ to denote the category of sheaves on $X$.
When we restrict to locally compact Hausdorff topological spaces, the functor $\Sh$ admits the celebrated Grothendieck six-functor formalism. 
That is, there exists a tensor product $\otimes$ on $\Sh(M)$ and the corresponding internal Hom $\sHom$, and
for a continuous map $f: X \rightarrow Y$, there exists two adjunction pairs $(f^*,f_*)$  and $(f_!,f^!)$ and various compatibility relations,
base change in particular, hold.

For a sheaf $F$ on $X$, one can associate a closed set $\supp(X) = \overline{ \{x | F_x \neq 0 \} }$,
the support of $F$, indicating the place where $F$ is non-trivial. 
When we further restrict to the case of finite dimensional manifolds, one has access to a more powerful invariant
$\ms(F) \subseteq T^*M$, the microsupport which is defined in \cite[Proposition 5.1.1]{KS}.
This is a conic closed subset in the cotangent bundle $T^* M$, roughtly indicating the codirections where the sheaf changes,
and a deep theorem \cite[Theorem 6.5.4]{KS} asserts that it is always coisotropic.

    Let $0_M$ be the zero section in $T^*M$. We note that since the singular support is always a conical subset, we can consider the singular support at infinity $\ms^\infty(F) = (\ms(F) \backslash 0_M) / \RR_{>0}$, which is always a coisotropic subset in the contact manifold $S^*M = (T^* M \backslash 0_M)/\RR_{>0}$.

Let $T_Z^*M$ the conormal bundle of a closed submanifold $Z \subseteq M$, and $N^*_{in}(U)$ (resp.~$N^*_{out}(U)$) the union of the inward (resp.~outward) conormal bundle and the zero section $0_U$ of a locally closed submanifold $U \subseteq M$ with piecewise $C^1$-boundary (the definition can be generalized to any locally closed submanifolds; see \cite[Section 5.3]{KS}). For $f: M \rightarrow N$, let $df^*$ and $f_\pi$ be the bundle maps
$$T^*M \xleftarrow{df^*} M \times_N T^*N \xrightarrow{f_\pi} T^*N.$$
Finally, for $X \subset T^*M$, let $-X \subseteq T^*M$ be its image under the antipodal map.

We list here a few ways to estimate microsupports of sheaves and some deformation result which we will use in this paper:

\begin{definition}[{\cite[Definition 5.4.12]{KS}}]
Let $A$ be a closed conic subset of $T^* N$.
We say $f$ is \textit{noncharacteristic} for $A$ if
$$ f_\pi^{-1}(A) \cap T_M^* N \subseteq M \times_N 0_N.$$
For a sheaf $F \in \Sh(M)$, we say $f$ is \textit{noncharacteristic} for $F$ if it is the case for $\ms(F)$.
\end{definition}

\begin{proposition}\label{prop:mses}
We have the following results:
\begin{enumerate}
\item (\cite[Proposition 5.1.3]{KS}) If $F \rightarrow G \rightarrow H$ is a fiber sequence in $\Sh(M)$, then
$$ \left( \ms(F) \setminus \ms(H) \right) \cup \left( \ms(H) \setminus \ms(F) \right)
\subseteq \ms(G) \subseteq \ms(F) \cup \ms(H).$$

\item (\cite[Proposition 5.4.1]{KS}) 
For $F \in \Sh(M)$, $G \in \Sh(N)$, $\ms(F \boxtimes G) \subseteq \ms(F) \times \ms(G)$.

\item (\cite[Proposition 5.4.2]{KS})
For $F \in \Sh(M)$, $G \in \Sh(N)$, $$\ms\left(\sHom(\pi_1^*F, \pi_2^*G) \right) \subseteq -\ms(F) \times \ms(G)$$

\item (\cite[Proposition 5.4.4]{KS}) 
For $f: M \rightarrow N$ and $F \in \Sh(M)$, if $f$ is proper on $\supp(F)$, then $$\ms(f_* F) 
\subseteq f_\pi \left( (d f^*)^{-1} \ms(F) \right).$$

\item (\cite[Proposition 5.4.5 and Proposition 5.4.13]{KS}) 
For $f: M \rightarrow N$ and $F \in \Sh(N)$,  if $f$ is noncharacteristic for $F$, then
$$\ms(f^* F) \subseteq df^* ( f_\pi^{-1}(\ms(F)) )$$
and the natural map $f^* F \otimes f^! 1_Y \rightarrow f^! F$ is an isomorphism.
If $f$ is furthermore a submersion\footnote{Kashiwara-Schapira use the word smooth morphism for a submersion between manifolds, which comes from the notion in algebraic geometry.}, the estimation is an equality.

\item  (\cite[Proposition 5.4.8]{KS}) 
Let $Z \subseteq M$ be closed. 
If $\ms(F) \cap N^*_{out} (Z) \subseteq 0_M$, then 
$$\ms(F_Z) \subseteq N^*_{in}(Z) + \ms(F).$$
Similarly, let $U \subseteq M$ be open. 
If $\ms(F) \cap N^*_{in} (U) \subseteq 0_M$, then 
$$\ms(F_U) \subseteq N^*_{out}(U) + \ms(F).$$

\item (\cite[Proposition 5.4.14]{KS}) 
For $F\,$and $G \in \Sh(M)$, if $\ms(F) \cap -\ms(G) \subseteq 0_M$, then
$$ \ms(F \otimes G) \subseteq \ms(F) + \ms(G).$$

\item  (\cite[Proposition 5.4.14 and Exercise V.13]{KS}) 
For $F\,$and $G \in \Sh(M)$, if $\ms(F) \cap \ms(G) \subseteq 0_M$, then
$$ \ms( \sHom(F,G) ) \subseteq \ms(G) - \ms(F).$$
Write $\ND{M}(F) := \sHom(F,1_M)$. If moreover $F$ is cohomological constructible, then
the natural map $$\ND{M}(F) \otimes G \rightarrow \sHom(F,G)$$ is an isomorphism.
If furthermore $F = \omega_M$, i.e., when $\sHom(G,F) \eqqcolon D_M (G)$ is the Verdier dual,
then $\ms(\VD{M} (G) ) = - \ms(G)$.


\end{enumerate}

\end{proposition}

\begin{remark}\label{rem:DFotimesG}
    When ${F}$ is cohomologically constructible, by Proposition \ref{prop:mses}~(8) we know that $$\Delta^*\mathscr{H}om(\pi_1^*{F}, \pi_2^*{G}) \simeq \Delta^*(\ND{M \times M}(\pi_1^*{F}) \otimes \pi_2^*{G}) \simeq \ND{M}({F}) \otimes {G},$$
    where we define $\ND{M}(F) := \sHom(F, 1_M)$. This special case will be of use in the following sections.
\end{remark}

    Here are some more delicate microsupport estimates. Let $A, B \subset T^*M$. Define a subset $A \,\widehat +\, B \subseteq T^*M$ such that $(x, \xi) \in A \,\widehat +\, B$ if there exists $(a_n, \alpha_n) \in A, (b_n, \beta_n) \in B$ such that
    $$a_n, b_n \rightarrow x, \, \alpha_n + \beta_n \rightarrow \xi, \, |a_n - b_n||\alpha_n| \rightarrow 0.$$
    Let $i: M \hookrightarrow N$ be a closed embedding. Then for $A \subset T^*N$, define $i^{\#}(A) \subseteq T^*M$ such that $(x, \xi) \in i^{\#}(A)$ if there exists $(y_n, \eta_n, x_n, \xi_n) \in A \times T^*M$ such that
    $$x_n, y_n \rightarrow x, \, \eta_n - \xi_n \rightarrow \xi, \, |x_n - y_n||\eta_n| \rightarrow 0.$$

\begin{proposition}\label{prop:noncharacteristic-ms-es}
We have the following results:

\begin{enumerate}
\item (\cite{KS}*{Theorem 6.3.1}) Let $j: U \hookrightarrow M$ be an open embedding, ${F} \in \Sh(U)$. Then
    $$\ms(j_*{F}) \subseteq \ms({F}) \,\widehat +\, N_\text{in}^*(U), \;\; \ms(j_!{F}) \subseteq \ms({F}) \,\widehat +\, N_{out}^*(U).$$
\item (\cite{KS}*{Corollary 6.4.4})
    Let $i: N \hookrightarrow M$ be a closed embedding, ${F} \in \Sh(M)$. Then
    $$\ms(i^*{F}) \subseteq i^{\#}\ms({F}).$$
\end{enumerate}

\end{proposition}

By combining the six-functors, one can produce functors between sheaves on two topological spaces from sheaves on their product. Let $X_i$, $i = 1, 2, 3$, be locally compact Hausdorff topological spaces, and write $X_{ij} = X_i \times X_j$, for $i < j$,  $X_{123} = X_1 \times X_2 \rightarrow X_3 $, and $\pi_{ij}: X_{123} \rightarrow X_{ij}$ for the corresponding projections. For $F \in \Sh(X_{12})$, $G \in \Sh(X_{23})$, the \textit{convolution} is defined to be 
$$G \circ_{X_2} F \coloneqq {\pi_{13}}_! (\pi_{23}^* G \otimes \pi_{12}^* F ) \in \Sh(X_{13}).$$
When there is no confusion what $X_2$ is, we will usually surpass the notation and simply write it as $G \circ F$. This is usually the case when $X_1 = \{*\}$, $X_2 = X$, and $X_3 = Y$ and we think of $X$ as the source and $Y$ as the target, $G \in \Sh(X \times Y)$ as a functor sending $F \in \Sh(X)$ to $G \circ F \in \Sh(Y)$. Note that from its expression, this functor is colimit-preserving. 
 
\begin{lemma}[{\cite[Proposition 3.6.2]{KS}}]\label{convrad}
For a fixed $G \in \Sh(X_{23})$, the functor $G \circ (-) : \Sh(X_{12}) \rightarrow \Sh(X_{13})$ induced by convoluting with $G$ has a right adjoint, which we denote by $\sHom^\circ(G,-): \Sh(X_{13}) \rightarrow \Sh(X_{12})$, that is given by
\begin{equation}\label{cvr}
H \mapsto  {\pi_{12}}_* \sHom(\pi_{23}^* G  ,\pi_{13}^! H).
\end{equation}
\end{lemma}

\begin{example}
We note that convolution recovers $*$-pullback and $!$-pushforward. For example, let $f: X \rightarrow Y$ be a continuous map and denote by $i: \Gamma_f \subseteq X \times Y$ its graph. Take $X_1 = \{*\}$, $X_2 = X$, and $X_3 = Y$, then for $F \in \Sh(X)$,
$$1_{\Gamma_f} \circ F = {\pi_Y}_! ( 1_{\Gamma_f} \otimes \pi_X^* F) = {\pi_Y}_! i_! i^* \pi_X^* F = f_! F.$$
\end{example}

We note that the base change formula implies that convolution satisfies associativity.
\begin{proposition}\label{conv:asso}
Let $F_i \in \Sh(X_{i, i+1})$ for $i = 1, 2, 3$. Then
$$ F_3 \circ_{X_3} (F_2 \circ_{X_2} F_1) = (F_3 \circ_{X_3} F_2) \circ_{X_2} F_1.$$
In particular, if $G_1$, $G_2 \in \Sh(X \times X)$, then there is an identification of functors
$$ G_2 \circ (G_1 \circ (-) ) = (G_2 \circ_X G_1) \circ (-).$$
\end{proposition}

We will use a relative version of convolution. Let $B$ be a locally compact Hausdorff space viewed as a parameter space.  Regard $F \in \Sh(X_{12} \times B)$, $G \in \Sh(X_{23} \times B)$ as $B$-family sheaves, one can similarly define the relative convolution $G \circ|_B F \in \Sh(X_{13} \times B)$ by replacing $\pi_{ij}$ with 
$$\pi_{{ij},B}: X_{123} \times B \rightarrow X_{ij} \times B.$$

In the case of manifolds, convolution satisfies certain compatibility with microsupport. For $A \subseteq T^* M_{12}$ and $B \subseteq T^* M_{23}$, we set
$$B \circ A = \{ (x,\xi,z,\zeta) \in T^* M_{13} \mid \exists (y,\eta), (x,\xi,y,\eta) \in A, (y,-\eta, z, \zeta) \in B \}.$$
Note if $A$ and $B$ are Lagrangian correspondences satisfying appropriate transversality condition, the set $B \circ A$ is the composite Lagrangian correspondence twisted by a minus sign on the second component. Write $q_{ij}: T^* M_{123} \rightarrow T^* M_{ij}$ to be the projection on the level of cotangent bundles and $q_{2^a3}$ the composition of $q_{23}$ with the antipodal map on $T^* M_2$. Then $B \circ A = q_{13} (q_{2^a3}^{-1}  B \cap q_{12}^{-1} A)$ and (4), (6) and (3) of Proposition \ref{prop:mses} imply the following corollary.
 
\begin{corollary}[{\cite[(1.12)]{Guillermou-Kashiwara-Schapira}}]\label{msconv}

Assume the following two conditions
\begin{enumerate}
\item $\pi_{13}$ is proper on $M_1 \times \supp(G) \cap \supp(F) \times M_3$;
\item $q_{2^a 3}^{-1} \ms(G)  \cap q_{12}^{-1} \ms(F)  \cap 0_{M_1} \times T^* M_2 \times 0_{M_3}
\subseteq 0_{M_{123}}.$ 
\end{enumerate} Then
$$\ms( G \circ F ) \subseteq \ms(G) \circ \ms(F).$$
\end{corollary}

A similar microsupport estimation holds for the $B$-family case. One noticeable difference for the microsupport estimation is that instead of $T^* M_{ij}$ and $T^* M_{123}$ one has to consider $T^* M_{ij} \times T^* B$ and $T^* M_{123} \times (T^* B \times_B T^* B)$ instead. Here $\times_B$ is taken over the diagonal $B \hookrightarrow B \times B$. Also the projection $r_{ij}: T^* M_{123} \times (T^* B \times_B T^* B) \rightarrow T^* M_{ij} \times T^* B$ for the $B$-component is now given by the first projection (with a minus sign) for $ij =12$, the addition for $ij = 13$, and the second projection $ij = 23$. Other than that the microsupport estimation is similar to the ordinary case.

The following results show that microsupport estimates detect propagation of sections in a continuous family of deformation of open subsets.

\begin{lemma}[non-characteristic deformation lemma, {\cite[Proposition 2.7.2]{KS}}]
    Let ${F} \in \Sh(M)$ and $\{U_t\}_{t \in \mathbb{R}}$ be a family of open subsets and $Z_t = \bigcap_{t > s}\overline{U_t \backslash U_s}$. Suppose that
    \begin{enumerate}
      \item $U_t = \bigcup_{s < t} U_s$, for $-\infty < t < +\infty$;
      \item $\overline{U_t \backslash U_s} \cap \mathrm{supp}({F})$ is compact, for $-\infty < s < t < +\infty$;
      \item $\Gamma_{M \backslash U_t}({F})_x = 0$, for $x \in Z_s \backslash U_t, \, -\infty < s \leq t < +\infty$.
    \end{enumerate}
    Then for any $t \in \mathbb{R}$ we have
    $$\Gamma\bigg(\bigcup_{s \in \mathbb{R}} U_s, {F}\bigg) \xrightarrow{\sim} \Gamma(U_t, {F}).$$
\end{lemma}

\begin{lemma}[microlocal Morse lemma, {\cite[Corollary 5.4.19]{KS}}]\label{prop:morselemma}
Let $F \in \Sh(M)$ and let $\phi: M \rightarrow \RR^1$ be a $C^1$-function such that $\phi: \supp(F) \rightarrow \RR$ is proper.
Let $a < b \in \RR$.
\begin{enumerate}
\item Assume $d \phi(x) \not \in \ms(F)$ for all $x \in M$ such that $a \leq \phi(x) < b$. Then the natural maps:
\begin{align*}
\Gamma(\phi^{-1}(-\infty,b); F) 
&\rightarrow \Gamma(\phi^{-1}(-\infty,a];F) \\
&\rightarrow \Gamma(\phi^{-1}(-\infty,a);F)
\end{align*}
are isomorphisms.
\item Assume $-d \phi(x) \not \in \ms(F)$ for all $x \in M$ such that $a < \phi(x) \leq b$ (resp. $a \leq \phi(x) < b$). 
Then the natural map:
$$ \Gamma_{\phi^{-1}(-\infty,a]}(X; F)  \rightarrow \Gamma_{\phi^{-1}(-\infty,b]}(X; F)$$
(resp. $\Gamma_c(\phi^{-1}(-\infty,a);F) \rightarrow \Gamma_c(\phi^{-1}(-\infty,b);F)$)
is an isomorphism.
\end{enumerate}
\end{lemma}

\subsection{Microsheaves}\label{sec:micro}

Fix a smooth manifold $M$. We recall the construction of microsheaves, a sheaf of categories $\msh$ on the cotangent bundle.
The concept of microsheaves  (also known as the Kashiwara-Schapira stack/sheaf) goes back to \cite{KS}*{Section 6} where they studied the $\mhom$ functor. 
More recently Guillermou introduced the Kashiwara-Schapira stack in the setting of bounded derived categories \cite{Gui}. 
Here, working with (compactly generated) stable categories, we follow the definition in \cite{NadShen}.

We first define a presheaf, i.e., a functor 
\begin{align*}
\msh^\pre: (\Op_{T^* M}^{\Rp})^{op} &\longrightarrow \st \\
\Omega &\longmapsto \Sh(M)/ \Sh_{\Omega^c}(M)
\end{align*}
where we restrict our attention to conic open sets $\Op_{T^* M}^{\Rp}$,
and the target $\st$ is the (large) category of stable categories with morphisms being exact functors.
We denote by $\msh$ its sheafification and refer it as the sheaf of microsheaves.
Note that since there is a canonical identification 
\begin{align*}
\msh^\pre(T^* U) &\xrightarrow{\sim} \Sh(U) \\
F &\mapsto F |_U
\end{align*}
and sheaves on the base $U \mapsto \Sh(U)$ forms a sheaf in $\st$,
compatibility of sheafification and pullback implies that $\msh |_{0_M} = \Sh$.

The first non-trivial statement is that the Hom's of $\msh$ can be computed by $\mhom$.
\begin{definition}[{\cite[Definition 4.1.1]{KS}}]
Let $F$, $G \in \Sh(M)$, we set
$$\mhom(F,G) \coloneqq \mu_{\Delta_M}   \sHom(p_2^* F, p_1^! G)$$
where $\mu_{\Delta_M}$ is the microlocalization along the diagonal \cite[Section 4.3]{KS}.
\end{definition}

\begin{definition_theorem}[{ \cite[Theorem 6.1.2]{KS}, \cite[Corollary 5.5.]{Gui} }] \label{mu-hom_as_hom} 
Let $F\,$and $G \in \Sh(M)$ represent objects with the same name in $\msh(\Omega)$.
Then there is an canonical isomorphism
$$ \mhom(F,G) \rightarrow \sHom_{\msh(\Omega)}(F,G)$$
between sheaves on $\Omega$. Thus, we abuse the notation and simply use $\mhom$ to denote the Hom of $\msh$ valued in sheaves on conic open sets of $T^* M$.
\end{definition_theorem}

We note also that the microsupport triangular inequality from (1) of Proposition \ref{prop:mses} implies that, 
for an object $F \in \msh^\pre(\Omega)$ represented by a sheaf with the same name,
the closed subset in $\Omega$, $\ms_\Omega(F) \coloneqq \ms(F) \cap \Omega$ is independent of representative in $\Sh(M)$.
This induces the notion of microsupport for objects in $\msh(\Omega)$.
That is, we define for $F \in \msh(\Omega)$, a point $(x,\xi) \in T^* M$ is in $\ms(F)$ if $F \neq 0 \in \msh_{(x,\xi)}$.
Note that this is notion of microsupport is well-defined and coincide with the originally one since
\begin{equation} \label{microsheaves_at_stalks}
 \msh_{(x,\xi)} = \msh^\pre_{(x,\xi)} = \clmi{\Omega \ni (x, \xi) } \msh^\pre(\Omega) = \Sh(M) / \Sh_{T^* M \setminus \Rp (x,\xi)}(M).
\end{equation}

\begin{remark}
We remark our choice of the coefficient category $\st$ for people who is familiar with the higher categorical derived setting.
First note that $\msh^\pre$ can be regarded also as presheaf in $\PrLst$ or $\PrRst$,
the (large) categories of presentable categories with morphisms being left or right adjoints.
These two are the more natural setting for the purpose of obtaining adjunctions as we shall see later when we restrict
to sheaves with a fixed microsupport condition.
However, stalks behaves badly in this setting and one can, in fact, compute that the inclusion $\Loc \rightarrow \Sh$
of local systems into sheaves is isomorphic on stalks.
On the contrary, the virtue of $\st$ is that isomorphisms can be checked on stalks \cite{Rozenblyum-filtered}.
In fact, the Definition-Theorem \ref{mu-hom_as_hom} in \cite[Theorem 6.1.2]{KS} is a statement on the stalk. 
The form in which it's written here is a simple corollary of the cited Theorem and this fact about $\st$.
\end{remark}

We note that since $\msh$ is conic, $\msh |_{\dT^* M}$ descents naturally to a sheaf on $S^* M$, 
and we abuse the notation, denoting it by $\msh$ as well.
Similar statements for $\mhom$ and $\ms^\infty$ hold in this case.

\begin{definition}
Fix a subanalytic isotropic subset $\Lambda \subseteq S^*M$. 
Let $\msh_\Lambda$ denote the subsheaf of $\msh$ which consists of objects microsupported in $\Lambda$ or $\widehat{\Lambda} = \Lambda \times \RR_{>0}$.
\end{definition}

We note that because of Equation (\ref{microsheaves_at_stalks}) this sheaf coincides with the sheafification of the following 
subpresheaf $\msh^\pre_\Lambda$ of $\msh^\pre$ (where $\ms_\Omega(F) := \ms(F) \cap \Omega$):

\begin{align*}
\msh^\pre_\Lambda: (\Op_{T^* M}^{\Rp})^{op} &\longrightarrow \st \\
\Omega &\longmapsto \{F \in \msh^\pre(\Omega) | \ms_\Omega(F) \subseteq \Lambda \}
\end{align*}

We note that $\msh^\pre_\Lambda$ in fact takes value in $\PrLcs$, the category of compactly generated stable categories
whose morphisms are given by functor which admits both the left and the right adjoints,
and its sheafification in $\st$ and $\PrLcs$ coincide.
In other word, restriction maps $\msh_\Lambda(\Omega) \rightarrow \msh_\Lambda(\Omega^\prime)$ admit both left and right adjoints. 
In particular, we will refer to the restriction map associated to $\dT^* M \subseteq T^* M$ as
the microlocalization functor along $\Lambda$
$$m_\Lambda: \Sh_\Lambda(M) \rightarrow \msh_\Lambda(\Lambda)$$
and denote its left and right adjoint by $m_\Lambda^l$ and $m_\Lambda^r$.
Note that, by the definition, $\msh_\Lambda$ is a constructible sheaf supported on $\Lambda$ or $\widehat{\Lambda}$,
and we will use the same notation $\msh_\Lambda$ to denote the corresponding sheaf on $\Lambda$ or $\widehat{\Lambda}$.

    One can estimate the singular support of the sheaf $\mhom({F, G})$ in $T^*M$. Recall that for $A, B \subset X$, we define the normal cone $C(A, B)$ such that $(x, \xi) \in TX$ iff there exists $a_n \in A, b_n \in B, c_n \in \mathbb{R}$ such that
    $$a_n, b_n \rightarrow x, \;\; c_n(a_n - b_n) \rightarrow \xi, \;\; n \rightarrow \infty.$$

\begin{proposition}[\cite{KS}*{Corollary 5.4.10 \& Corollary 6.4.3}]\label{prop:ss-muhom}
    Let ${F, G} \in \Sh(M)$.
    Then
    $$\ms(\mhom({F, G})) \subseteq C(\ms({F}), \ms({G})).$$
    In particular, $\mathrm{supp}(\mhom({F, G})) \subseteq \ms({F}) \cap \ms({G})$.
\end{proposition}
\begin{remark}\label{rem:ss-muhom}
    By Proposition \ref{prop:mses}~(4) \cite[Proposition 5.4.4]{KS}, we can show that \cite[Corollary 6.4.4 \& 6.4.5]{KS} for $\pi: T^*M \rightarrow M$ and $\dot{\pi}: \dot T^*M \rightarrow M$ we have
    \begin{align*}
        \ms(\pi_* \mhom({F, G})) &\subset \pi_\pi(d\pi^*)^{-1}C(\ms(F), \ms(G)) = -\ms(F) \,\widehat+\, \ms(G), \\
        \ms(\dot\pi_* \mhom({F, G})) &\subset \dot\pi_\pi(d\dot\pi^*)^{-1}C(\ms(F), \ms(G)) = -\ms(F) \,\widehat+_\infty\, \ms(G).
    \end{align*}
\end{remark}

In general, sheafifying a category-coefficient sheaf is complicated.
However, we notice that our on assumption on $\Lambda$ implies that $\msh_\Lambda^\pre$ stabilizes after restricting to small open. 
In fact $\msh_\Lambda$ is constructible and on small open sets it admits a simple description as mentioned in \cite[3.4]{Nadler-pants}: 
For any $(x,\xi) \in \Lambda$, we may choose a small open ball $\Omega \subseteq S^* M$ containing $\Lambda$
such that $\msh_\Lambda(\Omega)$ fits in a fiber sequence,
$$ K(B,\Omega) \hookrightarrow \Sh_\Lambda(B,\Omega) \rightarrow \msh_\Lambda(\Omega),$$
where $B = \pi_\infty (\Omega)$, $\Sh_\Lambda(B,\Omega)$ consists of sheaves $F$ on $B$ such that 
$\msif(F) \cap \Omega \subseteq \Lambda$, and $K(B,\Omega)$ is its subcategory so that 
$\msif(F) \cap \Omega = \varnothing$.


Consequently, one can characterize the stalks $\msh_\Lambda$ as follows. This is a consequence of (quantized) contact transformation \cite[Corollary 7.2.2]{KS}, which asserts that, on small neighborhoods on $S^* M$, $\msh^\pre$ looks the same everywhere.

\begin{theorem}[\cite{Gui}*{Proposition 6.6 \& Lemma 6.7}, \cite{JinTreu}*{Section 3.8 \& 3.9}, \cite{NadShen}*{Corollary 5.4}]
    Let $p = (x, \xi)$ be a smooth point in the subanalytic isotropic subset $\Lambda \subseteq S^{*}M$. Then the stalk of microsheaves is $\msh_{\Lambda,p} \simeq \cV$.
\end{theorem}


    We know that the microlocalization induces morphisms
    $$\mhom(F, G) \rightarrow \mhom(F, G)|_{S^{*}M}, \;\; \sHom(F, G) \rightarrow \dot\pi_*(\mhom({F, G})|_{S^{*}M}).$$
    By \cite[Equation (4.3.1)]{KS}, we immediately know that the second morphism fits into the following Sato's fiber sequence. This will be an important ingredient for the Sato-Sabloff fiber sequence in Section \ref{sec:sato-sab}. 

\begin{proposition}[Sato's fiber sequence \cite{Gui}*{Equation (2.17)}, \cite{Guisurvey}*{Equation (1.3.5)}]\label{thm:sato}
    Let ${F}\,$and ${G} \in \Sh(M)$.
    Then there is a fiber sequence
    $$\Delta^*\mathscr{H}om(\pi_1^*{F}, \pi_2^*{G}) \rightarrow \mathscr{H}om({F, G}) \rightarrow \dot\pi_*(\mu hom({F, G})|_{S^{*}M})$$
    where $\Delta: M \hookrightarrow M \times M$ is the diagonal embedding and $\pi_i: M \times M \rightarrow M$ are the projections.
\end{proposition}


\subsection{Various sheaf categories}\label{variouscat}
    We have defined the sheaf of stable categories $\Sh_\Lambda$ and $\msh_\Lambda$ consisting of sheaves and respectively microsheaves. However, in general we may want to work with the subcategories of compact objects or proper objects. We explain how to restrict to these categories. Most of the discussions can be found in \cite{Nadler-pants}*{Section 3.6 \& 3.8} and \cite{Ganatra-Pardon-Shende3}*{Section 4.5}.

    Throughout the discussion, we will be considering the microlocal sheaf category $\msh_\Lambda$ on a subanalytic Legendrian (or conical Lagrangian) subset.

\begin{definition}
    For ${F} \in \msh_\Lambda(\Omega)$, we call it a compact object if $\Hom_{\msh_\Lambda(\Omega)}({F}, -)$
    commutes with filtered colimits. Let $\msh_\Lambda^c(\Omega) \subset \msh_\Lambda(\Omega)$ be the full subcategory of compact objects.
\end{definition}

    In particular, when we consider the category of compact objects
    $$\Sh^c_\Lambda(M) = \msh^c_{M \cup \widehat{\Lambda}}(T^*M),$$
    we can prove that under the compactness assumption on $M$ it is a smooth category in the sense of \cite[Definition 8.1.2]{Kontsevich-Soibelman-Ainfty} (see also \cite[Definition 4.6.4.13]{Lurie-HTT}), namely that (for the small category $A = \Sh^c_\Lambda(M)$ under consideration) the diagonal bimodule
    $$A_\Delta(X, Y) = \Hom_{A}(X, Y)$$
    is a perfect $A^{op} \times A$ bimodule.

\begin{proposition}[{\cite[Corollary 4.25]{Ganatra-Pardon-Shende3}}]\label{prop:smooth}
    Let $M$ be compact and $\Lambda \subseteq S^{*}M$ be a subanalytic isotropic subset. Then $\Sh^c_\Lambda(M)$ is a smooth category.
\end{proposition}

    We know that $\msh_\Lambda$ is both a sheaf and a cosheaf of categories, and in addition, for $\Omega \subseteq \Omega'$, the restriction functor
    $$r_{\Omega\Omega'}^*: \, \msh_\Lambda(\Omega') \rightarrow \msh_\Lambda(\Omega)$$
    preserves limits and colimits and thus admits left and right adjoints \cite[Lemma 4.12]{Ganatra-Pardon-Shende3}. Since $r_{\Omega\Omega'}^*$ preserves colimits, its left adjoint, which is called the corestriction functor
    $$r_{\Omega\Omega'!}: \, \msh_\Lambda(\Omega) \rightarrow \msh_\Lambda(\Omega'),$$
    preserves compact objects. Hence the corestriction functor restricts to the subsheaf of category of compact objects
    $$r_{\Omega\Omega'!}: \, \msh^c_\Lambda(\Omega) \rightarrow \msh^c_\Lambda(\Omega').$$

\begin{remark}\label{rem:cores-cpt}
    For closed subanalytic isotropic subsets $\Lambda \subseteq S^{*}M$, the microlocalization and its left adjoint in Section \ref{sec:micro}
    $$m_\Lambda: \Sh_\Lambda(M) \rightarrow \msh_\Lambda(\Lambda), \;\; m_\Lambda^l: \msh_\Lambda(\Lambda) \rightarrow \Sh_\Lambda(M)$$
    are special cases of restriction functors and corestriction functors. In particular, the left adjoint of microlocalization $m_\Lambda^l$ preserves compact objects
    $m_\Lambda^l: \msh^c_\Lambda(\Lambda) \rightarrow \Sh_\Lambda^c(M).$
\end{remark}

    Given sheaves of categories $\msh_{\Lambda}$ and $\msh_{\Lambda'}$, where $\Lambda \subseteq \Lambda'$ is a closed subset of an isotropic, there is an inclusion functor of categories
    $$\iota_{\Lambda\Lambda'*}: \, \msh_\Lambda(\Lambda) \rightarrow \msh_{\Lambda'}(\Lambda')$$
    which also preserves limits and colimits. Since it preserves limits and is accessible, there is a left adjoint called the pullback functor
    $$\iota^*_{\Lambda\Lambda'}: \, \msh_{\Lambda'}(\Lambda') \rightarrow \msh_\Lambda(\Lambda).$$
    Since $\iota_{\Lambda\Lambda' *}$ preserves colimits, $\iota^*_{\Lambda\Lambda'}$ preserves compact objects. Hence the corestriction functor preserves the sub-cosheaf of categories of compact objects.
    
\begin{remark}\label{rem:stoprem-cpt}
    For closed subanalytic isotropic subsets $\Lambda \subseteq \Lambda' \subseteq S^{*}M$, the inclusion functor and its left adjoint 
    $$\iota_{\Lambda \Lambda' *}: \Sh_\Lambda(M) \hookrightarrow  \Sh_{\Lambda'}(M), \;\; \iota_{\Lambda \Lambda'}^*: \Sh_{\Lambda'}(M) \rightarrow \Sh_\Lambda(M)$$
    are special cases of the inclusion and pullback functors above. In particular, the pullback functor preserves compact objects
    $\iota_{\Lambda \Lambda'}^*: \Sh_{\Lambda'}^c(M) \rightarrow \Sh_\Lambda^c(M).$
    This is also called the stop removal functor \cite{Ganatra-Pardon-Shende3}*{Corollary 4.22} (one can compare it to the stop removal functors in partially wrapped Fukaya categories \cite{Ganatra-Pardon-Shende2}*{Theorem 1.16}).
\end{remark}

    On the other hand, we can consider the subcategory with perfect stalks, which turns out to be the subcategory of proper modules (equivalently, pseudoperfect modules) in the category of (micro)sheaves.

\begin{definition}
    Let $\msh^b_\Lambda(\Omega) \subseteq \msh_\Lambda(\Omega)$ be the full subcategory of objects with perfect stalks, and $\msh^{pp}_\Lambda(\Omega) = \mathrm{Fun}^\text{ex}(\msh^c_\Lambda(\Omega)^{op}, \cV_0)$ be the category of proper modules in $\msh^c_\Lambda(\Omega)$, where $\mathrm{Fun}^\text{ex}(-, -)$ is the stable category of exact functors.
\end{definition}

    Since restriction functors in $\msh_\Lambda$ preserves (micro)stalks, the sheaf of categories $\msh_\Lambda$ can be restricted to a subsheaf of categories $\msh^b_\Lambda$. Meanwhile, since $\msh^c_\Lambda$ forms a cosheaf of categories under corestriction functors, we know that the full subcategories of proper submodules $\msh^{pp}_\Lambda$ also forms a sheaf of categories under restriction functors.
    
    The following theorem shows that $\msh^b_\Lambda(\Omega)$ is the equivalent to the subcategories of proper modules $\msh^{pp}_\Lambda(\Omega)$ in $\msh^c_\Lambda(\Omega)$.

\begin{theorem}[Nadler \cite{Nadler-pants}*{Theorem 3.21}, \cite{Ganatra-Pardon-Shende3}*{Corollary 4.23}]\label{thm:perfcompact}
    Let $\Lambda \subseteq S^{*}M$ (resp.~$\Lambda \subseteq T^*M$) be a compact subanalytic isotropic subset (resp.~a conic subanalytic isotropic subset). Then the natural pairing $\Hom_{\msh_\Lambda(\Omega)}(-, -)$ defines an equivalence
    $$\msh^b_\Lambda(\Omega) \simeq \msh^{pp}_\Lambda(\Omega) = \mathrm{Fun}^\text{ex}(\msh^c_\Lambda(\Omega)^{op}, \cV_0).$$
    In particular, when $M$ is compact, $\Sh^b_\Lambda(M) \simeq \Sh^{pp}_\Lambda(M) = \mathrm{Fun}^\text{ex}(\Sh^c_\Lambda(M)^{op}, \cV_0)$.
\end{theorem}

    Using the above theorem, for a subanalytic Legendrian $\Lambda \subseteq S^{*}M$ the category of proper modules
    $$\Sh^{pp}_\Lambda(M) = \msh^{pp}_{M \cup \widehat{\Lambda}}(T^*M),$$
    is a proper category (see \cite[Definition 8.2.1]{Kontsevich-Soibelman-Ainfty} or \cite[Definition 4.6.4.2]{Lurie-HTT}), namely that (for the small category $A = \Sh^{pp}_\Lambda(M) = \msh^{pp}_{M \cup \widehat{\Lambda}}(T^*M)$ under consideration) the diagonal bimodule $A_\Delta$ is a proper module, i.e.~for any $X, Y \in A$,
    $$\Hom_{A}(X, Y) \in \cV_0.$$

\begin{proposition}[{\cite[Corollary 4.25]{Ganatra-Pardon-Shende3}}]\label{prop:proper}
    Let $M$ be compact and $\Lambda \subseteq S^{*}M$ be a subanalytic isotropic subset. Then $\Sh^{pp}_\Lambda(M)$ is a proper category.
\end{proposition}

    Since $\Sh^c_\Lambda(M)$ is a smooth category, we know by \cite[Lemma A.8]{Ganatra-Pardon-Shende3} that $\Sh^{pp}_\Lambda(M) \subseteq \Sh^c_\Lambda(M)$. Therefore we have the following corollary.
    
\begin{corollary}[{\cite[Lemma A.8]{Ganatra-Pardon-Shende3}}]\label{cor:prop-in-perf}
    Let $M$ be compact and $\Lambda \subseteq S^{*}M$ be a subanalytic isotropic subset. Then $\Sh^{b}_\Lambda(M) \subseteq \Sh^c_\Lambda(M)$.
\end{corollary}

\begin{remark}\label{rem:prop}
    From the discussion above, we can show that 
    for closed subanalytic isotropic subsets $\Lambda \subseteq S^{*}M$, the microlocalization in Section \ref{sec:micro} preserves proper objects
    $m_\Lambda: \Sh_\Lambda^b(M) \rightarrow \msh_\Lambda^b(\Lambda),$
    and so does the inclusion functor
    $\iota_{\Lambda\Lambda' *}: \Sh_\Lambda^b(M) \hookrightarrow  \Sh_{\Lambda'}^b(M).$
\end{remark}

\section{Isotopy of sheaves}\label{sec:isotopies_of_sheaves}


Let $V$ be a contact manifold and $(B,b_0)$ be a pointed finite dimensional manifold.
We say a map $\Phi: V \times B \rightarrow V$ is a $B$-family contact isotopy 
if $\varphi_b \coloneqq \Phi(-,b)$ is a contactomorphism for all $b \in B$ and $\varphi_{b_0} = \id_V$.
To simplify the convention, when $B$ is some open interval containing $b_0 = 0$, 
we will surpass the notation and simply refer to $\Phi$ as a contact isotopy. 
We use $t$ as the parameter for this case.
When $V$ is co-oriented by $\alpha$, i.e., the contact structure is given by $\ker \alpha$, 
we say that an isotopy is positive if $\alpha(\partial_t \varphi_t) \geq 0$.
One important feature of such isotopies, in the Weinstein manifold setting, 
is that they induce continuation maps on Floer homology and is a key ingredient to define the wrapped Fukaya categories
\cite[Section 3.3]{Ganatra-Pardon-Shende1}.
We will be working in the sheaf-theoretical setting and focusing on the case when $V = S^*M$, the cosphere bundle of a manifold.
The foundational construction is performed in \cite{Guillermou-Kashiwara-Schapira} where they show that an isotopy on $S^* M$
produces a family of end-functors on $\Sh(M)$.
When the isotopy is positive, this family of end-functors comes with a family of continuation maps in the form of natural transforms.
We will recall this construction and some relevant results following the setting of \cite[Section 3]{Kuo-wrapped-sheaves}. 

\subsection{Continuation maps}\label{continuation maps}

Denote by $(t,\tau)$ the coordinate of $T^* I$ and consider $F \in \Sh(M \times I)$ such that $\ms(F) \subseteq \{ \tau \leq 0\}$.
By (9) of Proposition \ref{prop:mses} and the adjoint functor theorem, the inclusion of the subcategory formed by such sheaves admits both a left and a right adjoint.
For this special case, there is an explicit description of the two adjoints by a sheaf kernel:

\begin{proposition}[{\cite[Proposition 4.8]{Guillermou-Kashiwara-Schapira}}, \cite{Kuo-wrapped-sheaves}]
Let $\iota_*: \Sh_{T^* M \times T^*_\leq I}(M \times I) \hookrightarrow \Sh(M \times I)$
denote the tautological inclusion. 
Then the left adjoint $\iota^* \dashv \iota_* \dashv \iota^!$ is given by convolution and its right adjoint is
$$ \iota^* F = 1_{ \{ t^\prime > t \} }[1] \circ F,\ \iota^! F = \sHom^\circ ( 1_{ \{ t^\prime > t \} }[1], F).$$
Here we denote by $(t,t^\prime)$ the coordinate of $I^2$ and $\sHom^\circ$ is defined in Lemma \ref{convrad}.
\end{proposition}

A consequence of this expression is that $F = 1_{ \{ t^\prime > t \} }[1] \circ F$ for $F \in \Sh_{T^* M \times T^*_\leq I}(M \times I)$.
The virtual of this is explicit description is that the sheaf kernel $1_{ \{ t^\prime > t \} }$ admits maps between its slices,
$$1_{(-\infty,s)} \rightarrow 1_{(-\infty,t)}, \ \text{for} \ s \leq t.$$
Denote by $i_t: M \hookrightarrow M \times I$ the inclusion of the $t$-slice and write $F_t \coloneqq i_t^* F$.
Since convolutions are compatible with $*$-pullback, namely $i_t^* ( 1_{ \{ t^\prime > t \} }[1] \circ F) = 1_{(-\infty,t)}[1] \circ F$,
we term for $F \in Sh_{T^* M \times T^*_\leq I}(M \times I)$ the canonical morphisms 
$$c(s,t,F): F_s \rightarrow F_t, \ \text{for} \ s \leq t$$
induced from $1_{(-\infty,s)} \rightarrow 1_{(-\infty,t)}$ the continuation map of $F$.
Some properties of the continuation maps are listed as follow:

\begin{proposition}\label{prop:properties_of_continuation_maps} 
Let $F \in \Sh_{T^* M \times T^*_\leq I}(M \times I)$. Then

\begin{enumerate}
\item For $r \leq s \leq t$, there is an equality $c(r,s,F) \circ c(s,t,F) = c(r,t,F)$.
\item If $F$ is a constant on the $I$-direction on $[s,t]$, then $c(s,t,F) = \id$.
\item Continuation maps respect colimits forward, the canonical map $\operatorname{colim}_{r < t}  F_r \rightarrow F_t$ is an isomorphism.
\item Assume further that $\ms(F)$ is $I$-noncharacteristic, then
continuation maps respect limits backward, i.e., the canonical map $F_t \rightarrow \lim_{s > t}  F_s$ is an isomorphism.
\end{enumerate}
\end{proposition}

\begin{remark}
We note that the noncharacteristic condition cannot be dropped.
For example, consider the case $M = \{ * \}$ and take $F = 1_{(-\infty,0]}$. 
Then $F_t = 1_\cV$ when $t \leq 0$ and $0$ otherwise.
Thus $F_0 = \operatorname{colim}_{r < 0}  F_r$ but $F_0 \neq \lim_{s > 0}  F_s$.
\end{remark}


We also mention some homotopical invariance properties of the continuation maps.
Let $J$ another open interval and use $(s,\sigma)$ to denote the coordinates of its cotangent bundle.
Let $G \in \Sh(M \times I \times J)$ be a sheaf such that $\ms(G) \subseteq \{ \tau \leq 0 \}$.
For any $x \in I$, we use $G_{t = x} \coloneqq G|_{M \times \{x\} \times J }$ to denote the restriction
and similarly for $G_{s = y}$, $y \in J$.
Note by (2) of Proposition \ref{prop:noncharacteristic-ms-es}, the same condition $\ms(G_{s = y}) \subseteq \{ \tau \leq 0 \}$ holds.
Assume further that there exists $a \leq b$ in $I$ such that $\ms(G_{t = a})$, $\ms(G_{t = b}) \subseteq T^* M \times 0_J$.
By Lemma \cite[Lemma 3.22]{Kuo-wrapped-sheaves}, this implies that there exist $F_a$, $F_b \in \Sh(M)$ such that
$G_{t = a} = p_s^* F_a$ and $G_{t = b} = p_s^* F_b$
where we use $p_s: M \times J \rightarrow M$ to denote the projection.
Note that, for each $y \in J$, the restriction $G_{s = y}$ induces a continuation map
$c(G,y,a,b): F_a \rightarrow F_b$.

$$
\begin{tikzpicture}

\draw [thick] (0,0) rectangle (5,3);

\draw [thick] (1,0) -- (1,3);
\draw [thick] (4,0) -- (4,3);

\node at (0.6,2.5) {$F_a$};
\node at (3.6,2.5) {$F_b$};

\draw [->, thick] (1.1,1.5) -- (3.9,1.5) node [midway, above] {$c(G,y,a,b)$};
\draw [->, thick] (1.1,0.5) -- (3.9,0.5) node [midway, above] {$c(G,y^\prime,a,b)$};

\node at (1,3.25) {$t = a$};
\node at (4,3.25) {$t = b$};
\node at (-0.7,0.5) {$s = y^\prime$};
\node at (-0.7,1.5) {$s = y$};

\end{tikzpicture}
 $$

\begin{proposition} \label{homotopy_independence_of_continuation_maps}
The morphism $c(G,y,a,b)$ is independent of $y \in J$.
More generally, similar statements can be made for higher homotopical independence.
\end{proposition}

\subsection{Isotopies of sheaves}\label{Isotopies of sheaves}

We recall the notion of isotopies of sheaves based on the main theorem of Guillermou, Kashiwara, and Schapira
in \cite{Guillermou-Kashiwara-Schapira} and some applications. 
This section is a summary of \cite[section 3.2]{Kuo-wrapped-sheaves} of the first author. 

\begin{theorem}[{\cite[Proposition 3.2, Remark 3.9]{Guillermou-Kashiwara-Schapira}}]\label{thm:GKS}
Let $M$ be a manifold and $B$ a contractible finite dimensional manifold.
For a $B$-family contact isotopies $\Phi:  S^* M \times B \times J \rightarrow S^* M$ where $J$ is an open interval,
there exists a unique sheaf kernel
$K(\Phi) \in \Sh(M \times M \times J \times B)$ such that
\begin{enumerate}
\item $K(\Phi) |_{b = b_0} = 1_{\Delta_M}$, and 
\item $\msnz(K(\Phi)) \subseteq \Lambda_\Phi$ where
\begin{equation}\label{contact_movie}
\Lambda_\Phi = \left\{ \left(x, -\xi, \varphi_{t,b}(x,\xi), t, - \alpha(V_{\Phi_b})(\varphi_{t,b}(x,\xi)),
b, - \alpha_{\varphi_{t,b}(x,\xi)} \circ d (\Phi \circ i_{x,\xi,t})_b(\cdot) \right) \right\}
\end{equation}
is the contact movie of $\Phi$.
\end{enumerate}
Moreover, $\msnz(K(\Phi)) = \Lambda_\Phi$ is simple along $\Lambda_\Phi$,
both projections $\supp(K) \hookrightarrow M \times M \times B \rightarrow M \times B$ are proper,
and the composition is compatible with convolution in the sense that
\begin{enumerate}
\item $K(\Psi \circ \Phi) = K(\Psi) \circ|_{B} K(\Phi)$,
\item $K(\Phi^{-1}) \circ|_{B} K(\Phi) = K(\Phi) \circ|_{B} K(\Phi^{-1}) = 1_{\Delta_M \times B}$.
\end{enumerate}
Here $\Phi^{-1}$ is the $B$-family of isotopies given by $\Phi^{-1}(-,t,b) \coloneqq \varphi_{t,b}^{-1} $.
\end{theorem}
This theorem is usually referred as \textit{Guillermou-Kashiwara-Schapira sheaf quantization},
since it is a categorical analogue of producing an operator from an contact isotopy,
and we will refer the sheaf kernel $K(\Phi)$ as the GKS sheaf quantization kernel associated to $\Phi$.
A corollary of this construction is that contact isotopies act on sheaves and the action is compatible with the microsupport:
\begin{corollary}[{\cite[Equation (4.4)]{Guillermou-Kashiwara-Schapira}}]\label{cor:GKS}
Let $\Phi:S^* M \times B \rightarrow S^* M$ be a contact isotopy.
Then the convolution
\begin{align*}
K(\Phi)_t \circ (-): \Sh(M) &\rightarrow \Sh(M) \\
F &\mapsto K(\Phi)_t \circ F
\end{align*} is an equivalence with inverse $K(\Phi^{-1})_t \circ (-)$.
For any $F \in \Sh(M)$,
$\msnz(K(\Phi) \circ F) = \Lambda_\Phi \circ \dot{\ms}(F)$.
In particular, if we set $F_t \coloneqq K(\Phi)_t \circ F$, then $\msnz(F_t) = \varphi_t(\msnz(F))$ for $t \in B$.
Furthermore, if $F$ has compact support, then so does $F_t$ for all $t \in B$.
\end{corollary}

\begin{corollary}[{\cite[Theorem 7.2.1]{KS}}, {\cite[Lemma 5.6]{NadShen}}]\label{cor:cont-trans}
Let $\Phi:S^* M \times B \rightarrow S^* M$ be a contact isotopy.
Then the convolution induces a morphism between sheaf of categories
\begin{align*}
K(\Phi)_t \circ (-): \msh_\Lambda &\rightarrow \varphi_t^* \msh_{\varphi_t(\Lambda)} \\
F &\mapsto K(\Phi)_t \circ F
\end{align*} and is an equivalence whose inverse is given by $K(\Phi^{-1})_t \circ (-)$.
\end{corollary}

\begin{example}\label{GKS_standard_Reeb}
Let $(M,g)$ be a Riemannian manifold with non-zero injective radius.
The cosphere bundle $S^* M$ can be identified with the unit sphere bundle in $T^* M$ by $g$ as a contact hypersurface.
Take $\Phi$ to be the Reeb flow and $F = 1_{x}$ to be a skyscraper at some point $x \in M$.
Then, for small $t < 0$, 
$F_t$ is given by $1_{\overline{B_{\epsilon(t)}(x)}}$ the constant sheaf supported on some small closed ball center at $x$,
and, for small $t > 0$, $F_t$ is given by $\sHom(1_{\overline{B_{\epsilon(t)}(x)}}, \omega_M)$ wher $\omega_M$ is the dualizing sheaf. 
When the base manifold $M$ is orientable, the later is isomorphic to
$1_{B_{\epsilon(t)}(x)}[\dim M]$ the constant sheaf supported on some
small open ball centered at $x$ with a shift by the dimension of $M$.
\end{example}

Now, consider the case when $B = I$ is given by an open interval containing $0$ and assume that $\Phi$ is positive, i.e., 
$\alpha(\partial_t \varphi_t) \geq 0$.
In this case, the consideration of the previous section \ref{continuation maps} implies that there are continuation maps
$$ K(\Phi)_s \rightarrow K(\Phi)_t, \ \text{for} \ s \leq t$$
and it induces continuation maps $F_s \rightarrow F_t$, $s \leq t$, for $F \in \Sh(M)$.
We note that if there is an homotopy between two positive isotopies $\Phi$ and $\Psi$,
then by Proposition \ref{homotopy_independence_of_continuation_maps}, the induced continuation maps
by $K(\Phi)$ and $K(\Psi)$ are identified.
Now given a closed subset $X \subseteq S^* M$, one can consider the totality of positive isotopies,
declare that a morphism between two positive isotopies $\Phi_1 \rightarrow \Phi_2$ is a further isotopy $\Psi$ of the same kind from $\Phi_1$
such that $\Phi_2 = \Psi \# \Phi_1$, and one compares different morphisms by homotopies, etc..
Then, for $F \in \Sh(M)$, GKS sheaf quantization produces a diagram whose vertices are given by the time-$1$ sheaves $F^w$,
and whose arrows are given by continuation maps $c(\Psi): F^w \rightarrow F^{w^\prime}$.
The main theorem we recall in this section will be that colimit/limit over increasingly positive/negative isotopies
provides a description for the tautological inclusion $\Sh_X(M) \subseteq \Sh(M)$.

\begin{theorem}[{\cite[Theorem 1.2]{Kuo-wrapped-sheaves}}]\label{w=ad}
Let $\iota_{X *}: \Sh_X(M) \hookrightarrow \Sh(M)$ denote the tautological inclusion.
Then the left and right adjoints are given by the positive/negative colimiting/limiting wrapping
$$ \wrap_X^+(F) \coloneqq \clmi{F \rightarrow F^w} \, F^w, \ \wrap_X^-(F) \coloneqq \lmi{F^{w-} \rightarrow F} \, F^{w-}.$$
\end{theorem}

For $G \in \Sh(M)$, it is in general hard to compute $\wrap_X^+(G)$ (resp.~$\wrap_X^-(G)$) since it is given by a colimit (resp.~a limit) over a rather large index category. 
Nevertheless, when $X = \Lambda$ and $\msif(G)$ are both isotropic, the underlying geometry can sometimes provide a cofinal (resp.~final) one parameter family $G_t$ so that $G_0 = G$ and $\wrap_\Lambda^+ G = \operatorname{colim}_{t \rightarrow +\infty} \, G_t$ (resp.~$\wrap_\Lambda^- G = \lim_{t \rightarrow -\infty} G_t$). In this case, a natural question is when, for a fixed $F \in \Sh(M)$ also with isotropic singular support, the canonical map
$$ \Hom(F,G) \rightarrow \Hom(F, \wrap_\Lambda^+G)$$
is an isomorphism. One such a case which we will encounter is the following:
\begin{lemma}\label{lem:nearby_cycle}
Let $\Lambda$ be a fixed compact isotropic and $F, G \in \Sh(M)$ with isotropic singular support. Assume that $\msif(F) \cap \Lambda = \varnothing$, and there is a positive isotopy $\varphi_t$, $t \in \RR$, on $S^* M$ such that for any open neighborhood $\Omega$ of $\Lambda$, there is $T = T(\Lambda)$ such that $\varphi_t(\msif(G)) \subseteq \Omega$ for $t \geq T$, and $\msif(F) \cap \varphi_t( \msif(G)) = \varnothing$ for all $t \geq 0$, then the canonical map
$$ \Hom(F,G) \rightarrow \Hom(F, \wrap_\Lambda^+G)$$
is an isomorphism. A similar statement holds for $\Hom(G,F) \rightarrow \Hom(\wrap_\Lambda^-G,F)$ when given a negative isotopy satisfying a similar condition.
\end{lemma}

\begin{proof}
This is essentially \cite[Theorem 5.15]{Kuo-wrapped-sheaves}. The point is that we would like to apply the main theorem about nearby cycle, \cite[Theorem 4.2]{NadShen}. Although we did not assume compactness on $\supp(F)$ and $\supp(G)$ here as in \cite[Theorem 5.15]{Kuo-wrapped-sheaves}, the compactness assumption on $\Lambda$ will be sufficient to implied the gappedness condition for \cite[Theorem 4.2]{NadShen}.
\end{proof}

\begin{remark}
    In practice, when we can find an increasing sequence of positive Hamiltonian flows $\varphi^k_t$, $k \in \NN$, such that for any open neighborhood $\Omega$ of $\Lambda$, there is $K \in \NN$ such that $\varphi^k_t(\msif(G)) \subseteq \Omega$ for $k \geq K$, then the condition in the lemma holds. Indeed, we can define a time dependent smooth Hamiltonian $\varphi^t$ such that $\varphi_t(\msif(G)) = \varphi^k_{t-k}(\msif(G))$ when $t \in [k + 1 - \epsilon, k+1]$. That satisfies the condition in the lemma.
\end{remark}



\section{Doubling and fiber sequence}\label{sec:doubling}

    Our goal in this section is to interpret the left and right adjoint functors of microlocalization
    $$m_\Lambda: \Sh_\Lambda(M) \rightarrow \msh_\Lambda(\Lambda).$$
    by the doubling construction in sheaf theory (which is also known as the antimicrolocalization functor \cite{NadShen} or the Guillermou convolution functor \cite{JinTreu}).

    First, we will realize the doubling functor with respect to an arbitrary Reeb flow $T_t$, $t \in \RR$, on $S^{*}M$ and show that this defines a fully faithful functor.
    
\begin{theorem}\label{thm:doubling}
    Let $\Lambda \subseteq S^{*}M$ be a compact subanalytic Legendrian and $c(\Lambda)$ be the length of the shortest Reeb chord on $\Lambda$ with respect to the Reeb flow $T_t$. Then for $0 < \epsilon < c(\Lambda)/2$, there is a fully faithful functor
    $$w_\Lambda: \msh_\Lambda(\Lambda) \hookrightarrow \Sh_{T_\epsilon(\Lambda) \cup T_{-\epsilon}(\Lambda)}(M).$$
\end{theorem}

    Then, we show that by wrapping positively or negatively around $S^{*}M \backslash \Lambda$, we will get the left and right adjoints from the doubling functor.

\begin{theorem}\label{thm:doubling_ad}
    Let $\Lambda \subseteq S^{*}M$ be a compact subanalytic Legendrian. Then there are equivalences
    $$m_\Lambda^l = \wrap_\Lambda^+ \circ w_\Lambda[-1], \;\; m_\Lambda^r = \wrap_\Lambda^- \circ w_\Lambda: \msh_\Lambda(\Lambda) \rightarrow \Sh_\Lambda(M)$$
    where $\wrap_\Lambda^\pm$ are the functors given in Theorem \ref{w=ad}.
    In particular, the left and the right adjoint $m_\Lambda^l$ and $m_\Lambda^r$ can be decomposed to a inclusion followed by a quotient:
    $$ \msh_\Lambda(\Lambda) \hookrightarrow \Sh_{T_\epsilon(\Lambda) \cup T_{-\epsilon}(\Lambda)}(M) \twoheadrightarrow \Sh_\Lambda(M).$$
\end{theorem}

    The doubling functor in sheaf theory goes back to Guillermou \cite{Gui}*{Section 13-15}, and is also formulated in a different way in Nadler-Shende \cite{NadShen}*{Section 6}. Here we will generalize that functor to arbitrary Reeb flows on $S^{*}M$. In Lagrangian Floer theory, the stop doubling construction has been discussed in the setting of Fukaya-Seidel categories \cite{AbGan} (see also \cite{AbSmithKhov,AbAurouxHMS}) as the cup functor and also in the setting of partially wrapped Fukaya categories as the doubling trick \cite{Ganatra-Pardon-Shende2}*{Example 8.7}, cup functor or Orlov functor \cite{SylvanOrlov}. Recently the doubling trick has been used in the theory of (twisted) generating families \cite{TwistGF}*{Theorem~C}. 
    
    While Theorem \ref{thm:doubling} is a immediate corollary of Nadler-Shende \cite{NadShen}*{Section 6}, Theorem \ref{thm:doubling_ad} which relates the doubling functor to both the left and right adjoint of microlocalization is rather a nontrivial result and does not seem to follow from \cite{NadShen}*{Section 6}, which is why we need to generalize the approach in Guillermou \cite{Gui}*{Section 13-15}. 
    
    We will deduce the doubling construction through two independent approaches (that are morally closely related), both using sheaf theoretic wrappings discussed in Section \ref{Isotopies of sheaves}.
    
    For the first approach in Section \ref{sec:sato-sab} and \ref{sec:doubling-local}, we consider a single wrapping defined globally. The difference between positive and negative wrapping which leads to the Sato-Sabloff fiber sequence, which provides a new interpretation of the Sato fiber sequence (Theorem \ref{thm:sato}) in microlocal theory of sheaves from the perspective of Hamiltonian isotopies of sheaves. Then we define the doubling by locally considering the difference between positive and negative wrapping of the sheaf. While this approach is very straight forward, the disadvantage is that it involves auxiliary choices and relies on the crucial observation on positive and negative wrappings to begin with.

    For the second approach in Section \ref{sec:ad-micro} and \ref{sec:double-ad-micro}, we consider functorial gluings of families of small wrappings defined locally. We interpret local adjoints of microlocalizations by small wrappings following Theorem \ref{w=ad} and then consider the gluing of the local adjoints of microlocalizations with respect to corestriction functors. While this approach is more functorial and independent of auxiliary choices, the disadvantage is that it is really involved when taking the colimits and showing the relation between the left and right adjoint also requires extra work.
    
    Finally, we also show in Section \ref{sec:serre} a Sabloff-Serre duality using the Verdier duality on sheaves, which has appeared in a number of works in symplectic geometry \cite{Sabduality,EESduality,SeidelFukI}.

\subsection{Sato-Sabloff fiber sequence}\label{sec:sato-sab}

For compact subanalytic Legendrians $\Lambda_{0}, \Lambda_{1} \subset S^{*}M$, we let $c(\Lambda_0, \Lambda_1)$ be the minimal absolute value of lengths of Reeb chords between $\Lambda_0$ and $\Lambda_1$ with respect to the Reeb flow $T_t$. Abusing notations, we also use $T_t$ to denote the associated functor of its time-$t$ flow which acts on sheaves on $M$. 
The key proposition of this section is that the $\Hom$ in $\msh_\Lambda(\Lambda)$ can be computed as a difference between the positive and negative wrappings.

Similar considerations have also appeared in previous works of for example Guillermou \cite[Section 11--13]{Gui} and Tamarkin \cite[Equation~(1)]{Tamarkin2}.

\begin{theorem}[Sato-Sabloff fiber sequence]\label{thm:sato-sab}
    Let $\Lambda, \Lambda' \subseteq S^*M$ be compact subanalytic Legendrians. Let $T_t: S^*M \to S^*M$ be a non-negative contact flow such that 
    $$\Lambda \cap T_\epsilon(\Lambda') = \varnothing$$
    for any small $\epsilon \neq 0$. Then for $F \in \Sh_{\Lambda}(M)$ and $G \in \Sh_{\Lambda'}(M)$ such that $\mathrm{supp}(F) \cap \mathrm{supp}(G)$ is compact in $M$, there is a fiber sequence
    $$\Hom(F, T_{-\epsilon}(G)) \xrightarrow{c} \Hom(F, T_{\epsilon}(G)) \rightarrow  \Gamma(\dT^* M, \mhom(F, G))$$
    where $c$ is induced by the continuation map $T_{-\epsilon}(G) \rightarrow T_\epsilon(G)$.
\end{theorem}
\begin{remark}\label{rem:multi-component-sato}
    We remark that the above computation in particular works in the case when $\Lambda = \Lambda'$ and we take microlocalization along a single connected component $\Lambda_i \subseteq \Lambda \subseteq S^{*}M$. Let $\widetilde{T}_t: S^{*}M \rightarrow S^{*}M$ be a Hamiltonian flow such that $\widetilde{T}_t|_{T_t(\Lambda_i)} = T_t$ is the Reeb flow while $\widetilde{T}_t|_{\Lambda \backslash \Lambda_i} = \mathrm{id}$. Then there is a fiber sequence 
    $$\Hom(F, \widetilde{T}_{-\epsilon}(G)) \rightarrow \Hom(F, \widetilde{T}_{\epsilon}(G)) \rightarrow \Gamma(\Lambda_i, \mhom(F,G) )$$
\end{remark}

The above expression will be useful in Section \ref{sec:doubling-local} when we construct the doubling functor in Theorem \ref{thm:doubling}. Here we note that there is another expression, 
$$\Hom(\wrap_\Lambda^- T_{-\epsilon}(F),G) \rightarrow \Hom(F,G) \rightarrow \Hom(m_\Lambda^l m_\Lambda (F),G)$$
where we use Theorem \ref{w=ad} to identify $\Hom(\wrap_\Lambda^- T_{-\epsilon}(F),G) = \Hom(T_{\epsilon}(F), G) = \Hom(F,T_{-\epsilon}(G))$ and Definition-Theorem \ref{mu-hom_as_hom} to identify $\Gamma(\Lambda,\mhom(F,G)) = \Hom(m_\Lambda^l m_\Lambda (F), G)$.
Equivalently, we have the fiber sequence
$$m_\Lambda^l m_\Lambda \rightarrow \id \rightarrow \wrap_\Lambda^+ T_{\epsilon}$$
between endofunctors on $\Sh_\Lambda(M)$. A similarly discussion holds for the left adjoints. Later in Section \ref{sec:ad-micro} we will explain another perspective of understanding the fiber sequence.

\begin{definition}\label{def:wrap_once_functors}
We define the positive and negative warp-once functor $S_\Lambda^+$ and $S_\Lambda^-: \Sh_\Lambda(M) \rightarrow \Sh_\Lambda(M)$ as the compositions
$$S_\Lambda^+(F) \coloneqq \wrap_\Lambda^+ T_{\epsilon}(F), \; S_\Lambda^-(F) \coloneqq \wrap_\Lambda^- T_{-\epsilon} (F) .$$
\end{definition}

The above discussion show that $S_\Lambda^-$ and $S_\Lambda^+$ is the cotwist and dual cotwist associated to the adjunctions $m_\Lambda^l \dashv m_\Lambda \dashv m_\Lambda^r$. See Section \ref{sec:spherical} for the terminology.

Using the Sato fiber sequence Theorem \ref{thm:sato}, the proof of Theorem \ref{thm:sato-sab} is reduced to the following proposition. Let $p: M \to \{*\}$ be the projection map.

\begin{proposition}\label{prop:hom_w_pm}
    Let $\Lambda_{0}$, $\Lambda_{1} \subseteq S^{*}M$ be compact subanalytic Legendrians, ${F} \in \Sh_{\Lambda_0}(M)$, ${G} \in \Sh_{\Lambda_1}(M)$ and $\mathrm{supp}(F) \cup \mathrm{supp}(G)$ is compact. Then there is a commutative diagram
$$
\begin{tikzpicture}
\node at (0,1.7) {$p_*\Delta^*\mathscr{H}om(\pi_1^*{F}, \pi_2^*{G})$};
\node at (6,1.7) {$\Hom({F, G})$};
\node at (0,0) {$\Hom({F}, T_{-\epsilon}({G}))$};
\node at (6,0) {$\Hom({F}, T_{\epsilon}({G}))$};

\draw [->, thick] (2.1,1.7) -- (4.9,1.7) node [midway, above] {$ $};
\draw [->, thick] (1.6,0) -- (4.5,0) node [midway, above] {$c$};

\draw [->, thick] (0,1.4) -- (0,0.3) node [midway, left] {\rotatebox{90}{$\sim$}}; 
\draw [->, thick] (6,1.4) -- (6,0.3) node [midway, right] {\rotatebox{90}{$\sim$}};
\end{tikzpicture}
$$

\noindent where $c$ is the continuation map associated to the Reeb flow and the bottom arrow is the canonical map in Theorem \ref{thm:sato}.
\end{proposition}

Before entering the proof of Proposition \ref{prop:hom_w_pm}, we recall that the continuation map $T_{-\epsilon}(F)\rightarrow T_\epsilon(F)$ is constructed the GKS sheaf kernel associated to the Reeb flow.
Write $q: M \times \mathbb{R} \rightarrow M$ and $t: M \times \mathbb{R} \rightarrow \mathbb{R}$ for the projection maps. For a subanalytic Legendrian $\Lambda \subseteq S^{*}M$, consider the Legendrian movie of $\Lambda$ under the identity flow
$$\Lambda_q = \{(x, \xi, t, 0) | (x, \xi) \in \Lambda, t \in \RR\} \subset S^{*}(M \times \mathbb{R}).$$ 
Let $T_t: S^{*}M \rightarrow S^{*}M$ be any Reeb flow defined by the positive Hamiltonian $H: S^{*}M \rightarrow \mathbb{R}$ and consider the Legendrian movie of $\Lambda$ under the Reeb flow
$$\Lambda_T = \{(x, \xi, t, \tau) | (x, \xi) \in T_t(x_0, \xi_0), \tau = -H \circ T_t(x_0, \xi_0), (x_0, \xi_0) \in \Lambda\} \subset S^{*}(M \times \mathbb{R}).$$
A standard trick is to consider the total sheaf Hom, $\mathscr{H}om(q^*{F}, K(T)  \circ {G})$. 
The following singular support estimate is essentially the same as \cite{LiEstimate}*{Lemma 4.1}.

    Let $\mathcal{Q}_\pm(\Lambda_0, \Lambda_1)$ be the set of unoriented Reeb chords from $\Lambda_0$ to $\Lambda_1$, namely
    $$\mathcal{Q}_\pm(\Lambda_0, \Lambda_1) = \{(x_0, \xi_0, x_1, \xi_1) \in \Lambda_0 \times \Lambda_1 |\, \exists\, t \in \mathbb{R}, T_t(x_0, \xi_0) = (x_1, \xi_1)\}.$$
    For a Reeb chord such that $T_t(x_0, \xi_0) = (x_1, \xi_1)$, we call $t \in \RR$ the length of the Reeb chord.

\begin{lemma}\label{lem:ss-reeb}
    Let $\Lambda_{0,1} \subset S^{*}M$ be subanalytic Legendrians, $T_t: S^{*}M \rightarrow S^{*}M$ be any Reeb flow and ${F} \in \Sh_{\Lambda_0}(M), {G} \in \Sh_{\Lambda_1}(M)$. Then
    \begin{gather*}
    \ms^\infty(\mathscr{H}om(q^*{F}, K(T)  \circ {G})) \cap \{(x, 0, t, \tau) \in S^{*}(M \times \mathbb{R}) | \tau > 0\} = \varnothing, \\
    \ms^\infty(\mathscr{H}om(q^*{F}, K(T) \circ {G})) \cap \{(x, 0, t, \tau) \in S^{*}(M \times \mathbb{R}) | \tau < 0\} \hookrightarrow \mathcal{Q}_\pm(\Lambda_0, \Lambda_1).
    \end{gather*}
    The $t$ coordinates in the intersection correspond to lengths of Reeb chords in $\mathcal{Q}_\pm(\Lambda_0, \Lambda_1)$.
    In particular, $\mathscr{H}om(q^*{F}, K(T) \circ {G})$ is $\RR$-noncharacteristic away from the length spectrum of Reeb chords.
\end{lemma}
\begin{proof}
    Since $\ms^\infty(q^*{F}) \cap \ms^\infty(K(T) \circ {G}) = \Lambda_{0,q} \cap \Lambda_{1,T} = \varnothing$, we can apply the singular support estimate (8) of Proposition \ref{prop:mses}
    $$\ms^\infty(\mathscr{H}om(q^*{F}, K(T) \circ {G})) \subset (-\ms^\infty(q^*{F})) + \ms^\infty(K(T) \circ {G}) = (-\Lambda_{0,q}) + \Lambda_{1,T}.$$
    Hence $(x, 0, t, \tau) \in (-\Lambda_{0,q}) + \Lambda_{1,T}$ if and only if there exists a pair $(x_0, \xi_0) \in \Lambda_0, (x_1, \xi_1) \in \Lambda_1$ such that $(x_1, \xi_1) = T_t(x_0, \xi_0)$, or in other words there is a Reeb chord from $\Lambda_0$ to $\Lambda_1$ of length $u$. In particular, we know that $\tau = -H(x_0, \xi_0) < 0$ is determined by such a pair. Hence when $\tau > 0$, there will never be $(x, 0, t, \tau) \in (-\Lambda_{0,q}) + \Lambda_{1,T}$. Therefore
    \begin{gather*}
    \ms^\infty(\mathscr{H}om(q^*{F}, K(T) \,\circ  {G})) \cap \mathrm{Graph}(dt) = \varnothing, \\
    \ms^\infty(\mathscr{H}om(q^*{F}, K(T) \circ {G}))  \,\cap \mathrm{Graph}(-dt) \hookrightarrow \mathcal{Q}_\pm(\Lambda_0, \Lambda_1),
    \end{gather*}
    where our injection maps $(x, 0, t, -\tau)$ to the Reeb chord of length $u$ connecting $(x_0, \xi_0) \in \Lambda_0$ and $(x_1, \xi_1) = T_t(x_0, \xi_0) \in \Lambda_1$.
\end{proof}

\begin{proof}[Proof of Proposition \ref{prop:hom_w_pm}]
Denote by $i_t$ the inclusion of the slice of $M \times M$ at $t$.
We first prove the more straightforward statement of $\Hom(F,G) \xrightarrow{\sim} \Hom(F,T_\epsilon(G))$.
The above Lemma \ref{lem:ss-reeb} implies that, by Proposition \ref{prop:mses}~(5), the $\epsilon$-slice of the total Hom sheaf $\mathscr{H}om(q^*{F}, K(T)  \circ {G})$ is the same as
$$ i_\epsilon^* \mathscr{H}om(q^*{F}, K(T)  \circ {G}) = i_\epsilon^! \mathscr{H}om(q^*{F}, K(T)  \circ {G})[-1] = \sHom(F,T_\epsilon(G)).$$
Thus, we may apply Proposition \ref{prop:properties_of_continuation_maps}~(3) and get
$$\sHom(F,G) \xrightarrow{\sim} \lmi{t \rightarrow 0^+} \sHom(F,T_t(G)).$$
Applying $\Gamma(M;-)$, we obtain that $\Hom(F,G) \xrightarrow{\sim} \lim_{t \rightarrow 0^+}  \Hom(F,T_t(G))$.
Denote by $t: M \times \RR \rightarrow \RR$ the projection to the parameter space.
But by the above Lemma \ref{lem:ss-reeb} and Proposition \ref{prop:mses}~(4), the sheaf $t_* \mathscr{H}om(q^*{F}, K(T)  \circ {G}) $ is a constant sheaf over $\left(0,c(\Lambda) \right)$.
Note we use the assumption that $\supp(F)$ and $\supp(G)$ are compact in order to obtain microsupport estimation for pushforward.
Thus the later limit when restricting to $0 < \epsilon < c(\Lambda)$ is a constant diagram and the projection
$$\lmi{t \rightarrow 0^+}  \Hom(F,T_t(G)) \xrightarrow{\sim} \Hom(F,T_{\epsilon}(G))$$
is an isomorphism for $0 < \epsilon < c(\Lambda)$.

To prove the statement for $\Hom(F,T_{-\epsilon}(G)) \rightarrow \Hom(F,G)$,
let $\pi_{i,\RR}: M \times M \times \RR \rightarrow M \times \RR$ denote the $\RR$-parameter version of the projection to the $i$-th component.
Instead of the total Hom sheaf $ \mathscr{H}om(q^*{F}, K(T)  \circ {G})$, we will consider its $*$-variant $(\Delta \times \id_{\RR})^* \sHom( \pi_{1,\RR}^* q^* F, \pi_{2,\RR}^* K(T) \circ G)$. Proposition \ref{prop:mses}~(5) implies that the canonical map $f^*H \otimes f^!1_Y \rightarrow f^!H$ is an isomorphism when $f$ is noncharacteristic to the sheaf $H$.
Thus, there is a canonical map 
$$\Delta^* \sHom(\pi_1^* F, \pi_2^* G) \rightarrow \Delta^! \sHom(\pi_1^* F, \pi_2^! G) = \sHom(F,G).$$
Here, we use the fact that $\Delta^! 1_{M \times M} = \omega_M^{-1}$, by \cite[Lemma 2.41]{Kuo-wrapped-sheaves}, is an invertible sheaf so we can multiply the morphism with its inverse $\omega_M$.
Similarly, there is an canonical map $$(\Delta \times \id_{\RR})^* \sHom( \pi_{1,\RR}^* q^* F, \pi_{2,\RR}^* K(T) \circ G) \rightarrow \mathscr{H}om(q^*{F}, K(T)  \circ {G})$$
which is an isomorphism over $\left(-c(\Lambda),0 \right)$ by a similar microsupport estimation as the above Lemma \ref{lem:ss-reeb} and Proposition \ref{prop:mses}~(5).
Thus, by consider the $\epsilon$-slice for $-c(\Lambda) < \epsilon < 0$ and the $0$-slice, we obtain the following commutative diagram:
$$
\begin{tikzpicture}
\node at (0,1.7) {$\Delta^* \sHom(\pi_1^* F, \pi_2^* T_{-\epsilon}(G))$};
\node at (5,1.7) {$\Delta^* \sHom (\pi_1^* F, \pi_2^* G)$};
\node at (0,0) {$\sHom(F,T_{-\epsilon} (G))$};
\node at (5,0) {$\sHom(F,G)$};

\draw [->, thick] (2.3,1.7) -- (3.2,1.7) node [midway, above] {$ $};
\draw [->, thick] (1.6,0) -- (3.9,0) node [midway, above] {$c$};

\draw [->, thick] (0,1.4) -- (0,0.3) node [midway, left] {\rotatebox{90}{$\sim$}}; 
\draw [->, thick] (5,1.4) -- (5,0.3) node [midway, right] {$ $};

\end{tikzpicture}
$$ 
Apply Proposition \ref{prop:properties_of_continuation_maps}~(4) to $(\Delta \times \id_{\RR})^* \sHom( \pi_{1,\RR}^* F, \pi_{2,\RR}^* K(T) \circ G)$, we obtain that
$$\Delta^* \sHom (\pi_1^* F, \pi_2^* G) \xleftarrow{\sim} \clmi{-t \rightarrow 0^-} \, \Delta^* \sHom(\pi_1^* F, \pi_2^* T_{-t}(G)) \xrightarrow{\sim}  \clmi{-t \rightarrow 0^-} \, \sHom(F,T_{-t}(G) ).$$
Since $\supp(F)$ and $\supp(G)$ are compact, $p_* = p_!$ is colimit preserving, and thus we conclude that $\operatorname{colim}_{-t \rightarrow 0^-} \Hom(F,T_{-t}(G)) \xrightarrow{\sim} p_*\Delta^* \sHom (\pi_1^* F, \pi_2^* G))$ is an isomorphism.
The same argument as in the positive case then implies that the colimit diagram is constant and thus the inclusion 
$$\Hom(F,T_{-\epsilon}(G)) \rightarrow \clmi{-t \rightarrow 0^-} \Hom(F,T_{-t}(G))$$
is an isomorphism for $-c(\Lambda) < -\epsilon < 0$.

Finally, we notice that the diagram commute in the statement commute because it is a composition of the following two commutative diagram:

\[
\begin{tikzpicture}
\node at (0,1.7) {$p_*\Delta^*\mathscr{H}om(\pi_1^*{F}, \pi_2^*{G}))$};
\node at (6,1.7) {$\Hom({F, G})$};
\node at (0,0) {$\Hom({F}, T_{-\epsilon}({G}))$};
\node at (6,0) {$\Hom({F}, T_{\epsilon}({G}))$};

\draw [->, thick] (2.1,1.7) -- (4.9,1.7) node [midway, above] {$ $};
\draw [->, thick] (1.5,0) -- (4.5,0) node [midway, above] {$ $};

\draw [->, thick] (0,1.4) -- (0,0.3) node [midway, left] {\rotatebox{90}{$\sim$}}; 
\draw [->, thick] (6,1.4) -- (6,0.3) node [midway, right] {\rotatebox{90}{$\sim$}};

\draw [->, thick] (1.6,0.3) -- (5,1.3) node [midway, above] {$ $};

\node[scale=1.5] at (1.6,1) {$\circlearrowleft$};
\node[scale=1.5] at (4.6,0.7) {$\circlearrowleft$};

\end{tikzpicture} \qedhere
\]
\end{proof}

\begin{remark}
    The identity $\Hom({F, G}) \simeq \Hom({F}, T_{\epsilon}({G}))$ is often referred to as the perturbation trick, and has in fact appeared in previous works of Guillermou \cite[Corollary 16.6]{Gui} for the special case of vertical translation on $M \times \RR$, and Zhou for arbitrary Reeb flows \cite{Zhou}.
    The proof here follows \cite[Proposition 3.18]{Kuo-wrapped-sheaves} of the first author. 
\end{remark}

\subsection{Sabloff-Serre duality}\label{sec:serre} 
    In this section, we illustrate an additional property that arises from the Sato-Sabloff fiber sequence and prove a Sabloff-Serr{e} duality that
    $$\Hom(F, T_{-\epsilon}(G) \otimes \omega_M) = \Hom(F, T_\epsilon(G))^\vee.$$
    Such duality between a positive Reeb pushoff and a negative Reeb pushoff has been understood in symplectic geometry in a number of works. In Legendrian contact homology, this is known as the Sabloff duality \cite{Sabduality,EESduality}, and in Fukaya-Seidel categories, this is known as the Poincar\'{e}-Lefschetz duality proved by Seidel \cite{SeidelFukI}.
    
    In the previous work of the second author \cite{LiEstimate}, we proved a version of Sabloff duality that for $\Lambda \subseteq J^1(N) \cong S^{*}_{\tau > 0}(N \times \bR)$ and $F, G \in \Sh^b_\Lambda(N \times \bR)$ with compact supports when $N$ is orientable. Here we prove a general version of it.

Recall that we have assumed throughtout the paper that $\cV$ is a rigid symmetric monoidal category. We will explicitly use the following consequence from the rigidity assumption on $\cV$ in the computation.

\begin{lemma}[{\cite[Proposition 4.9]{Hoyois-Scherotzke-Sibilla}}]
Assume $\cV_0$ is a rigid symmetric monoidal category.
Then there is a canonical equivalence of symmetric monoidal $\infty$-categories
$$\cV_0 \rightarrow \cV_0^{op}, \, X \mapsto X^\vee \coloneqq \Hom(X,1_{\cV}).$$
In particular, $(X^\vee)^\vee = X$.
\end{lemma}

If $L$ is invertible, then $\sHom(F,G \otimes L) = \sHom(F,G) \otimes L$. Thus, when $M$ is a manifold, the Verdier duality $\VD{M}$ differs from the naive duality $\ND{M}(F) \coloneqq \sHom(F,1_M)$ by tensing $\omega_M^{-1}$. The following proposition is the main result in this section, which generalizes the Sabloff duality in \cite{LiEstimate} to arbitrary manifolds.

\begin{proposition}[Sabloff-Serr{e} duality]\label{prop:sab-serre}
    Let $\Lambda, \Lambda' \subseteq S^*M$ be compact subanalytic Legendrians. Let $T_t: S^*M \to S^*M$ be a non-negative contact flow such that 
    $$\Lambda \cap T_\epsilon(\Lambda') = \varnothing$$
    for any small $\epsilon \neq 0$. Then for $F \in \Sh_{\Lambda}(M)$ and $G \in \Sh_{\Lambda'}(M)$ such that $\mathrm{supp}(F) \cap \mathrm{supp}(G)$ is compact in $M$ and $F$ has perfect stalks,
    $$\Hom({F}, T_{-\epsilon}({G}) \otimes \omega_M) \simeq p_!(\VD{M}(F) \otimes G) \simeq \Hom({G, F})^\vee.$$
    In particular, when $M$ is oriented, $\Hom({F}, T_{-\epsilon}({G}))[-n] \simeq \Hom({G, F})^\vee.$
\end{proposition}
\begin{proof} 
    By Proposition \ref{prop:hom_w_pm} and Remark \ref{rem:DFotimesG},
    \begin{align*}
        \Hom(F,T_{-\epsilon}(G) \otimes \omega_M) &= p_* \left( \Delta^* \sHom(\pi_1^* F, \pi_2^* (G \otimes \omega_M)) \right) \\
        &= p_* ( \ND{M}(F) \otimes G \otimes \omega_M) = p_* (\VD{M}(F) \otimes G).
    \end{align*}
    The compact support assumption then implies that 
    \begin{align*}
        \Hom(F,T_{-\epsilon}(G) \otimes \omega_M)^\vee &= \Hom(p_! ( \VD{M}(F) \otimes G), 1_\cV) = \Hom( \VD{M}(F) \otimes G, \omega_M) \\
        &= \Hom\left(G, \VD{M} \circ \VD{M}(F) \right) = \Hom(G, F). \qedhere
    \end{align*}
\end{proof}

    In Section \ref{sec:serre-proper}, we will see that the above proposition plays a key role in the result regarding Serr{e} functors. Actually, one may have noticed that by Theorem \ref{w=ad}, we have shown
    $$\Hom(F, S_\Lambda^-(G) \otimes \omega_M) = \Hom(G, F)^\vee.$$
    However, we do not know whether $S_\Lambda^-$ sends $\Sh^b_\Lambda(M)$ to $\Sh^b_\Lambda(M)$ (in fact, in general it does not; see Section \ref{sec:example}). This issue will be addressed in Section \ref{sec:serre-proper}.

\subsection{Doubling from Sato-Sabloff sequence}\label{sec:doubling-local}
    Let us construct the doubling functor in this section. We will realize the doubling construction by using one single Reeb flow and hence write down the doubling functor explicitly on each local chart. 
    
    Consider ${F} \in \msh_\Lambda(\Lambda)$. First, we use the following lemma to find local representatives of $F$ by sheaves. The following result has previously been obtained by Guillermou \cite{Gui} by applying the refined microlocal cut-off lemma \cite{KS}*{Proposition 6.1.4}.

\begin{lemma}[Guillermou \cite{Gui}*{Lemma 6.7} or \cite{Guisurvey}*{Lemma 10.2.5}]\label{lem:refine-cutoff-0}
    Let $\Lambda \subseteq S^{*}M$ be a locally closed subanalytic Legendrian such that $\pi|_\Lambda: \Lambda \rightarrow M$ is finite. Then for $(x, \xi) \in \Lambda$, there is a neighbourhood $U$ of $x \in M$ and ${F}_U \in \Sh_{\Lambda \cap S^{*}U}(U)$ such that $m_{\Lambda \cap S^{*}U}({F}_U) = {F} \in \msh_\Lambda(S^*U)$.
\end{lemma}

    By the lemma, we consider an open covering $\mathscr{U} = \{U_\alpha\}_{\alpha \subseteq I}$ of $M$ and ${F}_{\alpha} \in \Sh_{\Lambda}(U_\alpha)$ such that
    $$m_{\Lambda \cap \Omega_\alpha}({F}_{\alpha}) = {F}|_{\Lambda \cap \Omega_\alpha} \in \msh_\Lambda(\Lambda \cap S^{*}U_\alpha).$$
    Here, we will always use $\alpha \subseteq I$ for the multi-index subset where $U_{\alpha_1\cdots\alpha_k} \coloneqq U_{\alpha_1} \cap \dots \cap U_{\alpha_k}$.

    Using Theorem \ref{thm:sato-sab}, we can find that before further applying wrapping by $\wrap_\Lambda^\pm$, the left and right adjoints of the microlocalization $m_\Lambda$ can be characterized by the difference between small positive and negative wrappings $T_{-\epsilon}(F_\alpha) \to T_\epsilon(F_\alpha)$.
    Therefore, we would like to construct a sheaf $w_\Lambda(F)$ which locally on an open subset $U_\alpha$ will be of the form $$w_\Lambda({F})_{U_\alpha} = \mathrm{Cofib}(T_{-\epsilon}({F}_\alpha) \rightarrow T_\epsilon({F}_\alpha)).$$
    However there is some technical issue that, under the Reeb flow $T_t$ on $S^*M$, it is not even true that $T_{\pm \epsilon}(\Lambda \cap S^*U_\alpha) \subseteq S^*U_\alpha$, and hence the above formula does not seem to be meaningful even at the first place.

\begin{figure}[h!]
    \centering
    \includegraphics[width=\textwidth]{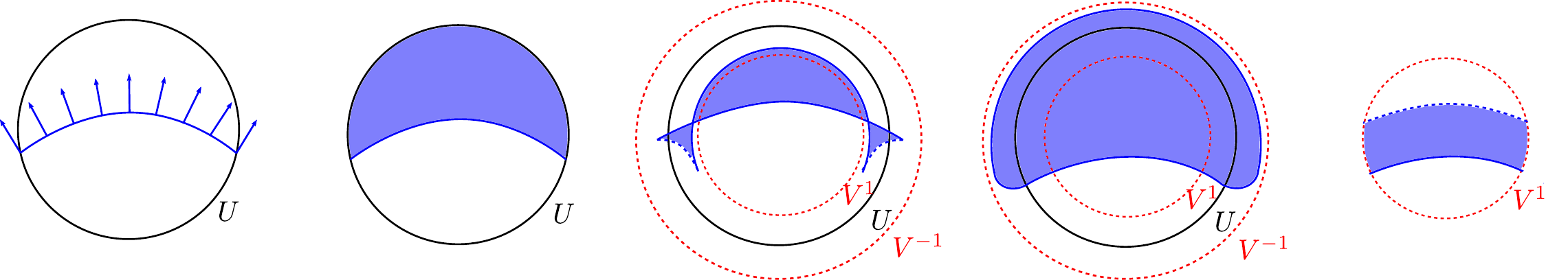}
    \caption{We consider the open subset $U$ (in black), the Legendrian $\Lambda$ (in blue), and the Reeb flow being the geodesic flow, where $T_{\pm \epsilon}(\Lambda \cap S^*U_\alpha) \not\subseteq S^*U_\alpha$. Let $F_U$ be the sheaf as in the 2nd figure. Then $T_{\pm \epsilon}(j_{U *}F_U)$ are illustrated in the 3rd and 4th figure. The supports of the sheaves are in $V^{-1}$, while the singular support coming from $T_{\pm \epsilon}(N^*_{in/out}U)$ are outside $V^1$. Finally, $w_\Lambda(F)_V$ is shown in the 5th figure.}\label{fig:doubling-cutoff}
\end{figure}

    Our solution to the above problem is as follows. We will need to push forward the sheaves on $U$ to sheaves on $M$ so as to apply the Reeb flow on the ambient manifold $S^*M$. The singular support of the resulting sheaves consists of both the Reeb pushoff of the Legendrian $T_{\pm \epsilon}(\Lambda \cap S^*U_\alpha)$ and the Reeb pushoff of the unit conormal bundle of the boundary $T_{\pm \epsilon}(N^*_{in/out} U_\alpha)$.
    
    To block off the effect coming from the second part (which may come into $S^*U_\alpha$ under the Reeb flow), we will need to restrict the sheaf to a smaller neighbourhood $V_\alpha \subseteq U_\alpha$. This accounts for the following complicated definition.

\begin{definition}\label{def:goodcover}
    Let $\mathscr{U} = \{U_\alpha\}_{\alpha \subseteq I}$ be an open covering of $M$, $\Lambda \subseteq S^{*}M$ a closed subanalytic Legendrian and $\Lambda'$ a generic Hamiltonian perturbation of $\Lambda$. Then $\mathscr{U}$ is a good covering with respect to $\Lambda \subset S^{*}M$ if
    \begin{enumerate}
      \item $U_{\alpha}$ are contractible;
      \item $\partial U_\alpha$ are piecewise smooth with transverse intersections;
      \item $N^*_{out}{U_{\alpha}} \cap \Lambda' = \varnothing$.
    \end{enumerate}
    Given a good covering $\mathscr{U}$ with respect to $\Lambda$, a good family of refinement with respect to $\Lambda$ is $\mathscr{V}^t = \{V_\alpha^t\}_{\alpha \subseteq I}$ where $t \in [-1, 1]$ is a family of open covering with $\mathscr{V}^0 = \mathscr{U}$ such that
    \begin{enumerate}
      \item $V_\alpha^{t'} \subseteq V_\alpha^t$ for any $-1 \leq t \leq t' \leq 1$;
      \item $V_{\alpha}^t$ are contractible for any $-1 \leq t \leq 1$;
      \item $\partial V_\alpha^t$ are piecewise smooth with transverse intersections for any $-1 \leq t \leq 1$;
      \item there exists some Riemannian metric $g$ on $M$ and some $\epsilon > 0$ so that
      $$\mathrm{dist}_g(\partial U_\alpha, \partial V_\alpha^{\pm 1}) \geq \epsilon.$$
      \item $N^*_{out}{V_{\alpha}^t} \cap \Lambda' = \varnothing\,(\alpha \subseteq I)$ for any $-1 \leq t \leq 1$;
    \end{enumerate}
    For simplicity, we will call $\mathscr{V} = \mathscr{V}^1$ a good refinement of $\mathscr{U}$.
\end{definition}
\begin{remark}
    This definition is in the same spirit as \cite{Guisurvey}*{Definition 11.4.1}. The reason that we also need to choose a family of good refinement instead of only a good covering is that here we need to consider an arbitrary Reeb flow, while Guillermou considered only the vertical translation in $S^*_{\tau>0}(M \times \mathbb{R})$ and chose only open subsets of the form $U_i \times I_i \subset M \times \mathbb{R}$. Here we are adding the contractibility assumption simply the discussion when constructing good refinements.
\end{remark}

\begin{lemma}
    For any open covering $\mathscr{U}_0$ on $M$ and a closed subanalytic Legendrian $\Lambda \subseteq S^{*}M$, there exists a refinement $\mathscr{U}$ with respect to $\Lambda$ such that $\mathscr{U}$ admits a good family of refinements $\mathscr{V}^t$.
\end{lemma}
\begin{proof}
    The existence of a refinement $\mathscr{U}$ of $\mathscr{U}_0$ satisfying (1)~\&~(2) follows from convex neighbourhood theorem in Riemannian geometry. The reason that
    $$N^*_{out}{U_{\alpha_1\cdots\alpha_k}} \cap \Lambda' = \varnothing, \,\,\alpha_1, \cdots, \alpha_k \in I$$
    for a generic perturbation $\Lambda'$ of $\Lambda$ is because $\bigcup_{\alpha_1, \cdots, \alpha_k \in I}N^*_{out}{U_{\alpha_1\cdots\alpha_k}}$ is also a subanalytic Legendrian and hence the sum of dimensions is less than $2\dim M - 1$.

    The existence of a family of refinement $\mathscr{V}^t$ of $\mathscr{U}$ satisfying (1)--(4) is again convex neighbourhood theorem. The reason that
    $$N^*_{out}{V_{\alpha_1\cdots\alpha_k}^t} \cap \Lambda' = \varnothing, \,\,\alpha_1, \cdots, \alpha_k \in I, \, -1 \leq t \leq 1$$
    for a generic perturbation $\Lambda'$ is because we can choose $\mathscr{V}^t$ such that $\bigcup_{\alpha_1, \cdots, \alpha_k \in I}N^*_{out}{V_{\alpha_1\cdots\alpha_k}^t}$ are small perturbations of $\bigcup_{\alpha_1, \cdots, \alpha_k \in I}N^*_{out}{U_{\alpha_1\cdots\alpha_k}}$. This completes the proof.
\end{proof}
\begin{remark}
    From now on, without loss of generality (by Theorem \ref{cor:GKS} and Theorem \ref{cor:cont-trans}) we will always assume that $\Lambda = \Lambda'$. In other words, we assume that $N^*_{out}{U_{\alpha_1\cdots\alpha_k}} \cap \Lambda = \varnothing$ and  $N^*_{out}{V_{\alpha_1\cdots\alpha_k}^t} \cap \Lambda = \varnothing$ for any $-1 \leq t \leq 1$.
\end{remark}

    The main microlocal properties of families of good refinements of coverings that we are going to use are given as follows.

\begin{lemma}\label{lem:lower*=lower!}
    Let $\mathscr{U}$ be a good covering of $M$ with a good family of refinements $\mathscr{V}^t$ with respect to $\Lambda$. Write $j_\alpha: U_\alpha \hookrightarrow M$. Then given ${F} \in \Sh(U_\alpha)$, for $\epsilon > 0$ sufficiently small, we have
    $$T_{\pm \epsilon}(j_{\alpha !}{F})|_{V_\alpha} \xrightarrow{\sim} T_{\pm \epsilon}(j_{\alpha *}{F})|_{V_\alpha}.$$
\end{lemma}
\begin{proof}
    Consider the mapping cone
    $$\mathrm{Cofib}(T_{\pm \epsilon}(j_{\alpha !}{F}) \rightarrow T_{\pm \epsilon}(j_{\alpha *}{F})) \simeq T_{\pm \epsilon} \mathrm{Cofib}(j_{\alpha !}{F} \rightarrow j_{\alpha *}{F}).$$
    Then by Proposition \ref{prop:noncharacteristic-ms-es}, $\ms^\infty(\mathrm{Cofib}(j_{\alpha !}{F} \rightarrow j_{\alpha *}{F})) \subseteq S^*M|_{\partial U_\alpha}$, since $j_{\alpha !}{F} \simeq j_{\alpha *}{F}$ in the interior of $T^*U_\alpha$. By Corollary \ref{cor:GKS} we know that
    $$\ms^\infty(T_{\pm \epsilon}\mathrm{Cofib}(j_{\alpha !}{F} \rightarrow j_{\alpha *}{F})) \subseteq T_{\pm \epsilon}(S^*M|_{\partial U_\alpha}).$$
    Since $T_{\pm \epsilon}(S^*M|_{\partial U_\alpha}) \cap S^*V_\alpha = \varnothing$, we know that in particular for $\epsilon > 0$ sufficiently small,
    $$\ms^\infty(T_{\pm \epsilon}\mathrm{Cofib}(j_{\alpha !}{F} \rightarrow j_{\alpha *}{F})) \cap S^*V_\alpha = \varnothing.$$
    Moreover, since $j_{\alpha !}{F} \simeq j_{\alpha *}{F}$ in $T^*V_\alpha$, by the above singular support estimate, we can conclude that $T_{\pm \epsilon}(j_{\alpha !}{F}) \simeq T_{\pm \epsilon}(j_{\alpha *}{F})$ in $T^*V_\alpha$, which proves the isomorphism as claimed.
\end{proof}

\begin{lemma}\label{lem:hom-cutoff}
    Let $\mathscr{U}$ be a good covering of $M$ with a good family of refinements $\mathscr{V}^t$ with respect to $\Lambda$. Write $j_\alpha: U_\alpha \hookrightarrow M$. Then given ${F}, {G} \in \Sh_{\Lambda \cap S^{*}U_\alpha}(U_\alpha)$, for $\epsilon > 0$ sufficiently small, we have
    $$\Hom(T_{\pm \epsilon}(j_{\alpha !}{F}), T_{\pm \epsilon}(j_{\alpha *}{G})) \xrightarrow{\sim} \Hom(T_{\pm \epsilon}(j_{\alpha !}{F})|_{V_\alpha}, T_{\pm \epsilon}(j_{\alpha *}{G})|_{V_\alpha}).$$
\end{lemma}
\begin{proof}
    Consider the sheaf $\mathscr{H}om(T_{\pm \epsilon}(j_{\alpha !}{F}), T_{\pm \epsilon}(j_{\alpha *}{G}))$. We know by Proposition \ref{prop:mses}~(8) that
    $$\ms(\mathscr{H}om(T_{\pm \epsilon}(j_{\alpha !}{F}), T_{\pm \epsilon}(j_{\alpha *}{G}))) \subseteq T_{\pm \epsilon}(-\ms(j_{\alpha !}{F})) + T_{\pm \epsilon}(\ms(j_{\alpha *}{F})).$$
    Consider the family of good refinements $V_\alpha^t$, $t \in [-1, 1]$, where $V_\alpha^0 = U_\alpha$ and $V_\alpha^1 = V_\alpha$. We know by Proposition \ref{prop:noncharacteristic-ms-es} that
    $$\ms^\infty(j_{\alpha !}{F}) \subseteq \Lambda \,\widehat +\, N^*_{out}{U_\alpha} , \;\; \ms^\infty(j_{\alpha *}{F}) \subseteq \Lambda \,\widehat +\, N^*_{in}{U_\alpha}.$$
    Therefore by the condition on the family of good refinements $V_\alpha^t$
    $$N^{*}_{out}{V_\alpha^t} \cap (-\ms^\infty(j_{\alpha !}{F})) = N^{*}_{out}{V_\alpha^t} \cap \ms^\infty(j_{\alpha *}{F}) = \varnothing.$$
    Indeed, for $(x, \xi) \in N^{*}_{out}{V_\alpha^t}$, $(x, \xi) \in N^*_{in}{U_\alpha} \,\widehat +\, (\pm \Lambda)$ only if there exists $(x_n, -\xi_n) \in N^*_{in}{U_\alpha}$, $(y_n, \eta_n) \in \pm \Lambda$ such that
    $x_n, y_n \rightarrow x, \, -\xi_n + \eta_n \rightarrow \xi, \, |x_n - y_n||\xi_n| \rightarrow 0.$
    However, the fact that $\Lambda \cap N^*_{in/out}V_\alpha^t = \varnothing$ immediately implies that $\eta_n \rightarrow 0$. This forces $-\xi_n \rightarrow \xi$, which implies $\xi = 0$, so in the unit cotangent bundle the intersections are empty.
    Hence for $\epsilon > 0$ sufficiently small, we have $\mathrm{supp}(T_{\pm \epsilon}(j_{\alpha !}{F}))$, $\mathrm{supp}(T_{\pm \epsilon}(j_{\alpha *}{G})) \subset V_\alpha^{-1}$, and
    $$N^{*}_{out}{V_\alpha^t} \cap T_{\pm \epsilon}(-\ms^\infty(j_{\alpha !}{F})) = N^{*}_{out}{V_\alpha^t} \cap T_{\pm \epsilon}(\ms^\infty(j_{\alpha *}{F})) = \varnothing.$$
    Therefore, by microlocal Morse lemma Proposition \ref{prop:morselemma}, restricting from $V_\alpha^{-1}$ to $V_\alpha^1$ we have
    $$\Gamma(M, \mathscr{H}om(T_{\pm \epsilon}(j_{\alpha !}{F}), T_{\pm \epsilon}(j_{\alpha *}{G}))) \simeq \Gamma(V_\alpha, \mathscr{H}om(T_{\pm \epsilon}(j_{\alpha !}{F}), T_{\pm \epsilon}(j_{\alpha *}{G}))),$$
    which shows the isomorphism.
\end{proof}

\begin{lemma}\label{lem:muhom-cutoff}
    Let $\mathscr{U}$ be a good covering of $M$ with a good family of refinements $\mathscr{V}^t$ with respect to $\Lambda$. Write $j_\alpha: U_\alpha \hookrightarrow M$. Then given ${F}, {G} \in \Sh_{\Lambda \cap S^{*}U_\alpha}(U_\alpha)$, for $\epsilon > 0$ sufficiently small, we have
    $$\Gamma(S^{*}V_\alpha^{-1}, \mhom(j_{\alpha !}{F}, j_{\alpha *}{G})) \xrightarrow{\sim} \Gamma(S^{*}V_\alpha^1, \mhom({F}, {G})).$$
\end{lemma}
\begin{proof}
    Consider the good family of refinements $V_\alpha^t$ where $V_\alpha^0 = U_\alpha$ and $V_\alpha^1 = V_\alpha$. Note that by Proposition \ref{prop:noncharacteristic-ms-es} we have
    $$\ms^\infty(j_{\alpha !}{F}) \subset \Lambda \,\widehat +\, N^*_{out}{U_\alpha}, \;\; \ms^\infty(j_{\alpha *}{F}) \subset \Lambda \,\widehat +\, N^*_{in}{U_\alpha},$$
    and then by Proposition \ref{prop:ss-muhom} \cite{KS}*{Corollary 6.4.3} one can deduce that
    $$\ms^\infty(\mhom(j_{\alpha !}{F}, j_{\alpha *}{G})) \subset C(\Lambda \,\widehat +\, N^*_{out}{U_\alpha}, \Lambda \,\widehat +\, N^*_{in}{U_\alpha}) \subseteq S^*(S^{*}M).$$
    Then by Proposition \ref{prop:mses}~(4) and Remark \ref{rem:ss-muhom} we know that
    $$\ms^\infty(\dot{\pi}_*\mhom(j_{\alpha !}{F}, j_{\alpha *}{G})) \subset -(\Lambda \,\widehat +\, N^*_{out}{U_\alpha}) \,\widehat+_\infty\, (\Lambda \,\widehat +\, N^*_{in}{U_\alpha}).$$
    We know $(x, \xi) \in (-\Lambda \,\widehat +\, N^*_{in}{U_\alpha}) \,\widehat+\, (\Lambda \,\widehat +\, N^*_{in}{U_\alpha}))$ if and only if there exists $(x_n, -\xi_n) \in -\Lambda \,\widehat +\, N^*_{in}{U_\alpha}$ and $(y_n, \eta_n) \in \Lambda \,\widehat +\, N^*_{in}{U_\alpha}$ such that
    $x_n, y_n \rightarrow x, \; -\xi_n + \eta_n \rightarrow \xi, \; |x_n - y_n| |\xi_n| \rightarrow 0$. When we consider $x \in \partial U_\alpha$, we know that $(x, \xi) \in \pm \Lambda \,\widehat+\, N^*_{in}U_\alpha$ and hence $(x, \xi) \notin  N^*_{out}V_\alpha^t$. When we consider $x \in U_\alpha$, we know that $(x, \xi) \in \Lambda$ and hence $(x, \xi) \notin  N^*_{out}V_\alpha^t$ because $\pm \Lambda \cap N^*_{out}V_\alpha^t = \varnothing$. These two facts imply that in the unit cotangent bundle the following intersection is empty
    $$N^*_{out}V_\alpha^t \cap \ms^\infty(\dot{\pi}_*\mhom(j_{\alpha !}{F}, j_{\alpha *}{G})) = \varnothing.$$
    Therefore, by microlocal Morse lemma Proposition \ref{prop:morselemma}, restricting from $V_\alpha^{-1}$ to $V_\alpha^1$, we can show that the isomorphism holds.
\end{proof}

    Given ${F} \in \msh_\Lambda(\Lambda)$, by Lemma \ref{lem:refine-cutoff} there exists an open covering $\mathscr{U} = \{U_\alpha\}_{\alpha \subseteq I}$ and a collection of sheaves $\{{F}_\alpha\}_{\alpha \subseteq I}$ where ${F}_\alpha \in \Sh_\Lambda(U_\alpha)$ such that ${F}|_{\Lambda \cap S^{*}U_\alpha} = m_\Lambda({F}_\alpha)$. Write $j_\alpha: U_\alpha \hookrightarrow M$. Choose a good family of refinements $\mathscr{V}^t$ of $\mathscr{U}$ for $t \in [-1, 1]$.

\begin{definition}
    For ${F} \in \msh_\Lambda(\Lambda)$, an open covering $\mathscr{U}$, a good family of refinements $\mathscr{V}^t$, and a collection of sheaves $\{{F}_\alpha\}_{\alpha \subseteq I}$ where ${F}_\alpha \in \Sh_\Lambda(U_\alpha)$, we define by Lemma \ref{lem:lower*=lower!}
    \[\begin{split}
    w_\Lambda({F})_{V_\alpha} & = \mathrm{Cofib}(T_{-\epsilon}(j_{\alpha *}{F}_\alpha)|_{V_\alpha} \rightarrow T_\epsilon(j_{\alpha *}{F}_\alpha)|_{V_\alpha}) \\
    & = \mathrm{Cofib}(T_{-\epsilon}(j_{\alpha !}{F}_\alpha)|_{V_\alpha} \rightarrow T_\epsilon(j_{\alpha !}{F}_\alpha)|_{V_\alpha}).
    \end{split}\]    
\end{definition}

\begin{proposition}\label{prop:doubling_ff}
    Let $\Lambda \subseteq S^{*}M$ be a compact subanalytic Legendrian, $\mathscr{U}$ a good open covering and $\mathscr{V}^t$ a good family of refinements with respect to $\Lambda$. Then for $\epsilon > 0$ sufficiently small, there is a natural isomorphism
    $$\Hom(w_\Lambda({F})_{V_\alpha}, w_\Lambda({G})_{V_\alpha}) \xrightarrow{\sim} \Gamma(S^{*}V_\alpha, \mhom({F}_{V_\alpha}, {G}_{V_\alpha})).$$
\end{proposition}
\begin{proof}
    Writing down the definition of $w_\Lambda({F})_{V_\alpha}$ and $w_\Lambda({G})_{V_\alpha}$, we have
    \[\begin{split}
    \Hom(w_\Lambda({F})&{}_{V_\alpha}, w_\Lambda({G})_{V_\alpha}) \\
    \simeq \mathrm{Cofib} \big(&\Hom(\mathrm{Cofib}(T_{-\epsilon}(j_{\alpha !}{F}_\alpha)|_{V_\alpha} \rightarrow T_\epsilon(j_{\alpha !}{F}_\alpha)|_{V_\alpha}), T_{-\epsilon}(j_{\alpha *}{G}_\alpha)|_{V_\alpha}) \\
    & \rightarrow \Hom(\mathrm{Cofib}(T_{-\epsilon}(j_{\alpha !}{F}_\alpha)|_{V_\alpha} \rightarrow T_\epsilon(j_{\alpha !}{F}_\alpha)|_{V_\alpha}), T_\epsilon(j_{\alpha *}{G}_\alpha)|_{V_\alpha})\big).
    \end{split}\]
    For the first term, we claim that
    \[\begin{split}
    \Hom(&\mathrm{Cofib}(T_{-\epsilon}(j_{\alpha !}{F}_\alpha)|_{V_\alpha} \rightarrow T_\epsilon(j_{\alpha !}{F}_\alpha)|_{V_\alpha}), T_{-\epsilon}(j_{\alpha *}{G}_\alpha)|_{V_\alpha}) \\
    &\simeq \Gamma(S^{*}V_\alpha, \mhom({F}_{V_\alpha}, {G}_{V_\alpha})).
    \end{split}\]
    To prove this, we apply the Sato-Sabloff fiber sequence Theorem \ref{thm:sato-sab} and get
    \[\begin{split}
    \Hom(T_\epsilon(j_{\alpha !}{F}_\alpha), T_{-\epsilon}(j_{\alpha *}{G}_\alpha)) & \rightarrow \Hom(T_{-\epsilon}(j_{\alpha !}{F}_\alpha), T_{-\epsilon}(j_{\alpha *}{G}_\alpha)) \\
    & \rightarrow \Gamma(T^{*,\infty}V_\alpha^{-1}, \mhom(j_{\alpha !}{F}_\alpha, j_{\alpha *}{G}_\alpha)).
    \end{split}\]
    Given the non-characteristic assumption on $\mathscr{V}^t$ with respect to $\Lambda$, we can apply Lemma \ref{lem:hom-cutoff} to restrict the corresponding $\mathscr{H}om(-, -)$ sheaves from $V_\alpha^{-1}$ to $V_\alpha^1$, and get quasi-isomorphisms
    \[\begin{split}
    \Hom(T_\epsilon(j_{\alpha !}{F}_\alpha), T_{-\epsilon}(j_{\alpha *}{G}_\alpha)) &\simeq \Hom(T_\epsilon(j_{\alpha !}{F}_\alpha)|_{V_\alpha}, T_{-\epsilon}(j_{\alpha *}{G}_\alpha)|_{V_\alpha}), \\
    \Hom(T_\epsilon(j_{\alpha !}{F}_\alpha), T_{\epsilon}(j_{\alpha *}{G}_\alpha)) &\simeq \Hom(T_\epsilon(j_{\alpha !}{F}_\alpha)|_{V_\alpha}, T_{\epsilon}(j_{\alpha *}{G}_\alpha)|_{V_\alpha}).
    \end{split}\]
    On the other hand, we can also apply Lemma \ref{lem:muhom-cutoff} to restrict the corresponding $\mhom(-, -)$ sheaves from $S^{*}V_\alpha^{-1}$ to $S^{*}V_\alpha^1$, and get
    $$\Gamma(S^{*}V_\alpha^{-1}, \mhom(j_{\alpha !}{F}_\alpha, j_{\alpha *}{G}_\alpha)) \simeq \Gamma(S^{*}V_\alpha^1, \mhom({F}_\alpha, {G}_\alpha)).$$
    Since the restriction maps commute with all the maps in the Sato-Sabloff fiber sequence, this proves our first claim. For the second term, we claim that there is a quasi-isomorphism
    $$\Hom(\mathrm{Cofib}(T_{-\epsilon}(j_{\alpha !}{F}_\alpha)|_{V_\alpha} \rightarrow T_\epsilon(j_{\alpha !}{F}_\alpha)|_{V_\alpha}), T_\epsilon(j_{\alpha *}{G}_\alpha)|_{V_\alpha}) \simeq 0.$$
    Indeed by Proposition \ref{prop:hom_w_pm} we know that
    $$\Hom(T_\epsilon(j_{\alpha !}{F}_\alpha), T_\epsilon(j_{\alpha *}{G}_\alpha)) \simeq \Hom(T_{-\epsilon}(j_{\alpha !}{F}_\alpha), T_\epsilon(j_{\alpha *}{G}_\alpha)) \simeq \Hom(j_{\alpha !}{F}_\alpha, j_{\alpha *}{G}_\alpha).$$
    and the isomorphism is witnessed by the precomposition with the canonical continuation map $T_{-\epsilon}(j_{\alpha *}{F}_\alpha) \rightarrow T_\epsilon(j_{\alpha *}{F}_\alpha)$. Again by the non-characteristic assumption on $\mathscr{U}, \mathscr{V}$ with respect to $\Lambda$, we can apply Lemma \ref{lem:hom-cutoff} to restrict the corresponding $\mathscr{H}om(-, -)$ sheaves from $V_\alpha^{-1}$ to $V_\alpha^1$, and the quasi-isomorphisms still hold. This proves our second claim.
\end{proof}


    By Proposition \ref{prop:doubling_ff}, we can conclude that indeed the family of sheaves $\{w_\Lambda({F})_\alpha\}_{\alpha \subseteq I}$ on the refined open cover of $\mathscr{V}$ can be glued to a global object.

\begin{proof}[Proof of Theorem \ref{thm:doubling}]
    Consider the functor $w_\Lambda: \msh_\Lambda(\Lambda) \rightarrow \Sh_{T_{-\epsilon}(\Lambda) \cup T_\epsilon(\Lambda)}(M)$. We apply Proposition \ref{prop:doubling_ff} to show that $w_\Lambda$ is fully faithful, i.e.
    $$\Hom(w_\Lambda({F}), w_\Lambda({G})) \xrightarrow{\sim} \Gamma(S^*M, \mhom({F}, {G})).$$

    First of all, since $\Lambda \subset S^{*}M$ is compact, we may choose assume that there are only finite open subsets $U_\alpha \in \mathscr{U}$ such that $\pi(\Lambda) \cap U_\alpha \neq \varnothing$. Hence there exists a uniform $\epsilon > 0$ sufficiently small such that Proposition \ref{prop:doubling_ff} hold for all $U_\alpha \in \mathscr{U}$. By Proposition \ref{prop:doubling_ff}, and the fact that these quasi-isomorphisms commute with restriction maps, we have the following diagram
    \[\xymatrix{
    \bigoplus_{\alpha \in I}\Hom(w_\Lambda({F})_{V_\alpha}, w_\Lambda({G})_{V_\alpha}) \ar[r] \ar[d]^{\rotatebox{90}{$\sim$}} & \bigoplus_{\alpha, \beta \in I}\Hom(w_\Lambda({F})_{V_{\alpha \beta}}, w_\Lambda({G})_{V_{\alpha \beta}}) \ar@<-.5ex>[r] \ar@<.5ex>[r] \ar[d]^{\rotatebox{90}{$\sim$}} & \cdots \\
    \bigoplus_{\alpha \in I}\Gamma(S^{*}V_\alpha, \mhom({F}, {G})) \ar[r] & \bigoplus_{\alpha, \beta \in I}\Gamma(S^{*}V_{\alpha \beta}, \mhom({F}, {G})) \ar@<-.5ex>[r] \ar@<.5ex>[r] & \cdots
    }\]
    Note that $\mathscr{V}$ is a good covering such that any finite intersection is contractible. Therefore, by taking the homotopy colimit of the above diagram, we get the quasi-isomorphism of global sections
    \[\begin{split}
    \Hom(w_\Lambda({F}), w_\Lambda({G})) & \xrightarrow{\sim} \lmi{\alpha \in I}\Hom(w_\Lambda({F})_{V_\alpha}, w_\Lambda({G})_{V_\alpha}) \\
    & \xrightarrow{\sim} \lmi{\alpha \in I}\, \Gamma(S^{*}V_\alpha, \mhom({F}, {G})) \xrightarrow{\sim} \Gamma(S^*M, \mhom({F}, {G})).
    \end{split}\]

    Finally, we need to check that the doubling functor can be defined for any $0 < \epsilon < c(\Lambda)/2$ where $c(\Lambda)$ is the length of the shortest Reeb chord on $\Lambda$. This is because when $0 < \epsilon < c(\Lambda)/2$, $T_{-\epsilon}(\Lambda) \cup T_\epsilon(\Lambda)$ are related by Hamiltonian isotopies supported away from $\Lambda$. We can choose $H: S^{*}M \rightarrow \mathbb{R}$ with compact support (since $\Lambda$ is compact) such that
    $$H|_\Lambda = 0, \;\; H|_{\bigcup_{\epsilon \in [c, c']}T_\epsilon(\Lambda)} = 1, \;\; H|_{\bigcup_{\epsilon \in [c, c']}T_{-\epsilon}(\Lambda)} = -1.$$
    Then the contact Hamiltonian flow is the integration of the corresponding compactly supported Hamiltonian vector field.
\end{proof}
\begin{remark}
    The condition that $\Lambda$ is compact plays an important role in the proof. First, it ensures that there exists a uniform $\epsilon > 0$ such that the doubling functor is locally defined among all $V_\alpha \in \mathscr{V}$. Secondly, it ensures that there exists a compactly supported Hamiltonian isotopy relating $T_{-\epsilon}(\Lambda) \cup T_\epsilon(\Lambda)$ for $0 < \epsilon < c(\Lambda)/2$ (otherwise, though we can similarly define the Hamiltonian function, it is unclear whether the Hamiltonian vector field is complete).
\end{remark}

    Moreover, we can in fact prove Theorem \ref{thm:doubling_ad} by almost the same computation as we did in proving full faithfulness using the following propositions.

\begin{proposition}\label{prop:doubling_rightad}
    Let $\Lambda \subseteq S^{*}M$ be a compact subanalytic Legendrian, $\mathscr{U}$ a good open covering and $\mathscr{V}^t$ a good family of refinements with respect to $\Lambda$. Then there is a natural isomorphism
    $$\Hom(F|_{V_\alpha}, w_\Lambda(G)_{V_\alpha}) \simeq \Gamma(S^*V_\alpha, \mhom(F_\alpha, G_\alpha)).$$
\end{proposition}
\begin{proof}
    Writing down the definition of $w_\Lambda({G})_{V_\alpha}$, we have
    \begin{align*}
    \Hom({F}|_{V_\alpha}, w_\Lambda({G})_{V_\alpha}) 
    \simeq \Hom\big({F}|_{V_\alpha}, \mathrm{Cofib}(T_{-\epsilon}(j_{\alpha *}{G}_\alpha)|_{V_\alpha} \rightarrow T_\epsilon(j_{\alpha *}{G}_\alpha)|_{V_\alpha})\big).
    \end{align*}
    Given the non-characteristic condition for the good family of refinements $\mathscr{V}^t$, by Lemma \ref{lem:hom-cutoff} we know that
    \[\begin{split}
    \Hom({F}|_{V_\alpha}, T_{-\epsilon}(j_{\alpha *}{G}_\alpha)|_{V_\alpha}) & \simeq \Hom(j_{\alpha !}{F}_{\alpha}, T_{-\epsilon}(j_{\alpha *}{G}_\alpha)),\\
    \Hom({F}|_{V_\alpha}, T_\epsilon(j_{\alpha *}{G}_\alpha)|_{V_\alpha}) & \simeq \Hom(j_{\alpha !}{F}_{\alpha}, T_\epsilon(j_{\alpha *}{G}_\alpha)).
    \end{split}\]
    In addition, it is easy to show that the restriction functors above commute with the canonical map $T_{-\epsilon}(j_{\alpha *}{G}_\alpha)|_{V_\alpha} \rightarrow T_\epsilon(j_{\alpha *}{G}_\alpha)|_{V_\alpha}$. Therefore, by the Sato-Sabloff exact triangle Theorem \ref{thm:sato-sab} we can conclude that
    \[\begin{split}
    \Hom&({F}|_{V_\alpha}, w_\Lambda({G})_{V_\alpha}) \\
    & \simeq \Hom\big(j_{\alpha !}{F}_{\alpha}, \mathrm{Cofib}(T_{-\epsilon}(j_{\alpha *}{G}_\alpha) \rightarrow  T_\epsilon(j_{\alpha *}{G}_\alpha))\big) \\
    & \simeq \Gamma(S^{*}U_\alpha, \mhom(j_{\alpha !}{F}_{\alpha}, j_{\alpha *}{G}_\alpha)) \simeq \Gamma(S^{*}V_\alpha, \mhom({F}_{\alpha}, {G}_\alpha)),
    \end{split}\]
    where the last inequality again follows from non-characteristic deformation in Lemma \ref{lem:muhom-cutoff}. Hence the proof is completed.
\end{proof}

    The following theorem follows from exactly the same argument and thus will be omitted.

\begin{proposition}\label{prop:doubling_leftad}
    Let $\Lambda \subseteq S^{*}M$ be a compact subanalytic Legendrian, $\mathscr{U}$ a good open covering and $\mathscr{V}^t$ a good family of refinements with respect to $\Lambda$. Then there is a natural isomorphism
    $$\Hom(w_\Lambda(F)_{V_\alpha}, G|_{V_\alpha}) \simeq \Gamma(S^*V_\alpha, \mhom(F_\alpha, G_\alpha)).$$
\end{proposition}

\begin{proof}[Proof of Theorem \ref{thm:doubling_ad}]
    Using Theorem \ref{w=ad}, it suffices to prove that
    $$\Hom(F, w_\Lambda(G)) \simeq \Gamma(S^*M, \mhom(F, G)), \; \Hom(w_\Lambda(F)[-1], G) \simeq \Gamma(S^*M, \mhom(F, G)).$$
    Then as in the proof of Theorem \ref{thm:doubling}, these isomorphisms follow from Proposition \ref{prop:doubling_rightad} and \ref{prop:doubling_leftad}.
\end{proof}

\begin{remark}\label{rem:doubling_ad}
    In fact, in the proof we have shown that  for ${F} \in \Sh_\Lambda(M), {G} \in \msh_\Lambda(M)$,
    $$\Hom({F}, w_\Lambda({G})) \simeq \Hom(T_\epsilon({F}), w_\Lambda({G})) \simeq \Gamma(S^*M, \mhom(m_\Lambda({F}), {G})).$$
    and respectively for ${F} \in \msh_\Lambda(M), {G} \in \Sh_\Lambda(M)$,
    $$\Hom(w_\Lambda({F})[-1], {G}) \simeq \Hom(w_\Lambda({F})[-1], T_{-\epsilon}({G})) \simeq \Gamma(S^*M, \mhom({F}, m_\Lambda({G}))),$$
    Note that this is also a direct corollary of Proposition \ref{prop:hom_w_pm}.
\end{remark}

\begin{remark}\label{rem:multi-component-double}
    Moreover, we remark that when we apply the doubling construction to a single connected component $\Lambda_i \subseteq \Lambda$, using the same argument, one can still show that 
    $$m_{\Lambda_i}^l = \mathfrak{W}_\Lambda^+ \circ w_{\Lambda_i}[-1], \;\; m_{\Lambda_i}^r = \mathfrak{W}_\Lambda^- \circ w_{\Lambda_i}.$$
    The reader may compare it with the discussion in Remark \ref{rem:multi-component-sato} and \ref{rem:multi-component-ad}.
\end{remark}

\begin{remark}
    Guillermou \cite{Gui} constructed a sheaf in $M \times \bR_t$ with singular support in a closed Legendrian with no Reeb chords $\Lambda \subseteq S^*_{\tau>0}(M \times \bR_t)$ by first constructing a doubling with singular support in $\Lambda \cup T_\epsilon(\Lambda)$ and then push one copy off via the Reeb flow so that one can cut off the sheaf in $M \times \bR_t$ along the $\bR_t$-direction and get a sheaf with singular support in $\Lambda$. Using Theorem \ref{thm:doubling_ad}, that construction can be interpreted as the image of the composition of the doubling functor $w_\Lambda$ and the positive wrapping $\wrap_\Lambda^+$, which therefore Guillermou's sheaf quantization $\msh_\Lambda(\Lambda) \to \Sh_\Lambda(M \times \bR)$ is simply the left adjoint of the microlocalization.
\end{remark}

    Finally, we remark that the construction immediately implies the following fiber sequence.

\begin{corollary}\label{cor:exact-tri}
    Let $\Lambda \subseteq S^{*}M$ be a compact subanalytic Legendrian. Then there is a fiber sequence of functors
    $$T_{-\epsilon} \rightarrow T_\epsilon \rightarrow w_\Lambda \circ m_\Lambda.$$
\end{corollary}

\subsection{Functorial local adjoints of microlocalization}\label{sec:ad-micro}

In Section \ref{sec:sato-sab}, we find that the cotwist and dual cotwist of the adjunctions $m_\Lambda^l \dashv m_\Lambda \dashv m_\Lambda^r$ defined by the fiber sequences
$$ m_\Lambda^l m_\Lambda \rightarrow \id \rightarrow S_\Lambda^+, \, S_\Lambda^- \rightarrow \id \rightarrow m_\Lambda^r m_\Lambda$$
are given by the wrap-once functors $S_\Lambda^\pm$ in Definition \ref{def:wrap_once_functors}.
The goal of this section is to show that, the above observation goes deeper to the functors $m_\Lambda^l$ and $m_\Lambda^r$ that they admit descriptions by wrappings as well.


Recall that $\msh_\Lambda$ is a sheaf on $T^* M$ supported on $\Lambda$.
To give a description for an object $F \in \msh_\Lambda(\Lambda)$, 
we choose an open cover $\{U_\alpha\}_{\alpha \subseteq I}$ of $M$ and $\{\Omega_\alpha\}_{\alpha \subseteq I}$ of $\Lambda \subseteq S^*M$ 
such that $\Omega_\alpha \subseteq S^{*}U_\alpha$.
Then we have the following commutative diagram induced by the corresponding inclusion of open sets in $T^* M$:
$$
\begin{tikzpicture}
\node at (0,1.7) {$\Sh_\Lambda(M)$};
\node at (4,1.7) {$\Sh_\Lambda(U_\alpha)$};
\node at (0,0) {$\msh_\Lambda(\dT^* M)$};
\node at (4,0) {$\msh_\Lambda(\Omega_\alpha)$};

\draw [->, thick] (0.8,1.7) -- (3.2,1.7) node [midway, above] {$j_\alpha^*$};
\draw [->, thick] (1.1,0) -- (3.1,0) node [midway, above] {$r_\alpha^*$};

\draw [->, thick] (0,1.4) -- (0,0.3) node [midway, left] {$m_{\Lambda}$};
\draw [->, thick] (4,1.4) -- (4,0.3) node [midway, right] {$m_{\Lambda \cap \Omega_\alpha}$};
\end{tikzpicture}
$$
Since the restriction maps admit left adjoints, we can pass to left adjoints and obtain the following diagram:
$$
\begin{tikzpicture}
\node at (0,1.7) {$\Sh_\Lambda(M)$};
\node at (4,1.7) {$\Sh_\Lambda(U_\alpha)$};
\node at (0,0) {$\msh_\Lambda(\dT^* M)$};
\node at (4,0) {$\msh_\Lambda(\Omega_\alpha)$};

\draw [->, thick] (3.2,1.7) -- (0.8,1.7) node [midway, above] {$\wrap_\Lambda^+ {j_\alpha}_!$};
\draw [->, thick] (3.1,0) -- (1.1,0) node [midway, above] {${r_\alpha}_!$};

\draw [->, thick] (0,0.3) -- (0,1.4) node [midway, left] {$m_\Lambda^l$};
\draw [->, thick] (4,0.3) -- (4,1.4) node [midway, right] {$m_{\Lambda \cap \Omega_\alpha}^l$};
\end{tikzpicture}
$$
Here we use the left adjoint $\wrap_\Lambda^+: \Sh(M) \rightarrow \Sh_\Lambda(M)$ from Theorem \ref{w=ad}.
Now the equivalence $\msh_\Lambda(\Lambda) \xrightarrow{\sim} \lim_{\alpha \in I} \msh_\Lambda(\Omega_\alpha)$ implies that
$$F = \clmi{\alpha \in I} \, {r_\alpha}_! r_\alpha^* F, \; F \in \msh_\Lambda(\Lambda)$$ where
$r_\alpha^*: \msh_\Lambda(\Lambda) \leftrightharpoons \msh_\Lambda(\Omega_\alpha): {r_\alpha}_!$
is the adjunction given by restrictions.

For each $\alpha \in I$, consider the fiber sequence (in $\PrLcs$)
$$ K(U_\alpha, \Omega_\alpha) \hookrightarrow \Sh_\Lambda(U_\alpha, \Omega_\alpha) \rightarrow \msh_\Lambda(\Omega_\alpha).$$
By Theorem \ref{w=ad}, the left adjoint of the inclusion
$K(U_\alpha,\Omega_\alpha) \hookrightarrow \Sh_\Lambda(U_\alpha,\Omega_\alpha)$
can be described by
\begin{align*}
\Sh_\Lambda(U_\alpha,\Omega_\alpha) &\rightarrow K(U_\alpha,\Omega_\alpha) \\
F &\mapsto \clmi{w \in W^+(\Lambda \cap \Omega_\alpha^c)} F^w \eqqcolon \wrap_\alpha^+(F)
\end{align*}
where $W^+(\Lambda \cap \Omega_\alpha^c)$ means the category of positive wrappings compactly supported in $\Omega_\alpha \cup (S^*U_\alpha \setminus \Lambda)$\footnote{Recall that $\wrap_\Lambda^\pm$ is defined by taking colimits/limits among all positive/negative wrappings in $S^*M \setminus \Lambda$. Here $\wrap_{\Omega^c}^\pm$ is therefore defined by wrappings supported in $\Omega$.}. Therefore, the formal property of the adjunction $m_\Lambda^l \dashv m_\Lambda$ implies the following (local) fiber sequence:

\begin{lemma}\label{lem:loc-ad-of-microlocalize}
    Let $\Lambda \subseteq S^{*}M$ be a subanalytic Legendrian, $U \subseteq M$ and $\Omega \subseteq S^*U$ be an open neighbourhood of some connected component of $\Lambda \cap S^*U$, such that there is a fiber sequence
    $$ K(U, \Omega) \hookrightarrow \Sh_\Lambda(U, \Omega) \rightarrow \msh_\Lambda(\Omega).$$
    Then given $F \in \Sh_\Lambda(U, \Omega)$, there is a fiber sequence where $\Omega^c = S^*U \backslash \Omega$
    $$m_{\Lambda \cap \Omega}^l m_{\Lambda \cap \Omega}(F) \rightarrow F \rightarrow \wrap_{\Lambda \cap \Omega^c}^+ F.$$
\end{lemma}

For a microsheaf $F \in \msh_\Lambda(\Lambda)$, the fiber sequence
$$ K(U_\alpha, \Omega_\alpha) \hookrightarrow \Sh_\Lambda(U_\alpha, \Omega_\alpha) \rightarrow \msh_\Lambda(\Omega_\alpha)$$
implies that there exists some $F_\alpha \in \Sh_\Lambda(U_\alpha, \Omega_\alpha)$ such that $r_\alpha^* F \xrightarrow{\sim} m_{\Lambda \cap \Omega_\alpha} (F_\alpha)$.
By adjunction this identification is induced by a morphism $m_{\Lambda \cap \Omega_\alpha}^l r_\alpha^* F \rightarrow F_\alpha$, 
and the fiber sequence provides the commuting diagram with rows being fiber sequences
$$
\begin{tikzpicture}

\node at (-4,1.7) {$m_{\Lambda \cap \Omega_\alpha}^l m_{\Lambda \cap \Omega_\alpha} m_{\Lambda \cap \Omega_\alpha}^l r_\alpha^* F$};
\node at (0,1.7) {$m_{\Lambda \cap \Omega_\alpha}^l r_\alpha^* F$};
\node at (4,1.7) {$\wrap_{\alpha}^+ m_{\Lambda \cap \Omega_\alpha}^l r_\alpha^* F$};
\node at (-4,0) {$m_{\Lambda \cap \Omega_\alpha}^l m_{\Lambda \cap \Omega_\alpha} F_\alpha$};
\node at (0,0) {$F_\alpha$};
\node at (4,0) {$\wrap_{\alpha}^+ F_\alpha$};

\draw [->, thick] (-1.7,1.7) -- (-1.1,1.7) node [midway, above] {$ $};
\draw [->, thick] (1.1,1.7) -- (2.5,1.7) node [midway, above] {$ $};
\draw [->, thick] (-2.3,0) -- (-0.4,0) node [midway, above] {$ $};
\draw [->, thick] (0.4,0) -- (3.2,0) node [midway, above] {$ $};

\draw [->, thick] (-4,1.2) -- (-4,0.3) node [midway, right] {$ $};
\draw [->, thick] (0,1.2) -- (0,0.3) node [midway, right] {$ $};
\draw [->, thick] (4,1.2) -- (4,0.3) node [midway, right] {$ $};
\end{tikzpicture}
$$

Here by definition $\wrap_{\alpha}^+ :  \Sh_\Lambda(U_\alpha ,\Omega_\alpha) \rightarrow K(U_\alpha,\Omega_\alpha)$
is the left adjoint of the standard inclusion and its effect is to blow away all microsupport in $\Omega_\alpha$. Hence by definition, $\wrap_{\alpha}^+ m_{\Lambda \cap \Omega_\alpha}^l r_\alpha^* F = 0$ since $m_{\Lambda \cap \Omega_\alpha}^l r_\alpha^* F$ has microsupport in $\Omega_\alpha$.
Therefore, we conclude that
$$m_{\Lambda \cap \Omega_\alpha}^l r_\alpha^* F  \rightarrow F_\alpha \rightarrow \wrap_{\alpha}^+ F_\alpha$$
with the morphisms given above is a fiber sequence.
Thus we can expression $m_{\Lambda \cap \Omega_\alpha}^l$ by
\begin{align*}
m_{\Lambda}^l(F)  &= m_{\Lambda}^l  \clmi{\alpha \in I} \, {r_\alpha}_! r_\alpha^* F 
=  \clmi{\alpha \in I} \, m_{\Lambda}^l  {r_\alpha}_! r_\alpha^* F \\
&=  \clmi{\alpha \in I} \, \wrap_\Lambda^+ {j_\alpha}_!  m_{\Lambda \cap \Omega_\alpha}^l r_\alpha^* F
=  \clmi{\alpha \in I} \, \wrap_\Lambda^+ {j_\alpha}_!  \mathrm{Fib} \left( F_\alpha \rightarrow \wrap_{\alpha}^+ F_\alpha \right).
\end{align*}
In summary, we have the lemma
\begin{lemma}\label{lem:left_right_adjoints_of_microlocalization}
Fix an open cover $\{U_\alpha\}_{\alpha \subseteq I}$ of $M$ and $\{\Omega_\alpha\}_{\alpha \subseteq I}$ of $\Lambda \subseteq S^*M$ 
such that $ \Omega_\alpha \subseteq S^{*}U_\alpha$.
Let $F \in \msh_\Lambda(\Lambda)$.
Then the left and right adjoint of $m_\Lambda: \Sh_\Lambda(M) \rightarrow \msh_\Lambda(\Lambda)$
can be described as
$$ m_{\Lambda}^l(F) =  \clmi{\alpha \subseteq I} \, \wrap_\Lambda^+ {j_\alpha}_!  \mathrm{Fib} \left( F_\alpha \rightarrow \wrap_{\alpha}^+ F_\alpha \right),$$
and
$$ m_{\Lambda}^r(F) =  \lmi{\alpha \subseteq I} \, \wrap_\Lambda^-  {j_\alpha}_* \mathrm{Cofib} \left( \wrap_{\alpha}^- F_\alpha \rightarrow F_\alpha \right)$$
where the $F_\alpha \in \Sh_\Lambda(U_\alpha, \Omega_\alpha)$'s are local representatives of $F$.
\end{lemma}

\begin{remark}
In the following section, we will make a more careful choice of the covers $\{U_\alpha\}_{\alpha \subseteq I}$, $\{\Omega_\alpha\}_{\alpha \subseteq I}$,
and representatives $\{ F_\alpha\}_{\alpha \subseteq I}$. Such a choice will provide us with a variant of antimicrolocalization result:
a result which embedding microsheaves into certain category whose objects are represented by sheaves.
\end{remark}

\begin{remark}\label{rem:multi-component-ad}
    Following Remark \ref{rem:multi-component-sato}, we remark that the above computation also works in the case when we take microlocalization along a single connected component $\Lambda_i \subseteq \Lambda \subseteq S^{*}M$. In this case 
    $$m_\Lambda^l, m_\Lambda^r: \msh_{\Lambda}(\Lambda_i) \rightarrow \Sh_\Lambda(M)$$
    can be computed by the above formula as well.
\end{remark}

    We remark that the lemma above allows us to further cut-off the singular support of the local representative $F_\alpha^0 \in \Sh_{\Lambda \cap \Omega_\alpha}(U_\alpha, \Omega_\alpha)$ and obtain a refined local representative $F_\alpha \in \Sh_{\Lambda \cap \Omega_\alpha}(U_\alpha)$ such that
    $$m_{\Lambda \cap \Omega_\alpha}(F_\alpha) = F|_{\Lambda \cap \Omega_\alpha} \in \msh_\Lambda(\Lambda \cap \Omega_\alpha),$$
    which give a different proof of the lemma stated in the previous section.

\begin{corollary}[Lemma \ref{lem:refine-cutoff-0}]\label{lem:refine-cutoff}
    Let $\Lambda \subseteq S^{*}M$ be a locally closed subanalytic Legendrian such that $\pi|_\Lambda: \Lambda \rightarrow M$ is finite. Then for $(x, \xi) \in \Lambda$, there is a neighbourhood $U$ of $x \in M$ and ${F}_U \in \Sh_{\Lambda \cap S^{*}U}(U)$ such that $m_{\Lambda \cap S^{*}U}({F}_U) = {F} \in \msh_\Lambda(S^*U)$.
\end{corollary}
\begin{proof}
    Consider an open neighbourhood $\Omega \subset S^{*}U$ of all finite components of $\Lambda \cap S^{*}U$ such that $\msh_\Lambda(\Omega)$ fits in a fiber sequence
    $$ K(U, \Omega) \hookrightarrow \Sh_\Lambda(U, \Omega) \rightarrow \msh_\Lambda(\Omega).$$
    For a representative $F \in \Sh_\Lambda(U, \Omega)$, by Lemma \ref{lem:loc-ad-of-microlocalize} we know the fiber sequence
    $$ m_{\Lambda \cap S^{*}U}^l m_{\Lambda \cap S^{*}U}( F|_{U} ) \rightarrow F|_{U} \rightarrow \wrap_{\Lambda \cap \Omega^c}^+(F|_{U}).$$
    Then we claim that $m_{\Lambda \cap S^{*}U}^l m_{\Lambda \cap S^{*}U} F|_{U} \in \Sh_{\Lambda \cap S^{*}U}(U)$. Actually, since $W^+(\Omega)$ only consists of positive wrappings supported in $\Omega$, we know that
    $$F|_{U} \rightarrow \wrap_{\Lambda \cap \Omega^c}^+(F|_{U})$$
    is an isomorphism in $S^{*}U \backslash \Omega$. Hence $m_{\Lambda \cap S^{*}M}^l m_{\Lambda \cap S^{*}M}( F|_{U}) \in \Sh_\Omega(U)$. Therefore, since $F|_U$ and $\wrap_{\Omega^c}^+(F|_{U}) \in \Sh_\Lambda(U, \Omega)$, we can conclude that 
    \begin{equation*}
        m_{\Lambda \cap S^{*}M}^l m_{\Lambda \cap S^{*}M}( F|_{U}) \in \Sh_{\Lambda \cap S^*M}(U). \qedhere
    \end{equation*}
\end{proof}

Finally, we use the characterization of $m_\Lambda^l$ to characterize the cotwist $S_\Lambda^+$ of the adjunction pair $m_\Lambda \vdash m_\Lambda^l$ in terms of global wrappings in Definition \ref{def:wrap_once_functors} without appealing to the Sato-Sabloff fiber sequence (though one may realize that the ingredients of the Sato fiber sequence is hidden in Lemma \ref{lem:loc-ad-of-microlocalize}). 
Consider
$$m_\Lambda^l m_\Lambda (F)  = \wrap_\Lambda^+ \clmi{\alpha \subseteq I} \, j_{\alpha !}\mathrm{Fib}(F|_{U_\alpha}\rightarrow \wrap_\alpha^+ F|_{U_\alpha}).$$
It implies the following different characterization of the cofiber in terms of local wrappings
$$S_{\Lambda,{alg}}^+(F) = \wrap_\Lambda^+ \left( \clmi{\alpha \subseteq I} \, j_{\alpha !} (\wrap_\alpha^+F|_{U_\alpha}) \right).$$

\begin{theorem}\label{thm:wrap-local-to-global}
For $\Lambda \subseteq S^*M$ a compact subanalytic Legendrian, we have
$$S_\Lambda^+(F) = \wrap_\Lambda^+ T_\epsilon(F) = \wrap_\Lambda^+ \left( \clmi{\alpha \in I} \, j_{\alpha !} (\wrap_\alpha^+F|_{U_\alpha}) \right).$$
\end{theorem}

The statement relating (the colimit of) local wrappings to global wrappings is nontrivial since it is in general hard to glue positive Hamiltonian flows geometrically. We need the following technical lemma which glue local Hamiltonian pushoffs of sheaves in $\Sh_\Lambda(M)$.

\begin{lemma}\label{lem:continuation}
Let $\Phi$ and $\Psi: S^* M \times I \rightarrow S^* M$ be compactly supported contact isotopies with Hamiltonians $H_\varphi$ and $H_\psi$. Let $K_{\min}$, $K_{\max} \in \Sh(M \times M \times I)$ be the sheaf kernel given by 
$$\clmi{H \leq \min(H_\varphi,H_\psi)} \, K(\Phi_H), \ \clmi{H \leq \max(H_\varphi,H_\psi)} K(\Phi_H)  \in \Sh(M \times M \times I)$$ where $K(H)$ is the GKS sheaf quantization of the contact isotopy $\Phi_H$ associated to the smooth function $H$.
Then the diagram of sheaves
$$
\begin{tikzpicture}
\node at (0,1.7) {$K_{\min}$};
\node at (3.3,1.7) {$K(\Phi)$};
\node at (0,0) {$K(\Psi)$};
\node at (3.3,0) {$K_{\max}$};
\draw [->, thick] (0.6,1.7) -- (2.7,1.7) node [midway, above] {$ $};
\draw [->, thick] (0.6,0) -- (2.7,0) node [midway, above] {$ $};
\draw [->, thick] (0,1.3) -- (0,0.3) node [midway, right] {$ $}; 
\draw [->, thick] (3.3,1.3) -- (3.3,0.3) node [midway, right] {$ $};
\end{tikzpicture}
$$ 
is a pullback/pushout diagram. 
In particular, when $H_{\max} = \max(H_\varphi, H_\psi )$ is a smooth Hamiltonian, $K_{\max} = K(H_{\max}) \in \Sh_{\Lambda_{\Phi_{\max}}}(M \times M \times I)$.
\end{lemma}

\begin{proof}
Pick an increasing sequence of non-negative Hamiltonians $H_k \rightarrow \min(H_\varphi,H_\psi)$ in $L^\infty$-norm as $k \rightarrow \infty$. One can check that this is a cofinal sequence of Hamiltonians.
Denote by $\Theta_k$ for the associated isotopies and $K(\Theta_k) \in \Sh(M \times M \times I)$ the corresponding GKS sheaf quantizations. 
Note that 
$(H_k)_{k \in \NN}$ form a Cauchy sequence with respect to the $L^\infty$-norm on smooth functions. Let
$$K_{\min} = \clmi{k \rightarrow \infty}\, K(\Theta_k).$$
As a corollary of Proposition \ref{prop:mses}~(9), we have the microsupport estimation
$$\ms( K_{\min} ) = \ms\Big( \clmi{k \rightarrow \infty} K(\Theta_k)\Big) \subseteq \bigcap_{N \in \NN} \overline{ \bigcup_{k \geq N} \ms(K(\Theta_k)) } = \bigcap_{N \in \NN} \overline{ \bigcup_{k \geq N} \Lambda_{\Theta_k} }.$$
Since the colimit is also the limit, in the sense of analysis, 
of the sheaf kernels with respect to the $L^\infty$-norm (equivalently, the Hofer norm) on Hamiltonian functions, using the completeness theorem of sheaf kernels by Asano-Ike \cite[Proposition 5.12]{AsanoIke-C0}, we know that $K_{\min}|_{t=0} = 1_\Delta$ and $K_{\min}$ is invertible as a sheaf kernel. 
Similarly, by considering an increasing cofinal sequence $H'_k \rightarrow \max(H_\varphi, H_\psi)$ in $L^\infty$-norm as $k \rightarrow \infty$ we can define $K_{\max} \in \Sh(M \times M \times I)$.

Define $K'_{\max} \coloneqq \mathrm{Cofib}(K(\Phi) \leftarrow K_{\min} \rightarrow K(\Psi))$. We now show that $K_{\max} = K'_{\max}$. Since $K_{\max}|_{t = 0} = K'_{\max}|_{t = 0} = 1_\Delta$, it suffices to check the isomorphism on the singular support by considering the microlocal restriction $\Sh(M) \rightarrow \msh^\pre(\dT^*M)$. We make use of the fact that
$$K'_{\max} = \clmi{k \rightarrow \infty}\,\mathrm{Cofib}(K(\Phi) \leftarrow K(\Theta_k) \rightarrow  K(\Psi)).$$ 
Let $\Omega_{\varphi} = \{(x, \xi) \,|\, H_\varphi(x, \xi) \geq H_\psi(x, \xi)\}$ and $\Omega_{\psi} = \{(x, \xi) \,|\, H_\psi(x, \xi) \geq H_\varphi(x, \xi)\}$.
By picking $\varphi_k$ more carefully, we may assume that $\varphi_k = \psi$ on a family of increasing exhausting open subsets of $\Omega_\varphi$ and $\varphi_k = \varphi$ on a family of increasing exhausting open subsets of $\Omega_\psi$. Thus, the colimit sheaf on the interior $\dT^*M \times \Omega_{\varphi}^\circ \times T^*I$ is computed by
$$\mathrm{Cofib}(K(\Phi) \leftarrow K(\Theta_k) \xrightarrow{\sim} K(\Psi)).$$
Thus on the interior $\dT^*M \times \Omega_{\varphi}^\circ \times T^*I$ we have isomorphisms of the microlocal restrictions
\begin{align*}
    K_{\max}|_{\dT^*M \times \Omega_{\varphi}^\circ \times T^*I} \simeq K'_{\max}|_{\dT^*M \times \Omega_{\varphi}^\circ \times T^*I} \simeq K(\Phi)|_{\dT^*M \times \Omega_{\varphi}^\circ \times T^*I}.
\end{align*}
A similar observation holds on the interior $\dT^*M \times \Omega_{\psi}^\circ \times T^*I$. 

Finally, consider the closed subset $Z = \dT^*M \backslash (\Omega_\varphi^\circ \cup \Omega_\psi^\circ)$. Since $H_\varphi = H_\psi$ on $Z$, we can choose $H'_k \rightarrow \max(H_\varphi, H_\psi)$ in $L^\infty$-norm as $k \rightarrow \infty$ such that $H'_k = H_k$ over a decreasing family of neighborhoods $\{Z_k\}_{k \in \NN}$ of $Z$. We also choose $H_{\varphi,k} \to H_\varphi$ and $H_{\psi, k} \to H_{\psi}$ in  $L^\infty$-norm as $k \rightarrow \infty$ such that $H_{\varphi,k} = H_{\psi,k} = H_k$ over the decreasing family of neighborhoods of $Z$. 
We make use of the fact that
$$K'_{\max} = \clmi{k \to \infty}\, \mathrm{Cofib}(K(\Phi_k) \leftarrow K(\Theta_k) \rightarrow K(\Psi_k)).$$
We obtain that on the neighbourhood $Z_k$ of $Z$, the colimit is computed by
$$\mathrm{Cofib}(K(\Phi_k) \xleftarrow{\sim} K(\Theta_k) \xrightarrow{\sim} K(\Psi_k)) \xrightarrow{\sim} K(\Theta'_k).$$
Therefore, we know the isomorphism $K_{\max}|_{\dT^*M \times Z \times T^*I} \simeq K_{\max}|_{\dT^*M \times Z \times T^*I}$ of microlocal restrictions. This concludes the pullback/pushout diagram.
\end{proof}

\begin{remark}
    In particular, our proof implies that $\ms(K_{\max}) = (\ms(K(\Phi)) \cap \dT^*M \times \Omega_\varphi \times T^*I) \cup (\ms(K(\Psi)) \cap \dT^*M \times \Omega_\psi \times T^*I) \cup (\ms(K_{\min}) \cap \dT^*M \times Z \times T^*I)$.
\end{remark}

\begin{remark}
    The sequence of non-negative Hamiltonians defines a Cauchy sequence of sheaf kernels in the category of sheaves with respect to the Hofer norm of Hamiltonians (more precisely, with respect to the interleaving distance on the sheaf category). However, we do not know whether the sequence of Hamiltonian diffeomorphisms converge to a homeomorphism.
\end{remark}

Using the above lemma, one may hope to glue all small wrappings $\varphi_\alpha \in W^+(\Lambda \cap \Omega_\alpha^c)$ for $\alpha \subseteq I$ and conclude that the colimit $\mathrm{colim}_{\alpha \subseteq I} j_{\alpha !}(\wrap_\alpha^+ F|_{U_\alpha})$ is isomorphic to $T_\epsilon(F)$. However, this is not true as the large colimit will push the singular support of $F$ off to the boundary of the union of the open subsets which is impossible for one to control. Therefore, we will have to choose a smaller wrapping category whose colimit is isomorphic to $T_\epsilon(F)$ and make sure that any further wrapping is contained in $W^+(\Lambda)$ which can be absorbed into the colimit $\wrap_\Lambda^+: \Sh(M) \to \Sh_\Lambda(M)$.

\begin{proof}[Proof of Theorem \ref{thm:wrap-local-to-global}]
Consider the left hand side. Without loss of generality, now assume that the open cover $\{\Omega_\alpha\}_{\alpha \subseteq I}$ is locally finite and $\Omega_\alpha \subseteq S^*M$ are sufficiently small open subsets. Consider for $\alpha \subseteq I$ all the combinations of small positive wrapping $\varphi_\alpha^\star \in W^+(\Omega_\alpha^c)$ such that the Hamiltonian $\max_{\alpha \subseteq I} H_{\varphi_\alpha^*} = \epsilon H_T$ where $H_T$ is the Hamiltonian that defines the Reeb flow $T_t: S^*M \rightarrow S^*M$. 
Denote by $W^+(\Omega_\alpha)^\star$ the corresponding subcategories of positive wrappings. 
Since $\max_{\alpha \subseteq I}(H_{\varphi_\alpha^\star}) = \epsilon H_T$, by Lemma \ref{lem:continuation}, there exists a continuation map 
$$\clmi{\alpha \subseteq I, \varphi_\alpha^\star \in W^+(\Omega_\alpha^c)^\star}\, F^{\varphi_\alpha^\star} \xrightarrow{\sim} T_\epsilon(F).$$
For any $p \in \Lambda \cap S^*U_\alpha$, we know that there exists $\varphi_\alpha^\star \in W^+(\Omega_\alpha^c)^\star$ such that $\varphi_\alpha^\star(p) \notin \Lambda$. Therefore, the diagram of compositions $\varphi_\alpha' \circ \varphi_\alpha^\star$ for all $\varphi_\alpha^\star \in W^+(\Omega_\alpha^c)^\star$ and all $\varphi_\alpha' \in W^+(\Lambda)$ forms a cofinal diagram in $W^+(\Omega_\alpha^c)$. Since $\Omega_\alpha \subseteq S^*M$ are chosen to be sufficiently small, any further wrapping in $W^+(\Omega_\alpha^c) \backslash W^+(\Omega_\alpha^c)^\star$ are also supported away from $\Lambda$. Hence by the definition of $\wrap_\Lambda^+$ we can conclude that
\begin{equation*}
    S_\Lambda^+(F) = \wrap_\Lambda^+ T_\epsilon(F) = \wrap_\Lambda^+ \clmi{\alpha \subseteq I, \varphi_\alpha^\star \in W^+(\Omega_\alpha^c)^\star} \, F^{\varphi_\alpha^\star} =  \wrap_\Lambda^+ \clmi{\alpha \subseteq I, \varphi_\alpha \in W^+(\Lambda \cap \Omega_\alpha^c)} \, F^{\varphi_\alpha}.
\end{equation*}

Consider the right hand side. Fix $\alpha \subseteq I$ and consider $j_{\alpha !} (\wrap_\alpha^+F|_{U_\alpha})$.
The functor $\wrap_\alpha^+$ is given by wrappings $\varphi_\alpha$ which are compactly supported on $\Omega_\alpha \cup (S^* U_\alpha \setminus \Lambda)$, i.e.~$\varphi_\alpha \in W^+(\Lambda \cap \Omega_\alpha^c)$. 
Thus
$$j_{\alpha !} (\wrap_\alpha^+F|_{U_\alpha}) = j_{\alpha !} \left(\clmi{\varphi_\alpha \in W^+(\Lambda \cap \Omega_\alpha^c)} (F|_{U_\alpha})^{\varphi_\alpha} \right)
= \clmi{\varphi_\alpha \in W^+(\Lambda \cap  \Omega_\alpha^c)} (F_{U_\alpha})^{\varphi_\alpha}.$$
Here, we use the fact that $\varphi_\alpha$ is compactly supported in $S^* U_\alpha$.
Now, we notice that the colimit $\operatorname{colim}_{\alpha \subseteq I}$ is given by a \v{C}ech diagram, so for any $l \leq k$, the intersection contributes a family of maps $F_{U_{\alpha_1 \dots \alpha_k}}^{\varphi_{\alpha_1 \dots \alpha_k}} \rightarrow F_{U_{\alpha_1 \dots \alpha_l}}^{\varphi_{\alpha_1 \dots \alpha_l}}$
where $\varphi_{\alpha_1 \dots \alpha_k}$ runs through all wrappings which are compactly supported in $\Omega_{\alpha_1 \dots \alpha_k}\coloneqq \Omega_{\alpha_1} \cap \dots \cap \Omega_{\alpha_k}$ and smaller than $\varphi_{\alpha_1 \dots \alpha_l}$ as wrappings on $\Omega_{\alpha_1 \dots \alpha_l}$. 
Hence by considering the colimit with respect to the diagram of the continuation maps $F_{U_{\alpha_1 \dots \alpha_k}}^{\varphi_{\alpha_1 \dots \alpha_k}} \rightarrow F_{U_{\alpha_1 \dots \alpha_l}}^{\varphi_{\alpha_1 \dots \alpha_l}}$, we get
$$S_{\Lambda,alg}^+(F) = \wrap_\Lambda^+ \clmi{\alpha \subseteq I, \varphi_\alpha \in W^+(\Lambda \cap \Omega_\alpha^c)}F_{U_\alpha}^{\varphi_\alpha}.$$
In other words, the colimit is small positive wrappings of $F_{U_{\alpha_1 \dots \alpha_l}}$ on each $\Omega_{\alpha_1 \dots \alpha_l}$ glued along their intersections by small positive wrappings of $F_{U_{\alpha_1\dots\alpha_k}}$ on $\Omega_{\alpha_1 \dots \alpha_k}$ for $l \leq k$. Then we complete the diagram of the continuation maps by including $F_{U_{\alpha_1 \dots \alpha_k}}^{\varphi_{\alpha_1 \dots \alpha_{l}}} \rightarrow F_{U_{\alpha_1\dots\alpha_{l}}}^{\varphi_{\alpha_1 \dots \alpha_{l}}}$ for any $l \leq k$ as follows. 
\[\mbox{\footnotesize
\xymatrix{
\cdots \ar[r] \ar[d] \ar[dr] & \cdots \ar[r] \ar[d] \ar[dr] & \cdots \ar[r] \ar[d] \ar[dr] &  \cdots \ar[d] \\
\cdots \ar[r] \ar[d] & \bigoplus F_{U_{\alpha_1\alpha_2\alpha_3}}^{\varphi_{\alpha_1\alpha_2\alpha_3}} \ar[r] \ar[d] \ar[dr] & \bigoplus F_{U_{\alpha_1\alpha_2}}^{\varphi_{\alpha_1\alpha_2\alpha_3}} \ar[r] \ar[d] \ar[dr] & \bigoplus F_{U_{\alpha_1}}^{\varphi_{\alpha_1\alpha_2\alpha_3}} \ar[d] \\
\cdots \ar[r] \ar[d] & \bigoplus F_{U_{\alpha_1\alpha_2\alpha_3}}^{\varphi_{\alpha_1\alpha_2}} \ar[r] \ar[d] & \bigoplus F_{U_{\alpha_1\alpha_2}}^{\varphi_{\alpha_1\alpha_2}} \ar[r] \ar[d] \ar[dr] & \bigoplus F_{U_{\alpha_1}}^{\varphi_{\alpha_1\alpha_2}} \ar[d] \\
\cdots \ar[r] & \bigoplus F_{U_{\alpha_1\alpha_2\alpha_3}}^{\varphi_{\alpha_1}} \ar[r] & \bigoplus F_{U_{\alpha_1\alpha_2}}^{\varphi_{\alpha_1}} \ar[r] & \bigoplus F_{U_{\alpha_1}}^{\varphi_{\alpha_1}} .
}}\]
Using the fact that the upper triangular part of the diagram is cofinal, we can conclude that $S_{\Lambda, alg}^+(F)$ can also be considered as the colimit of the large diagram.
Consider the colimit along each horizontal line which yields a cofinal sequence in the large diagram. We can thus conclude that
\begin{equation*}
    S_{\Lambda,alg}^+(F) = \wrap_\Lambda^+ \clmi{\alpha \subseteq I, \varphi_\alpha \in W^+(\Lambda \cap \Omega_\alpha^c)}F_{U_\alpha}^{\varphi_\alpha} = \wrap_\Lambda^+ \clmi{\alpha \subseteq I, \varphi_\alpha \in W^+(\Lambda \cap \Omega_\alpha^c)} \, F^{\varphi_\alpha} = S_\Lambda^+(F). \qedhere
\end{equation*}
\end{proof}

\begin{theorem}\label{thm:wrap-local-to-global-2}
For $\Lambda \subseteq S^*M$ a compact subanalytic Legendrian, we have
$$S_\Lambda^-(F) = \wrap_\Lambda^- T_{-\epsilon}(F) = \wrap_\Lambda^- \left( \lmi{\alpha \subseteq I} \, j_{\alpha !} (\wrap_\alpha^-F|_{U_\alpha}) \right).$$
\end{theorem}

\subsection{Functorial global adjoints of microlocalizations}\label{sec:double-ad-micro}
    Consider ${F} \in \msh_\Lambda(\Lambda)$. Then by Corollary \ref{lem:refine-cutoff} there exists an open covering $\{U_\alpha\}_{\alpha\subseteq I}$ of $M$ and ${F}_{\alpha} \in \Sh_{\Lambda \cap S^*U_\alpha}(U_\alpha)$ such that
    $$m_{\Lambda \cap S^*U_\alpha}({F}_{\alpha}) = {F}|_{\Lambda \cap S^*U\alpha} \in \msh_\Lambda(\Lambda \cap S^{*}U_\alpha).$$
    From Section \ref{sec:ad-micro}, we have seen that the adjoints of microlocalization can be characterized by
    \begin{align*}
        m_\Lambda^l(F) &= \wrap_\Lambda^+ \clmi{\alpha \subseteq I} \, j_{\alpha !}\mathrm{Fib}(F_\alpha \rightarrow \wrap_\alpha^+F_\alpha), \\
        m_\Lambda^r(F) &= \wrap_\Lambda^- \lmi{\alpha \subseteq I} \, j_{\alpha *}\mathrm{Cofib}(\wrap_\alpha^- F_\alpha \rightarrow F_\alpha).
    \end{align*}
    Our goal in this section is to interpret the local adjoints as positive and respectively negative pushoff $\mathrm{Fib}(F_\alpha \to F_\alpha^{\varphi_\alpha})$ and $\mathrm{Cofib}(F_\alpha^{\varphi_\alpha^{-1}} \rightarrow F_\alpha)$
    and then obtain the doubling functor as the adjoint functors of microlocalization from local-to-global principle. 
    
    One may hope to glue all small wrappings $\varphi_\alpha \in W^+(\Lambda \cap \Omega_\alpha^c)$ for $\alpha \subseteq I$ and define the doubling as the colimit $\mathrm{colim}_{\alpha \subseteq I} j_{\alpha !}\mathrm{Fib}(F_\alpha \to \wrap_\alpha^+ F_\alpha)$. However, similar to the issue in Theorem \ref{thm:wrap-local-to-global}, this is not true as the large colimit will push the singular support of $F$ off to the boundary of the union of the open subsets which is impossible for one to control. Moreover, the local representatives $F_\alpha \in \Sh_{\Lambda}(U_\alpha)$ may not be identical on the intersections and that will introduce extra difficulty to control the singular support of the gluing. However, similar to the proof of Theorem \ref{thm:wrap-local-to-global}, we will choose a smaller wrapping category and prove that the colimit has singular support $\Lambda \cup T_\epsilon(\Lambda)$ and make sure that any further wrapping is contained in $W^+(\Lambda)$ which can be absorbed into the colimit $\wrap_\Lambda^+: \Sh(M) \to \Sh_\Lambda(M)$.

    For each $\alpha \subseteq I$, consider all small positive wrapping $\varphi_\alpha^\star \in W^+(\Omega_\alpha^c)$ such that the Hamiltonian $\max_{\alpha \subseteq I} H_{\varphi_\alpha^*} = \epsilon H_T$ where $H_T$ is the Hamiltonian that defines the Reeb flow. Denote by $W^+(\Omega_\alpha)^\star$ the corresponding subcategories of positive wrappings. It is clear that for any $p \in \Lambda \cap S^*U_\alpha$, there exists $\varphi_\alpha^\star \in W^+(\Omega_\alpha^c)^\star$ such that $\varphi_\alpha^\star(p) \notin \Lambda$. See Figure \ref{fig:doubling-glue}.

\begin{figure}[h!]
    \centering
    \includegraphics[width=0.85\textwidth]{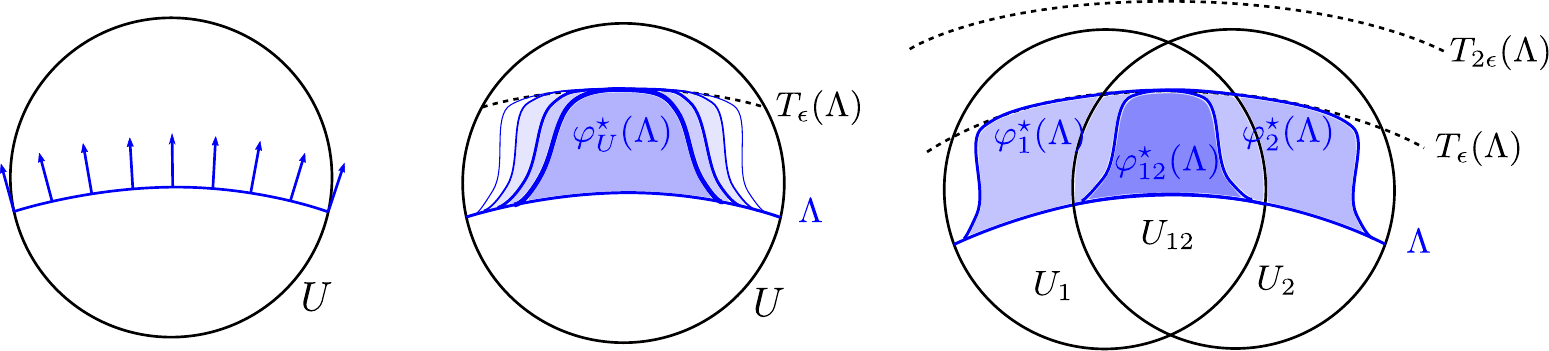}
    \caption{The collection of all small positive wrappings $\varphi_\alpha^\star \in W^+(\Omega_\alpha^c)$ such that $\max_{\alpha \subseteq I} H_{\varphi_\alpha^\star} = \epsilon H_T$ and the gluing of the small wrappings on the intersection.}
    \label{fig:doubling-glue}
\end{figure}
    
    Given $F \in \msh_\Lambda(\Lambda)$, by Corollary \ref{lem:loc-ad-of-microlocalize} there exists a locally finite open covering $\mathscr{U} = \{U_\alpha\}_{\alpha \subseteq I}$, $\{\Omega_\alpha\}_{\alpha \subseteq I}$ and a collection of sheaves $F_\alpha \in Sh_\Lambda(U_\alpha)$ such that $$F|_{\Lambda \cap S^*U_\alpha} = m_\Lambda(F_\alpha).$$
    Given the wrapping subcategories $W^+(\Omega_\alpha)^\star$, we define
    $$w_\Lambda^l(F) = \clmi{\alpha \subseteq I, \varphi_\alpha^\star \in W^+(\Omega_\alpha^c)^\star}\, j_{\alpha !} \mathrm{Fib} \big( F_\alpha \rightarrow F_\alpha^{\varphi_\alpha^\star} \big).$$ 
    By Lemma \ref{lem:loc-ad-of-microlocalize}, 
    there are continuation maps $j_{\alpha_1 \dots \alpha_k !}F_{\alpha_1 \dots \alpha_k} \to j_{\alpha_1 \dots \alpha_l !}F_{\alpha_1 \dots \alpha_l}$ for $l \leq k$, which induce continuation maps $j_{\alpha_1 \dots \alpha_k !} \mathrm{Fib} (F_{\alpha_1 \dots \alpha_k} \to F_{\alpha_1 \dots \alpha_k}^{\varphi_{\alpha_1 \dots \alpha_k}}) \to j_{\alpha_1 \dots \alpha_l !} \mathrm{Fib} (F_{\alpha_1 \dots \alpha_l} \to F_{\alpha_1 \dots \alpha_l}^{\varphi_{\alpha_1 \dots \alpha_l}})$.

\begin{theorem}\label{thm:doubling_l_ad_twocopy}
    For $\Lambda \subseteq S^{*}M$ a compact subanalytic Legendrian, we have
    \begin{align*}
    m_{\Lambda \subseteq \Lambda \cup T_{2\epsilon}(\Lambda)}^l(F) = \wrap_{\Lambda \cup T_{2\epsilon}(\Lambda)}^+ \circ w_\Lambda^l(F).
    \end{align*}
\end{theorem}
\begin{proof}
    From what we have proved in Section \ref{sec:ad-micro}, it suffices to show that
    $$\wrap_{\Lambda \cup T_{2\epsilon}(\Lambda)}^+ \clmi{\alpha \subseteq I} \, j_{\alpha !}\mathrm{Fib}(F_\alpha \rightarrow \wrap_\alpha^+F_\alpha) = \wrap_{\Lambda \cup T_{2\epsilon}(\Lambda)}^+ \clmi{\alpha \subseteq I, \varphi_\alpha^\star \in W^+(\Omega_\alpha^c)^\star} j_{\alpha !}\mathrm{Fib} \big( F_\alpha \rightarrow F_\alpha^{\varphi_\alpha^\star} \big).$$
    Note that the colimit $\mathrm{colim}_{\alpha \subseteq I}$ is given by a \v{C}ech diagram, so any intersection $U_{\alpha_1 \dots \alpha_k} \subseteq U_{\alpha_1 \dots \alpha_l}$ for $l \leq k$ contributes to a family of continuation maps
    $$j_{\alpha_1 \dots \alpha_k !}\mathrm{Fib}\big( F_{\alpha_1 \dots \alpha_k} \rightarrow F_{\alpha_1 \dots \alpha_k}^{\varphi_{\alpha_1 \dots \alpha_k}} \big) \to j_{\alpha_1 \dots \alpha_l !}\mathrm{Fib}\big( F_{\alpha_1 \dots \alpha_l} \rightarrow F_{\alpha_1 \dots \alpha_l}^{\varphi_{\alpha_1 \dots \alpha_l}} \big)$$
    where $\varphi_{\alpha_1 \dots \alpha_k}$ runs through all wrappings compactly supported in $\Omega_{\alpha_1 \dots \alpha_k}$ and are smaller than $\varphi_{\alpha_1 \dots \alpha_l}$ in $\Omega_{\alpha_1 \dots \alpha_l}$ for $l \leq k$. Hence, by considering the colimit with respect to continuation maps $j_{\alpha_1 \dots \alpha_k !}\mathrm{Fib}(F_{\alpha_1 \dots \alpha_k} \rightarrow F_{\alpha_1 \dots \alpha_k}^{\varphi_{\alpha_1 \dots \alpha_k}}) \to j_{\alpha_1 \dots \alpha_l !}\mathrm{Fib}(F_{\alpha_1 \dots \alpha_l} \rightarrow F_{\alpha_1 \dots \alpha_l}^{\varphi_{\alpha_1 \dots \alpha_l}})$, we can write
    \begin{align*}
    \clmi{\alpha \subseteq I} \,j_{\alpha !}\mathrm{Fib}(F_\alpha \rightarrow \wrap_\alpha^+F_\alpha) = \clmi{\alpha \subseteq I, \varphi_\alpha \in W^+(\Lambda \cap \Omega_\alpha^c)} \,j_{\alpha !} \mathrm{Fib}\big( F_\alpha \rightarrow F_\alpha^{\varphi_\alpha} \big). 
    \end{align*}
    Consider now the wrappings $\varphi_\alpha'$ that are supported in $S^*U_\alpha \setminus \Lambda$, i.e.~$\varphi_\alpha' \in W^+(\Lambda)$. There are canonical isomorphisms $\mathrm{Fib}\big( F_\alpha \rightarrow (F_\alpha^{\varphi_\alpha^\star})^{\varphi_\alpha'} \big) \to \mathrm{Fib}\big( F_\alpha \rightarrow F_\alpha^{\varphi_\alpha^\star} \big)^{\varphi_\alpha'}$. Since for any $p \in \Lambda \cap S^*U_\alpha$, there is $\varphi_\alpha^\star \in W^+(\Omega_\alpha^c)^\star$ such that $\varphi_\alpha^\star(p) \notin \Lambda$, the diagram of compositions $\varphi_\alpha' \circ \varphi_\alpha^\star$ for all possible $\varphi_\alpha^\star \in W^+(\Omega_\alpha^c)^\star$ and all $\varphi_\alpha' \in W^+(\Lambda)$ forms a cofinal diagram in the wrapping category $W^+(\Lambda \cap \Omega_\alpha^c)$. Since $\Omega_\alpha \subseteq S^*M$ are chosen to be sufficiently small, any further wrapping in $W^+(\Omega_\alpha^c) \backslash W^+(\Omega_\alpha^c)^\star$ are also supported away from $\Lambda$. 
    we can conclude that
    \begin{align*}
    \clmi{\alpha \subseteq I} \,j_{\alpha !}\mathrm{Fib}(F_\alpha \rightarrow \wrap_\alpha^+F_\alpha) = \clmi{\alpha \subseteq I, \varphi_\alpha^\star \in W^+(\Omega_\alpha^c)^\star, \varphi_\alpha \in W^+(\Lambda)} \,j_{\alpha !} \left( \mathrm{Fib}\big( F_\alpha \rightarrow F_\alpha^{\varphi_\alpha^\star} \big)^{\varphi_{\alpha}'} \right).
    \end{align*}
    Then we complete the diagram of the continuation maps by including $j_{\alpha\alpha_2\dots \alpha_k !}\mathrm{Fib}\big( F_{\alpha_1\dots \alpha_k} \rightarrow F_{\alpha_1\dots \alpha_k}^{\varphi_{\alpha\alpha_2\dots \alpha_k}^\star} \big)^{\varphi_{\alpha\alpha_2\dots \alpha_l}'} \to j_{\alpha_1 \dots \alpha_l !}\mathrm{Fib}\big( F_{\alpha_1 \dots \alpha_l} \rightarrow F_{\alpha_1 \dots \alpha_l}^{\varphi_{\alpha_1 \dots \alpha_l}^\star} \big)^{\varphi_{\alpha\alpha_2\dots \alpha_l}'}$ for any $l \leq k$ as follows.
\[\mbox{\tiny
\xymatrix@C=2mm{
\ar[r] \ar[dr] & \cdots \ar[r] \ar[d] \ar[dr] & \cdots \ar[r] \ar[d] \ar[dr] &  \cdots \ar[d] \\ 
\ar[r] & \bigoplus j_{\alpha_1 \alpha_2 \alpha_3 !} \mathrm{Fib}\big(F_{\alpha_1 \alpha_2 \alpha_3} \rightarrow F_{\alpha_1 \alpha_2 \alpha_3}^{\varphi_{\alpha_1 \alpha_2 \alpha_3}^\star}\big)^{\varphi_{\alpha_1\alpha_2\alpha_3}'} \ar[r] \ar[d] \ar[dr] & \bigoplus j_{\alpha_1 \alpha_2 !}\mathrm{Fib}\big( F_{\alpha_1\alpha_2} \to F_{\alpha_1\alpha_2}^{\varphi_{\alpha_1\alpha_2}^\star} \big)^{\varphi_{\alpha_1\alpha_2\alpha_3}'} \ar[r] \ar[d] \ar[dr] & \bigoplus j_{\alpha !}\mathrm{Fib}\big( F_{\alpha_1} \to F_{\alpha_1}^{\varphi_{\alpha_1}^\star} \big)^{\varphi_{\alpha_1\alpha_2\alpha_3}'} \ar[d] \\
\ar[r] & \bigoplus j_{\alpha_1 \alpha_2 \alpha_3 !}\mathrm{Fib}\big(F_{\alpha_1 \alpha_2 \alpha_3}  \rightarrow F_{\alpha_1 \alpha_2 \alpha_3}^{\varphi_{\alpha_1 \alpha_2 \alpha_3}^\star} \big)^{\varphi_{\alpha_1\alpha_2}'} \ar[r] \ar[d] & \bigoplus j_{\alpha_1 \alpha_2 !}\mathrm{Fib}\big(F_{\alpha_1\alpha_2} \to F_{\alpha_1\alpha_2}^{\varphi_{\alpha_1\alpha_2}^\star} \big)^{\varphi_{\alpha_1\alpha_2}'} \ar[r] \ar[d] \ar[dr] & \bigoplus j_{\alpha !}\mathrm{Fib}\big(F_{\alpha_1} \to F_{\alpha_1}^{\varphi_{\alpha_1}^\star} \big)^{\varphi_{\alpha_1\alpha_2}'} \ar[d] \\
\ar[r] & \bigoplus j_{\alpha_1 \alpha_2 \alpha_3 !}\mathrm{Fib}\big( F_{\alpha_1 \alpha_2 \alpha_3} \rightarrow F_{\alpha_1 \alpha_2 \alpha_3}^{\varphi_{\alpha_1 \alpha_2 \alpha_3}^\star}\big)^{\varphi_{\alpha_1}'} \ar[r] & \bigoplus j_{\alpha_1 \alpha_2 !}\mathrm{Fib}\big(F_{\alpha_1\alpha_2} \to F_{\alpha_1\alpha_2}^{\varphi_{\alpha_1\alpha_2}^\star}\big)^{\varphi_{\alpha_1}'} \ar[r] & \bigoplus j_{\alpha_1 !}\mathrm{Fib}\big(F_{\alpha_1} \to F_{\alpha_1}^{\varphi_{\alpha_1}^\star} \big)^{\varphi_{\alpha_1}'} .
}}\]
    Using the fact that the upper triangular part of the diagram is cofinal, we can conclude that $m_{\Lambda \cup T_\epsilon(\Lambda)}^l(F)$ can be considered as the colimit of the large diagram. By considering the colimit along each horizontal line which yields a cofinal sequence in the large diagram, this discussion implies that
    \begin{equation*}
    m_{\Lambda \subseteq \Lambda \cup T_{2\epsilon}(\Lambda)}^l(F) = \wrap_{\Lambda \cup T_{2\epsilon}(\Lambda)}^+ \clmi{\alpha \subseteq I, \varphi_\alpha^\star \in W^+(\Omega_\alpha^c)^\star} \,j_{\alpha !} \mathrm{Fib}\big( F_\alpha \rightarrow F_\alpha^{\varphi_\alpha^\star}\big) . \qedhere
    \end{equation*}
\end{proof}
    
    Moreover, we can prove that the doubling functor $m_{\Lambda \cup T_{2\epsilon}(\Lambda)}^l = \wrap_{\Lambda \cup T_{2\epsilon}(\Lambda)}^+ \circ w_\Lambda^l$ is fully faithful. The following lemma is analogous to Lemma \ref{lem:continuation}.

\begin{lemma}\label{lem:double-glue-ss}
    Let $\Phi_1$ and $\Phi_2: S^*M \times I \to S^*M$ be Hamiltonian isotopies with Hamiltonians $H_{\varphi_1}$ and $H_{\varphi_2}$ supported in $S^*U_1$ and $S^*U_2$. Let $F_1 \in \Sh_\Lambda(U_1)$, $F_2 \in \Sh_\Lambda(U_2)$ and $F_{12} \in \Sh_\Lambda(U_{12})$ be sheaves such that $m_\Lambda(F_1) = m_\Lambda(F_2) = m_\Lambda(F_{12})$. Let $K_{\min}$, $K_{\max} \in \Sh(M \times M \times I)$ be the sheaf kernel given by 
    $$\clmi{H \leq \min(H_{\varphi_1},H_{\varphi_2})} \, K(H), \ \clmi{H \leq \max(H_{\varphi_1},H_{\varphi_2})} K(H) \in \Sh(M \times M)$$ where $K(H)$ is the GKS sheaf quantization of the contact isotopy associated to the smooth function $H$. Let $q: U_\alpha \times I \to U_\alpha$ be the projections.
    Then consider the pullback/pushout diagram of sheaves
$$
\begin{tikzpicture}
\node at (0,1.7) {$j_{12 !} \mathrm{Fib}(q^*F_{12} \to K_{\min} \circ 
F_{12})$};
\node at (6.3,1.7) {$j_{1 !} \mathrm{Fib}(q^*F_1 \to K(\Phi_1) \circ F_{1})$};
\node at (0,0.1) {$j_{2 !} \mathrm{Fib}(q^*F_2 \to K(\Phi_2) \circ F_2)$};
\node at (6.3,0) {$\clmi{ \alpha \subseteq \{1,2\}}\,j_{\alpha\, !} \,\mathrm{Fib}(q^*F_\alpha \to K(\Phi_\alpha) \circ F_\alpha)$ }; 
\draw [->, thick] (2.5,1.7) -- (3.9,1.7) node [midway, above] {$ $};
\draw [->, thick] (2.4,0.1) -- (3.2,0.1) node [midway, above] {$ $};

\draw [->, thick] (0,1.3) -- (0,0.4) node [midway, right] {$ $}; 
\draw [->, thick] (6.3,1.3) -- (6.3,0.4) node [midway, right] {$ $};
\end{tikzpicture}
$$ 
Let $U \coloneqq U_1 \cup U_2$, $\Lambda_{\min} \coloneqq \ms(K_{\min}) \circ \Lambda$ and $\Lambda_{\max} \coloneqq  \ms(K_{\max}) \circ \Lambda$. Then 
$$\clmi{\alpha \subseteq \{1, 2\}}\,j_{\alpha !} \mathrm{Fib}(q^*F_\alpha \to K(\Phi_\alpha) \circ F_\alpha)|_{U \times I} \in \Sh_{\Lambda_{\mathrm{id}} \cup \Lambda_{\max}}(U \times I).$$
In particular, when $H_{\max} = \max( H_1, H_2 )$ is smooth, the singular support is contained in $\Lambda_{\mathrm{id}} \cup \Lambda_{\Phi_{\max}}$.
\end{lemma}
\begin{proof}
    Let $\Omega_{1} = \{(x, \xi) \,|\, H_{\varphi_1}(x, \xi) \geq H_{\varphi_2}(x, \xi)\}$ and $\Omega_{2} = \{(x, \xi) \,|\, H_{\varphi_2}(x, \xi) \geq H_{\varphi_1}(x, \xi)\}$.
    The restriction $\Sh(M) \rightarrow \msh^\pre(\Omega_1)$ or $\msh^\pre(\Omega_2)$ preserves limits, colimits, and microsupports, so for the purpose of computing 
    $$\ms\left( \clmi{ \alpha \subseteq \{1,2\} }\,j_{\alpha !} \mathrm{Fib}(q^*F_\alpha \to K(\Phi_\alpha) \circ F_\alpha) \right) \cap {\Omega}_1,$$
    it is enough to compute the diagram there.
    Consider a family of Hamiltonians $H_{k} \leq \min\{H_{\varphi_1}, H_{\varphi_2}\}$. Assume that $H_k = H_{\varphi_2}$ on an increasing family of exhausting open subsets of $\Omega_1$. Since we know that the microlocal restriction of $F_2$ and $F_{12}$ on $\Lambda \cap S^*U_{12}$ are identical, the colimit in the interior $\Omega_{1}^\circ$ is computed by
    $$\mathrm{Cofib}\left(j_{1 !} \mathrm{Fib}(q^*F_{1} \to K(\Phi_1) \circ F_{1}) \leftarrow j_{12 !} \mathrm{Fib}(q^*F_{12} \to K(H_k) \circ F_{12}) \xrightarrow{\sim} j_{2 !} \mathrm{Fib}(q^*F_{2} \to K(\Phi_2) \circ F_{2})\right).
    $$
    A similar observation holds on the interior $\Omega_{2}^\circ $. Thus
    \begin{align*}
    \ms\left(\clmi{ \alpha \subseteq \{1,2\} } \mathrm{Fib}(q^*F_\alpha \to K(\Phi_\alpha) \circ F_\alpha) \right) \cap \Omega_{1}^\circ \subseteq \Lambda_{\mathrm{id}} \cup \Lambda_{\Phi_1}, \\
    \ms\left(\clmi{ \alpha \subseteq \{1,2\} } \mathrm{Fib}(q^*F_\alpha \to K(\Phi_\alpha) \circ F_\alpha) \right) \cap \Omega_{2}^\circ \subseteq \Lambda_{\mathrm{id}} \cup \Lambda_{\Phi_2}.
    \end{align*}
    We know that the singular support of the colimit sheaf is contained in the union of $\Lambda_{\mathrm{id}}$, $\Lambda_{\Phi_1}$, $\Lambda_{\Phi_2}$ and $\Lambda_{\min}$. Let $Z = \dT^*M \setminus (\Omega_1^\circ \cup \Omega_2^\circ )$. Our singular support estimate shows that $\Lambda_{\Phi_1} \cap {\Omega}_2^\circ$ and $\Lambda_{\Phi_2} \cap {\Omega}_1^\circ$ are not contained in the singular support. On the other hand, Lemma \ref{lem:continuation} implies that
    $$\Lambda_{\max} = (\Lambda_{\Phi_1} \cap {\Omega}_1) \cup (\Lambda_{\Phi_2} \cap {\Omega}_2) \cup (\Lambda_{\min} \cap Z).$$
    Hence we can conclude that the colimit sheaf has singular support in $\Lambda_{\mathrm{id}} \cup \Lambda_{\max}$.
\end{proof}

\begin{theorem}\label{thm:doubling_l_ad_ff}
    For $\Lambda \subseteq S^{*}M$ a compact subanalytic Legendrian, the left doubling is fully faithful: $$w_{\Lambda}^l: \msh_\Lambda(\Lambda) \hookrightarrow \Sh_{\Lambda \cup T_{\epsilon}(\Lambda)}(M).$$
    Moreover, there is a nonnegative Hamiltonian $\widetilde{T}_{t}: S^*M \to S^*M$ so that $m_{\Lambda \subseteq \Lambda \cup T_{2\epsilon}(\Lambda)}^l = \widetilde{T}_\epsilon \circ w_\Lambda^l$.
\end{theorem}
\begin{proof}
    For $k \in \NN$, let $\widetilde{T}_{\epsilon-1/k}: S^*M \to S^*M$ be the nonnegative Hamiltonian flows such that $\widetilde{T}_{\epsilon-1/k}(\Lambda) = \Lambda$ and $\widetilde{T}_{\epsilon-1/k} \circ T_{\epsilon}(\Lambda) = T_{2\epsilon-1/k}(\Lambda)$. 
    This defines a cofinal sequence of wrappings from $\Lambda \cup T_\epsilon(\Lambda)$ to $\Lambda \cup T_{2\epsilon}(\Lambda)$ which is supported away from $\Lambda \subseteq S^*M$. Since $\max_{\alpha \subseteq I}H_{\varphi_\alpha^\star} = \epsilon H_T$ where $H_T$ is the Hamiltonian for the Reeb flow, we know by Lemma \ref{lem:double-glue-ss} that
    $$\ms\left(\clmi{\alpha \subseteq I, \varphi_\alpha^\star \in W^+(\Omega_\alpha^c)^\star}\, j_{\alpha !} \mathrm{Fib}\big ( F_\alpha \to F_\alpha^{\varphi_\alpha^\star}\big)\right) \subseteq \Lambda \cup T_\epsilon(\Lambda).$$
    By Lemma \ref{lem:nearby_cycle} we know that
    $\wrap_{\Lambda \cup T_{2\epsilon}(\Lambda)}^+$ is fully faithful on $\Sh_{\Lambda \cup T_{\epsilon}(\Lambda)}(M)$. Hence full faithfulness of $w_\Lambda^l$ can be reduced to the full faithfulness of $m_{\Lambda \subseteq \Lambda \cup T_{2\epsilon}(\Lambda)}^l$.
    
    For any $F, G \in \msh_\Lambda(\Lambda)$, using adjunction we have
    $$\Hom(m_{\Lambda \subseteq \Lambda \cup T_{2\epsilon}(\Lambda)}^l(F), m_{\Lambda \subseteq \Lambda \cup T_{2\epsilon}(\Lambda)}^l(G)) = \Gamma(\Lambda, \mhom(F, m_{\Lambda } \circ m_{\Lambda \subseteq \Lambda \cup T_{2\epsilon}(\Lambda)}^l(G))).$$
    Therefore, it suffices to show that for any $F \in \msh_\Lambda(\Lambda)$, $m_{\Lambda} \circ m_{\Lambda \subseteq \Lambda \cup T_{2\epsilon}(\Lambda)}^l(F) = F$. 
    Since  $\widetilde{T}_\epsilon: S^*M \to S^*M$ defines a cofinal wrapping from $\Lambda \cup T_\epsilon(\Lambda)$ to $\Lambda \cup T_{2\epsilon}(\Lambda)$ which is supported away from $\Lambda \subseteq S^*M$. Therefore, as the microlocalization functor $m_\Lambda$ only depends on the microlocal behaviour of the sheaf in a neighbourhood of $\Lambda \subseteq S^*M$, we have
    $$m_{\Lambda} \circ m_{\Lambda \subseteq \Lambda \cup T_{2\epsilon}(\Lambda)}^l(F) = m_\Lambda \circ w_\Lambda^l (F) = m_{\Lambda}\left(\clmi{\alpha \in I, \varphi_\alpha^\star \in W^+(\Omega_\alpha^c)^\star}\, j_{\alpha !} \mathrm{Fib}\big ( F_\alpha \to F_\alpha^{\varphi_\alpha^\star}\big)\right).$$
    Note that for any $\alpha \in I$, $\mathrm{Fib}\big ( F_\alpha \to F_\alpha^{\varphi_\alpha^\star}\big)$ is compactly supported in $U_\alpha$. Since the microlocalization functor $m_\Lambda$ is limit preserving and colimit preserving, we can thus write
    $$m_{\Lambda \cap \Omega_\alpha} \circ m_{\Lambda \subseteq \Lambda \cup T_{2\epsilon}(\Lambda)}^l(F) = \clmi{\alpha \subseteq I, \varphi_\alpha^\star \in W^+(\Omega_\alpha^c)^\star}\, m_{\Lambda \cap \Omega_\alpha} \left( \mathrm{Fib}\big ( F_\alpha \to F_\alpha^{\varphi_\alpha^\star}\big) \right).$$
    For any $\varphi_\alpha^\star \in W^+(\Omega_\alpha^c)^\star$, write $r_{\alpha}: \mathrm{supp}(H_{\varphi_{\alpha}^\star} )^\circ \hookrightarrow \Omega_{\alpha}$ the inclusion of the interior of the supports. Consider the morphism $F_{\alpha} \to F_{\alpha}^{\varphi_{\alpha}^\star}$ under the microlocal restriction $\Sh(M) \to \msh(S^*M) \to \msh(\Lambda)$. 
    On $\mathrm{supp}(H_{\varphi_{\alpha}^\star} )^\circ$, we know that $m_\Lambda(F_{\alpha}^{\varphi_{\alpha}^\star}) = 0$ since the singular support is pushed off from $\Lambda$, and thus the morphism is zero microlocally. On the complement of $\mathrm{supp}(H_{\varphi_{\alpha}^\star} )^\circ$, we know that $m_\Lambda(F_{\alpha}^{\varphi_{\alpha}^\star}) = m_\Lambda(F_{\alpha})$, and thus the morphism is the identity microlocally. 
    Therefore,
    $$ m_{\Lambda \cap \Omega_{\alpha}} \mathrm{Fib}\Big( F_{\alpha} \to F_{\alpha}^{\varphi_{\alpha}^\star} \Big) = r_{\alpha !} r_{\alpha}^* F_{\alpha}.$$
    Choose the collections $\varphi_\alpha^\star \in W^+(\Omega_\alpha^c)$ such that $H_{\varphi_{\alpha_1 \dots \alpha_k}^\star} = \min_{1\leq i\leq k}H_{\varphi_\alpha^\star}$. Since $\mathrm{supp}(H_{\varphi_{\alpha_1 \dots \alpha_k}^\star} )^\circ = \bigcup_{1 \leq i \leq k} \mathrm{supp}(H_{\varphi_{\alpha_i}^\star})^\circ$, by taking the colimit with respect to the \v{C}ech cover, we have
    $$m_{\Lambda} \circ m_{\Lambda \subseteq \Lambda \cup T_{2\epsilon}(\Lambda)}^l(F) = \clmi{\alpha \in I, \varphi_\alpha^\star \in W^+(\Omega_\alpha^c)^\star}\, r_{\alpha !} r_\alpha^* F_\alpha = F.$$
    This proves full faithfulness of the doubling functor.
\end{proof}

    We can conclude that the two constructions of doubling functors in Section \ref{sec:doubling-local} and \ref{sec:double-ad-micro} are identical. By Theorem \ref{thm:doubling_l_ad_ff}, we know that there is a Hamiltonian such that $m_{\Lambda \cup T_{2\epsilon}(\Lambda)}^l = \widetilde{T}_\epsilon \circ w_\Lambda^l$. By Theorem \ref{thm:doubling_ad} and Remark \ref{rem:multi-component-ad}, we know that for the doubling functor $w_\Lambda$ defined in Section \ref{sec:doubling-local}, there is also a Hamiltonian such that $m_{\Lambda \cup T_{2\epsilon}(\Lambda)}^l = {T}_\epsilon \circ w_\Lambda[-1]$. Therefore, they are identical up to a Hamiltonian isotopy action and degree shifting.

    Similarly, given any $F \in \msh_\Lambda(\Lambda)$, consider the open coverings $\{U_\alpha\}_{\alpha \subseteq I}$, $\{\Omega_\alpha\}_{\alpha \subseteq I}$, the collection of sheaves $F_\alpha \in \Sh_\Lambda(U_\alpha)$, and the collection of positive wrappings $\varphi_\alpha^\star \in W^+(\Omega_\alpha^c)^\star$ such that $\max_{\alpha \subseteq I} H_{\varphi_\alpha^\star}|_{\Omega_\alpha'} = \epsilon H_T$. We can define
    $$w_\Lambda^r(F) = \lmi{\alpha \subseteq I, \varphi_\alpha^\star \in W^+(\Omega_\alpha^c)^\star} j_{\alpha *} \mathrm{Cofib}\big(F_\alpha^{(\varphi_\alpha^\star)^{-1}} \to F_{\alpha} \big).$$ 
    Using the same argument, we can prove the following result:

\begin{theorem}\label{thm:doubling_r_ad_twocopy}
    For $\Lambda \subseteq S^{*}M$ a compact subanalytic Legendrian, we have
    \begin{align*}
    m_{\Lambda \subseteq T_{-2\epsilon}(\Lambda) \cup \Lambda}^r(F) = \wrap_{T_{-2\epsilon}(\Lambda) \cup \Lambda}^- \circ w_\Lambda^r(F).
    \end{align*}
    Moreover, the right doubling is fully faithful
    $$w_{\Lambda}^r: \msh_\Lambda(\Lambda) \hookrightarrow \Sh_{T_{-\epsilon}(\Lambda) \cup \Lambda}(M)$$
    and there is a nonnegative Hamiltonian $\widetilde{T}_t: S^*M \to S^*M$ so that $m_{\Lambda \subseteq T_{-2\epsilon}(\Lambda) \cup \Lambda}^r = \widetilde{T}_{-\epsilon} \circ w_\Lambda^r(F).$
\end{theorem}

    We can prove that the left and right doubling functor that we defined are related in the following way, which will thus give an alternative proof of Theorem \ref{thm:doubling_ad}.

\begin{theorem}\label{thm:l_doubling=r_doubling}
    For $\Lambda \subseteq S^{*}M$ a compact subanalytic Legendrian, there is a natural isomorphism
    $$T_{\epsilon/2} \circ w_\Lambda^r (F) = T_{\epsilon/2} \circ w_\Lambda^l (F)[1].$$
\end{theorem}

\begin{lemma}\label{lem:negative-wrap-positive-doubling}
    For $\Lambda \subseteq S^{*}M$ a compact subanalytic Legendrian, we have
    $$\wrap_{T_{-\epsilon}(\Lambda) \cup \Lambda}^- \circ w_\Lambda^l (F) = 0.$$
\end{lemma}
\begin{proof}
    By Lemma \ref{lem:double-glue-ss}, we know that $w_\Lambda^l (F) \in \Sh_{\Lambda \cup T_\epsilon(\Lambda)}(M)$. For $k \in \NN$, let $\widetilde{T}_{-\epsilon+1/k}: S^*M \to S^*M$ be a sequence of nonpositive Hamiltonian flows such that $\widetilde{T}_{-\epsilon+1/k}(\Lambda) = \Lambda$ and $\widetilde{T}_{-\epsilon+1/k}\circ T_\epsilon(\Lambda) = T_{1/k}(\Lambda)$. Then this defines a cofinal wrapping from $\Lambda \cup T_{\epsilon}(\Lambda)$ to $\Lambda$. By Lemma \ref{lem:double-glue-ss} and the uniqueness theorem of GKS sheaf quantizations, we know that the sheaf defined by the cofinal sequence of wrapping is exactly
    $$\widetilde{T}_{-\epsilon+1/k} \circ w_\Lambda^l (F) = \clmi{\alpha \in I}\,j_{\alpha !} \mathrm{Cofib}(F_\alpha \to F_\alpha^{(\varphi_\alpha^\star)^{1/k}}).$$
    Follwing the argument in Lemma \ref{lem:nearby_cycle} or \cite[Theorem 5.15]{Kuo-wrapped-sheaves}, we can realize $\wrap_{T_{-\epsilon}(\Lambda) \cup \Lambda}^- \circ w_\Lambda^l (F)$ as the nearby cycle sheaf of
    $\mathrm{colim}_{\alpha \subseteq I}\,j_{\alpha !} \mathrm{Cofib}(q^*F_\alpha \to K(\Phi_\alpha) \circ F_\alpha)|_{M \times (0, 1]}.$
    This therefore implies 
    $$\wrap_{T_{-\epsilon}(\Lambda) \cup \Lambda}^- \circ w_\Lambda^l (F) = \clmi{\alpha \subseteq I}\,j_{\alpha !} \mathrm{Cofib}(q^*F_\alpha \to K(\Phi_\alpha) \circ F_\alpha)|_{M \times 0} = 0,$$
    which completes the proof.
\end{proof}

\begin{proof}[Proof of Theorem \ref{thm:l_doubling=r_doubling}]
    By Lemma \ref{lem:double-glue-ss} and Lemma \ref{lem:nearby_cycle}, we know that $\wrap_{T_{-2\epsilon}(\Lambda) \cup \Lambda}^-$ and $\wrap_{\Lambda \cup T_{2\epsilon}(\Lambda)}^+$ are fully faithful on the essential images of $w_\Lambda^r$ and $w_\Lambda^l$ and are realized by a single Hamiltonian flow. Therefore, by Theorem \ref{thm:doubling_l_ad_twocopy} and \ref{thm:doubling_r_ad_twocopy}, 
    it suffices to show that there exists a natural isomorphism
    $$T_{\epsilon} \circ m_{\Lambda \subseteq T_{-2\epsilon}(\Lambda) \cup \Lambda}^r (F) = T_{\epsilon} \circ m_{\Lambda \subseteq \Lambda \cup T_{2\epsilon}(\Lambda)}^l (F)[1].$$

    First, we construct a natural morphism as follows. Consider the counit of the adjunction $\mathrm{id} \to m_{\Lambda} m_{\Lambda \subseteq \Lambda \cup T_{2\epsilon}(\Lambda)}^l$, this induces a natural morphism
    $$m_{\Lambda \subseteq T_{-2\epsilon}(\Lambda) \cup \Lambda}^r \to m_{\Lambda \subseteq T_{-2\epsilon}(\Lambda) \cup \Lambda}^r m_{\Lambda} m_{\Lambda \subseteq \Lambda \cup T_{2\epsilon}(\Lambda)}^l.$$
    Then consider the unit of the adjunction $\mathrm{id} \to m_{\Lambda \subseteq T_{-2\epsilon}(\Lambda) \cup \Lambda}^r m_{\Lambda}$. By Theorem \ref{thm:wrap-local-to-global-2}, the fiber is 
    $$S_{\Lambda\subseteq T_{-2\epsilon}(\Lambda) \cup \Lambda}^- \coloneqq \wrap_{T_{-2\epsilon}(\Lambda) \cup \Lambda}^- \widetilde{T}_{-\epsilon}: \Sh(M) \to \Sh_{T_{-2\epsilon}(\Lambda) \cup \Lambda}(M),$$
    where $\widetilde{T}_{\pm \epsilon}: S^*M \to S^*M$ is the Hamiltonian flow compactly supported away from $T_{-2\epsilon}(\Lambda)$ that acts as the Reeb flow on $\Lambda$ in the neighbourhood. Thus by the fiber sequence Lemma \ref{lem:left_right_adjoints_of_microlocalization} and Theorem \ref{thm:wrap-local-to-global-2} we have a natural morphism 
    $$m_{\Lambda \subseteq T_{-2\epsilon}(\Lambda) \cup \Lambda}^r m_{\Lambda} \to S_{\Lambda\subseteq T_{-2\epsilon}(\Lambda) \cup \Lambda}^-[1].$$
    whose fiber is $\wrap_{T_{-2\epsilon}(\Lambda) \cup \Lambda}^-: \Sh(M) \to \Sh_{T_{-2\epsilon}(\Lambda) \cup \Lambda}(M)$. This induces the natural morphism
    $$m_{\Lambda \subseteq T_{-2\epsilon}(\Lambda) \cup \Lambda}^r m_{\Lambda} m_{\Lambda \subseteq \Lambda \cup T_{2\epsilon}(\Lambda)}^l \to S_{\Lambda\subseteq T_{-2\epsilon}(\Lambda) \cup \Lambda}^- m_{\Lambda \subseteq \Lambda \cup T_{2\epsilon}(\Lambda)}^l [1],$$
    whose fiber is $\wrap_{T_{-2\epsilon}(\Lambda) \cup \Lambda}^- \circ m_{\Lambda \subseteq \Lambda \cup T_{2\epsilon}(\Lambda)}^l$. Without loss of generality, assume $\widetilde{T}_{\pm \epsilon}: S^*M \to S^*M$ acts as the Reeb flow ${T}_{\pm \epsilon}: S^*M \to S^*M$ in the complement of a neighbourhood of $T_{-2\epsilon}(\Lambda)$. We get ${T}_{-\epsilon} \circ m_{\Lambda \subseteq \Lambda \cup T_{2\epsilon}(\Lambda)}^l (F) \in \Sh_{T_{-\epsilon}(\Lambda) \cup T_{\epsilon}(\Lambda)}(M)$. For $k \in \NN$, the sequence of Reeb flows $T_{-\epsilon+1/k}$ gives a cofinal wrapping from $T_{-\epsilon}(\Lambda) \cup T_{\epsilon}(\Lambda)$ to $T_{-2\epsilon}(\Lambda) \cup \Lambda$. Hence we have
    $$S_{\Lambda}^- m_{\Lambda \subseteq \Lambda \cup T_{2\epsilon}(\Lambda)}^l(F) = \wrap_{T_{-2\epsilon}(\Lambda) \cup \Lambda}^- T_{-\epsilon} \circ m_{\Lambda \subseteq \Lambda \cup T_{2\epsilon}(\Lambda)}^l(F) = T_{-2\epsilon} \circ m_{\Lambda \subseteq \Lambda \cup T_{2\epsilon}(\Lambda)}^l(F).$$
    Thus by composing the two natural morphisms, we get a natural morphism
    $$m_{\Lambda \subseteq T_{-2\epsilon}(\Lambda) \cup \Lambda}^r \to m_{\Lambda \subseteq T_{-2\epsilon}(\Lambda) \cup \Lambda}^r m_{\Lambda} m_{\Lambda \subseteq \Lambda \cup T_{2\epsilon}(\Lambda)}^l \to T_{-2\epsilon} \circ m_{\Lambda \subseteq \Lambda \cup T_{2\epsilon}(\Lambda)}^l [1].$$
    
    Consider the natural isomorphism given by Theorem \ref{thm:doubling_l_ad_ff} that $F \simeq m_{\Lambda} m_{\Lambda \subseteq \Lambda \cup T_{2\epsilon}(\Lambda)}^l (F)$. This implies that
    $$m_{\Lambda \subseteq T_{-2\epsilon}(\Lambda) \cup \Lambda}^r(F) \simeq m_{T_{-2\epsilon}(\Lambda) \cup \Lambda}^r m_{\Lambda} m_{\Lambda \subseteq \Lambda \cup T_{2\epsilon}(\Lambda)}^l (F).$$
    By Lemma \ref{lem:negative-wrap-positive-doubling}, we know that
    $\wrap_{T_{-2\epsilon}(\Lambda) \cup \Lambda}^- \circ m_{\Lambda \subseteq \Lambda \cup T_{2\epsilon}(\Lambda)}^l \simeq 0$. Therefore, 
    $$m_{\Lambda \subseteq T_{-2\epsilon}(\Lambda) \cup \Lambda}^r m_{\Lambda} m_{\Lambda \subseteq \Lambda \cup T_{2\epsilon}(\Lambda)}^l(F) \simeq T_{-2\epsilon} \circ m_{\Lambda \subseteq \Lambda \cup T_{2\epsilon}(\Lambda)}^l(F) [1].$$
    Therefore, we can conclude that $T_{\epsilon} \circ m_{\Lambda \subseteq T_{-2\epsilon}(\Lambda) \cup \Lambda}^r (F) = T_{\epsilon} \circ m_{\Lambda \subseteq \Lambda \cup T_{2\epsilon}(\Lambda)}^l (F).$
\end{proof}

\subsection{Relative doubling and fiber sequence}
    In this last section, we generalize the above results on fiber sequences and doubling functor to the microlocalization along a general open subset $\Omega$ inside $\Lambda \subseteq S^*M$. Since the arguments will be the same as (either approach in) the previous sections, we will not write down detailed proofs again.

\begin{definition}
    Let $\Lambda \subseteq S^*M$ be a compact subanalytic Legendrian, $\Omega \subseteq \Lambda$ be an open subset and $\tilde{\Omega}^c \subseteq S^*M$ be an open neighbourhood of $\Omega^c \subseteq \Lambda$. Define a smooth cut-off function $\hat H: S^*M \to [0, 1]$ such that 
    $$\hat H^{-1}(0) \cap \Lambda = \Lambda \setminus \Omega, \;\; \hat H^{-1}(1) \subseteq \Omega^c.$$
    using the smooth Urysohns' lemma. Let the non-negative contact flow $\hat T_{t}: S^*M \to S^*M$ be the flow defined by the contact Hamiltonian $\hat H$ such that $\alpha(\partial_t \hat T_t) = \hat H$.
\end{definition}

\begin{theorem}[Relative Sato-Sabloff fiber sequence]
    Let $\Lambda, \Lambda' \subseteq S^*M$ be a compact subanalytic Legendrian and $\Omega \subseteq \Lambda \cap \Lambda'$ be an open subset. Let $\hat T_t: S^*M \to S^*M$ be a non-negative contact flow such that $\hat H^{-1}(0) \cap (\Lambda \cap \Lambda') = (\Lambda \cap \Lambda') \setminus \Omega$ and 
    $$\Lambda \cap \hat T_\epsilon(\Lambda') = \varnothing$$
    for any small $\epsilon \neq 0$. Then for $F \in \Sh_\Lambda(M)$ and $G \in \Sh_{\Lambda'}(M)$, there is a fiber sequence
    $$\Hom(F, \hat T_{-\epsilon}(G)) \xrightarrow{c} \Hom(F, \hat T_{\epsilon}(G)) \to \Gamma(\Omega, \mhom(F, G)).$$
\end{theorem}
\begin{remark}
    Let $T_t: S^*M \to S^*M$ be a positive contact flow such that $T_t = \hat T_t$ away from a neighbourhood of $(\Lambda \cap \Lambda') \setminus \Omega$. Then by Sato-Sabloff fiber sequence, we can show the following commutative diagram of fiber sequences
    $$\xymatrix{
    \Hom(F, \hat T_{-\epsilon}(G)) \ar[r] \ar[d] & \Hom(F, \hat T_{\epsilon}(G)) \ar[r] \ar[d] & \Gamma(\Omega, \mhom(F, G)) \ar[d] \\
    \Hom(F, T_{\epsilon}(G)) \ar[r] \ar[d] & \Hom(F, T_{\epsilon}(G)) \ar[r] \ar[d] & \Gamma(\Lambda \cap \Lambda', \mhom(F, G)) \ar[d] \\
    \Gamma((\Lambda \cap \Lambda') \setminus \Omega, \mhom(F, G)) \ar[r] & 0 \ar[r] & \Gamma((\Lambda \cap \Lambda') \setminus \Omega, \mhom(F, G))[1].
    }$$ 
    Thus, we can deduce the relative Sato-Sabloff fiber sequence (the vertical sequence in the left column) for open subsets from the Sato-Sabloff sequence for closed subsets in Theorem \ref{thm:sato-sab} (the two horizontal sequences in the first and second rows).
\end{remark}

    Let $\Lambda = \Lambda'$ be the same subanalytic Legendrian. As a consequence, we can obtain the relative doubling functor on microsheaves on an open subset of the subanalytic Legendrian $\msh_\Lambda(\Omega)$.

\begin{theorem}\label{thm:rel-double}
    Let $\Lambda \subseteq S^{*}M$ be a compact subanalytic Legendrian, $\Omega \subseteq \Lambda$ an open subset, and $c(\Omega)$ be the length of the shortest Reeb chord on $\Omega$ with respect to the non-negative Hamiltonian flow $\hat T_t$ supported on $\Omega$. Then for $0 < \epsilon < c(\Lambda)/2$, there is a fully faithful functor
    $$w_\Omega^l: \msh_\Lambda(\Omega) \hookrightarrow \Sh_{\overline{\Omega} \cup \hat T_{\epsilon}(\overline{\Omega})}(M), \; w_\Omega^r: \msh_\Lambda(\Omega) \hookrightarrow \Sh_{\overline{\Omega} \cup \hat T_{-\epsilon}(\overline{\Omega})}(M)$$
    that are related by $\hat T_\epsilon: \Sh_{\overline{\Omega} \cup \hat T_{-\epsilon}(\overline{\Omega})}(M) \xrightarrow{\sim} \Sh_{\overline{\Omega} \cup \hat T_{\epsilon}(\overline{\Omega})}(M)$. Moreover, there are equivalences
    $$m_\Omega^l = \wrap_\Lambda^+ \circ w_\Omega^l[-1], \;\; m_\Omega^r = \wrap_\Lambda^- \circ w_\Omega^r: \msh_\Lambda(\Omega) \rightarrow \Sh_\Lambda(M)$$
    where $\wrap_\Lambda^\pm$ are the functors given in Theorem \ref{w=ad}.
\end{theorem}
\begin{remark}\label{rem:rel-double}
    For the closed subset $\Omega^c \subseteq \Lambda$, let $\hat T_t$ be the non-negative Hamiltonian flow supported on $\Omega$, and $T_t: S^*M \to S^*M$ be a positive contact flow such that $T_t = \hat T_t$ away from an open neighbourhood $\tilde{\Omega}^c$ of $\Omega^c$. On the contrary, we obtain two doubling functors that are a priori different from each other
    $$w_{\Omega^c}^l: \msh_\Lambda(\Omega^c) \hookrightarrow \Sh_{\hat T_\epsilon(\tilde{\Omega}^c) \cup T_{\epsilon}(\tilde{\Omega}^c)}(M), \; w_{\Omega^c}^r: \msh_\Lambda(\Omega^c) \hookrightarrow \Sh_{\hat T_{-\epsilon}(\tilde{\Omega}^c) \cup T_{-\epsilon}(\tilde{\Omega}^c)}(M)$$
    such that $m_{\Omega^c}^l = \wrap_\Lambda^+ \circ w_{\Omega^c}^l[-1]$ and $m_{\Omega^c}^r = \wrap_\Lambda^- \circ w_{\Omega^c}^r$. When there is an open neighbourhood $\tilde{\Omega}^c$ of $\Omega^c$ such that $\msh_\Lambda(\tilde{\Omega}^c) \xrightarrow{\sim} \msh_\Lambda(\Omega^c)$, it follows that $m_{\tilde{\Omega}^c}^{l} = m_{\Omega^c}^{l}$ and $m_{\tilde{\Omega}^c}^{r} = m_{\Omega^c}^{r}$. In particular, when $\Omega^c$ is a smooth point on some Legendrian stratum $\tilde{\Omega} \subseteq \Lambda$, this construction recovers the sheaf theoretic linking disks at the Legendrian point \cite[Definition 4.14]{Kuo-wrapped-sheaves}.
\end{remark}

\begin{figure}
    \centering
    \includegraphics[width=0.8\textwidth]{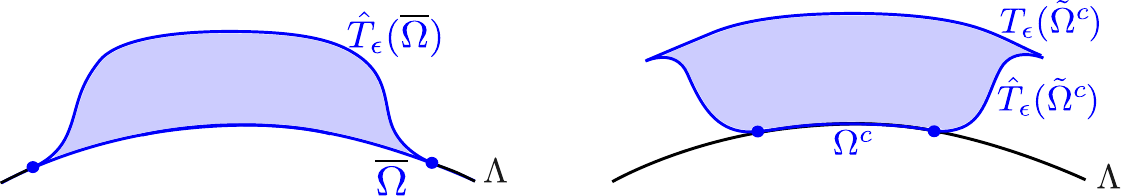}
    \caption{The relative doubling functor for an open subset $\Omega \subseteq \Lambda$ and for a closed subset $\Omega^c \subseteq \Lambda$ in Theorem \ref{thm:rel-double} and Remark \ref{rem:rel-double}.}
\end{figure}

\section{Spherical adjunction from microlocalization}

    With the preparation in the previous sections, we are able to prove Theorem \ref{thm:main}. Our main result in the section is the following theorem.
    
\begin{theorem}\label{thm:main-fun}
    Let $\Lambda \subset S^{*}M$ be a compact subanalytic Legendrian full stop or swappable stop. Then the microlocalization functor
    $$m_\Lambda: \Sh_\Lambda(M) \rightarrow \msh_\Lambda(\Lambda)$$
    is a spherical functor.
\end{theorem}

    Theorem \ref{thm:main} is going to be a formal consequence of Theorem \ref{thm:main-fun}, as we will discuss in Section \ref{sec:spherical}. In particular, this will imply that the left adjoint of microlocalization is also a spherical functor.
    
\begin{corollary}\label{cor:main-fun}
    Let $\Lambda \subseteq S^{*}M$ be a compact subanalytic Legendrian full stop or swappable stop. Then the left adjoint of the microlocalization functor
    $$m_\Lambda^l: \msh_\Lambda(\Lambda) \rightarrow \Sh_\Lambda(M)$$
    is a spherical functor.
\end{corollary}

    In Section \ref{sec:serre-proper}, we will combine the result with the Sabloff-Serr{e} duality in Proposition \ref{prop:sab-serre}, and show that when $\Lambda \subseteq S^*M$ is a full stop or swappable stop, the spherical cotwist, after tensoring with the dualizing sheaf, $S_\Lambda^- \otimes \omega_M$, is the Serr\'{e} functor on the subcategory $\Sh^b_\Lambda(M)_0$ of compactly supported sheaves with perfect stalks. This proves a folklore conjecture on Fukaya-Seidel categories when the ambient manifold is a cotangent bundle.

\subsection{Spherical adjunction and spherical functors}\label{sec:spherical}
    First of all, we recall the definition of spherical adjunctions in Dykerhoff-Kapranov-Schechtman-Soibelman \cite{SphericalInfty} in the setting of stable $\infty$-categories.
    
\begin{definition}[\cite{SphericalInfty}*{Definition 1.4.8}]\label{def:sphericalad}
    Let ${\sA, \sB}$ be stable ($\infty$-)categories and
    $$F: {\sA} \leftrightharpoons {\sB} : F^l$$
    be an adjunction of $\infty$-functors. Let $T'$ and $S'$ be the functors that fit into the fiber sequences
    $$T' \rightarrow \mathrm{id}_{\sB} \rightarrow F \circ F^l, \,\, F^l \circ F \rightarrow \mathrm{id}_{\sA} \rightarrow S'.$$
    Then $F: {\sA} \leftrightharpoons {\sB}: F^l$ is called a spherical adjunction if $T'$ and $S'$ are autoequivalences.
\end{definition}

    Given a spherical adjunction $F^l \dashv F$, one can in fact show that both $F$ and $F^l$ are spherical functors in the sense of Anno-Logvinenko \cite{Spherical}. We recall the definition of spherical functors in the setting of dg categories \cite{Spherical} and in the general case \cite{SphericalInfty,SphericalChrist}.

\begin{definition}\label{def:spherical}
    Let ${\sA, \sB}$ be stable ($\infty$-)categories and $F: {\sA} \rightarrow {\sB}$ an ($\infty$-)functor, with left and right adjoints $F^l$ and $F^r$. Let the spherical twist $T$, dual twist $T'$, cotwist $S$ and dual cotwist $S'$ be the functors that fit into the fiber sequences
    \[\begin{split}
    F \circ F^! \rightarrow \mathrm{id}_{\sB} \rightarrow T,\, & \, T' \rightarrow \mathrm{id}_{\sB} \rightarrow F \circ F^*, \\
    S \rightarrow \mathrm{id}_{\sA} \rightarrow F^! \circ F,\, & \, F^* \circ F \rightarrow \mathrm{id}_{\sA} \rightarrow S'.
    \end{split}\]
    Then $F$ is a spherical functor if the following conditions hold:
    \begin{enumerate}
      \item the spherical twist $T$ is an autoequivalence;
      \item the spherical cotwist $S$ is an autoequivalence;
      \item the composition $F^l \circ T[-1] \rightarrow F^l \circ F \circ F^r \rightarrow F^r$ is an isomorphism;
      \item the composition $F^r \rightarrow F^r \circ F \circ F^l \rightarrow S \circ F^l[1]$ is an isomorphism.
    \end{enumerate}
\end{definition}

\begin{proposition}[\cite{SphericalInfty}*{Corollary 2.5.13}]
    Let $F: {\sA} \leftrightharpoons {\sB} : F^l$ be a spherical adjunction. Then both $F$ and $F^l$ are spherical functors.
\end{proposition}
\begin{remark}\label{rem:sphere-ad}
    Given a spherical adjunction $F^l \dashv F$, let $T$ be the inverse of $T'$ and $S$ the inverse of $S'$. One can construct the right adjoint of $F$ by setting $F^r = F^l \circ T[-1]$. In fact, any spherical functor has iterated left and right adjoints of any order.
\end{remark}

    Therefore, to prove a spherical adjunction $F \vdash F^l$, it suffices to show that either of the functors is a spherical functor as in Definition \ref{def:spherical}. Moreover, the following theorem shows that it suffices to prove any two out of the four conditions.

\begin{theorem}[Anno-Logvinenko \cite{Spherical}, Christ \cite{SphericalChrist}]
    Let ${\sA, \sB}$ be stable categories, and $F: {\sA} \rightarrow {\sB}$ a functor satisfying any two of the four conditions in Definition \ref{def:spherical}. Then $F$ is a spherical functor. Moreover, $T, T'$ and $S, S'$ are inverse autoequivalences.
\end{theorem}

    From the discussion above, we know that in order to prove Theorem \ref{thm:main}, it suffices to prove Theorem \ref{thm:main-fun} stated at the beginning of the section.

\subsection{Natural transform between adjoints}\label{sec:natural-trans}
    Given the adjoint functors and the candidate cotwist in Section \ref{sec:doubling}, we will investigate the relation between the left and right adjoints via the algebraically defined natural transformation by the cotwist
    $$m_\Lambda^r \rightarrow m_\Lambda^r \circ m_\Lambda \circ m_\Lambda^l  \rightarrow {S}_\Lambda^- \circ m_\Lambda^l[1].$$
    The composition of the natural transformations should induce an equivalence in order for microlocalization to be a spherical functor, as stated in Definition \ref{def:spherical} Condition~(4).
    
    For $F \in \Sh_\Lambda(M)$ and $G \in \msh_\Lambda(\Lambda)$, we thus need to prove that the algebraically defined natural morphism
    $$\Hom(F, m_\Lambda^r(G)) \rightarrow  \Hom(F, m_\Lambda^r  m_\Lambda  m_\Lambda^l(G)) \rightarrow  \Hom(F, {S}_\Lambda^- m_\Lambda^l(G)[1])$$
    is an equivalence. On the other hand, Proposition \ref{w=ad} and \ref{prop:hom_w_pm} imply that
    \begin{align*}
        \Hom(F, m_\Lambda^r(G)) = \Hom(T_\epsilon(F), w_\Lambda(G)), \;\;
        \Hom(F, S_\Lambda^- m_\Lambda^l(G)[1]) = \Hom(T_\epsilon(F), \wrap_\Lambda^+ w_\Lambda(G)).
    \end{align*}
    Positive isotopies will then induce a geometrically defined natural morphism
    $$\Hom(T_\epsilon(F), w_\Lambda(G)) \rightarrow \Hom(T_\epsilon(F), \wrap_\Lambda^+  w_\Lambda(G)).$$
    Our main result in this section claims that the algebraically defined natural morphism induces an isomorphism if and only if the geometrically defined natural morphism induces an isomorphism.
    
\begin{proposition}\label{prop:natural-trans}
    Let $\Lambda \subseteq S^{*}M$ be a compact subanalytic Legendrian. Then for  ${F} \in \Sh_\Lambda(M)$ and ${G} \in \msh_\Lambda(\Lambda)$, the natural morphism induced by adjunctions
    $$\Hom(F, m_\Lambda^r(G)) \rightarrow  \Hom(F, m_\Lambda^r  m_\Lambda m_\Lambda^l(G)) \rightarrow  \Hom(F, {S}_\Lambda^-  m_\Lambda^l(G)[1])$$
    is an isomorphism if and only if the natural morphism induced by positive isotopies
    $$\Hom(T_\epsilon(F), w_\Lambda(G)) \rightarrow \Hom(T_\epsilon(F), \wrap_\Lambda^+  w_\Lambda(G))$$
    is an isomorphism.
\end{proposition}

    We need to unpack the algebraic adjunctions between microlocalization and its left and right adjoints using results on the doubling functor.
    
    Firstly, we consider the natural transformation to the cotwist $m_\Lambda^r \circ m_\Lambda \rightarrow S_\Lambda^-[1]$. The following lemma follows directly from Corollary \ref{cor:exact-tri} that there is a fiber sequence $T_{-\epsilon} \rightarrow T_\epsilon \rightarrow w_\Lambda \circ m_\Lambda$.
    
\begin{lemma}\label{lem:natural-trans1}
    Let $\Lambda \subseteq S^{*}M$ be a compact subanalytic Legendrian. Then for ${F} \in \Sh_\Lambda(M)$ and ${G} \in \msh_\Lambda(\Lambda)$, there is a commutative diagram induced by natural transformations
    \[\xymatrix@C=5mm{
    \Hom(F, m_\Lambda^r  m_\Lambda  m_\Lambda^l(G)) \ar[r] \ar[d]^{\rotatebox{90}{$\sim$}} & \Hom(F, {S}_\Lambda^- m_\Lambda^l(G)[1]) \ar[d]^{\rotatebox{90}{$\sim$}} \\
    \Hom(w_\Lambda m_\Lambda({F}), \mathfrak{W}_\Lambda^+  w_\Lambda({G})) \ar[r] & \Hom(T_{\epsilon}({F}), \mathfrak{W}_\Lambda^+  w_\Lambda({G})).
    }\]
\end{lemma}
\begin{proof}
    Consider ${F} \in \Sh_\Lambda(M)$ and ${G} \in \msh_\Lambda(\Lambda)$. Theorem \ref{w=ad} implies that it suffices for us to show the following diagram
    \[\xymatrix@C=2.5mm{
    \Hom({F}, w_\Lambda  m_\Lambda  \circ \mathfrak{W}_\Lambda^+  w_\Lambda({G})[-1]) \ar[r] \ar[d]^{\rotatebox{90}{$\sim$}} & \Hom({F}, T_{-\epsilon}(  \mathfrak{W}_\Lambda^+  w_\Lambda({G})))  \ar[d]^{\rotatebox{90}{$\sim$}} \\
    \Hom(w_\Lambda  m_\Lambda({F}), \mathfrak{W}_\Lambda^+  w_\Lambda({G})) \ar[r] & \Hom(T_\epsilon({F}), \mathfrak{W}_\Lambda^+  w_\Lambda({G})).
    }\]
    Since the two horizontal morphisms are induced by the transformation $w_\Lambda \circ m_\Lambda[-1] \rightarrow T_{-\epsilon}$ and respectively $T_\epsilon \rightarrow w_\Lambda \circ m_\Lambda$ in Corollary \ref{cor:exact-tri}, we can conclude that the diagram commutes.
\end{proof}
    
    Secondly, we need to consider the unit $\mathrm{id} \rightarrow m_\Lambda \circ m_\Lambda^l$, which is slightly more difficult. The following lemma relies on Corollary \ref{cor:exact-tri}, and the fact that the adjunction in Theorem \ref{thm:doubling_ad} factors through the doubling functor by computation in Theorem \ref{thm:doubling}:
    $$\Gamma(\Lambda, \mhom(m_\Lambda(F), G)) \xrightarrow{\sim} \Hom(w_\Lambda m_\Lambda(F), w_\Lambda(G)) \xrightarrow{\sim} \Hom(T_\epsilon(F), w_\Lambda(G)).$$

\begin{lemma}\label{lem:natural-trans2}
    Let $\Lambda \subseteq S^{*}M$ be a compact subanalytic Legendrian. Then for  ${F} \in \Sh_\Lambda(M)$ and ${G} \in \msh_\Lambda(\Lambda)$, there is a commutative diagram induced by natural transformations
    \[\xymatrix@C=5mm{
    \Hom(F, m_\Lambda^r(G)) \ar[r] \ar[d]^{\rotatebox{90}{$\sim$}} & \Hom(F, m_\Lambda^r  m_\Lambda  m_\Lambda^l(G)) \ar[d]^{\rotatebox{90}{$\sim$}} \\
    \Hom(T_\epsilon({F}), w_\Lambda({G})) \ar[r] & \Hom(w_\Lambda  m_\Lambda({F}), \mathfrak{W}_\Lambda^+  w_\Lambda({G})).
    }\]
\end{lemma}    
\begin{proof}
    Consider ${F} \in \Sh_\Lambda(M)$ and ${G} \in \msh_\Lambda(\Lambda)$. By Theorem \ref{w=ad}, it suffices to show that there is a commutative diagram
    \[\xymatrix@C=2mm{
    \Hom({F}, w_\Lambda({G})) \ar[r]  \ar[d]^{\rotatebox{90}{$\sim$}} & \Hom({F}, w_\Lambda  m_\Lambda \circ \mathfrak{W}_\Lambda^+  w_\Lambda({G})[-1]) \ar[d]^{\rotatebox{90}{$\sim$}} \\
    \Hom(T_\epsilon({F}), w_\Lambda({G})) \ar[r] & \Hom(w_\Lambda  m_\Lambda({F}), \mathfrak{W}_\Lambda^+  w_\Lambda({G})).
    }\]
    where the morphism on the top is induced by adjunction, and the morphism on the bottom is the composition
    $$\Hom(T_\epsilon({F}), w_\Lambda({G})) \xrightarrow{\sim}   \Hom(w_\Lambda m_\Lambda({F}), w_\Lambda({G})) \rightarrow \Hom(w_\Lambda  m_\Lambda({F}), \mathfrak{W}_\Lambda^+  w_\Lambda({G})).$$
    Consider the unit of the adjunction between $m_\Lambda$ and $m_\Lambda^l = \wrap_\Lambda^+ \circ w_\Lambda[-1]$ in Theorem \ref{thm:doubling_ad}. Then we know that the morphism on the top factors as
    \[\xymatrix{
    \Gamma(\Lambda, \mhom(m_\Lambda(F), G)) \ar[r] \ar[d]^{\rotatebox{90}{$\sim$}} & \Gamma(\Lambda, \mhom(m_\Lambda(F), m_\Lambda \circ \mathfrak{W}_\Lambda^+  w_\Lambda({G})[-1])) \ar[d]^{\rotatebox{90}{$\sim$}} \\
    \Hom(F, w_\Lambda(G)) \ar[r] & \Hom(F, w_\Lambda m_\Lambda \circ \mathfrak{W}_\Lambda^+  w_\Lambda({G})[-1]).
    }\]
    Since the top horizontal morphism factors through $\Hom(w_\Lambda  m_\Lambda({F}), \mathfrak{W}_\Lambda^+  w_\Lambda({G}))$, it suffices to show that the following composition factors through $\Gamma(\Lambda, \mhom(m_\Lambda(F), G))$
    $$\Hom(T_\epsilon({F}), w_\Lambda({G})) \xrightarrow{\sim}   \Hom(w_\Lambda m_\Lambda({F}), w_\Lambda({G})) \rightarrow \Hom(w_\Lambda  m_\Lambda({F}), \mathfrak{W}_\Lambda^+  w_\Lambda({G})).$$
    Since the adjunction in Theorem \ref{thm:doubling_ad} factors through the isomorphism of doubling functor in Theorem \ref{thm:doubling}
    $$\Hom(T_\epsilon({F}), w_\Lambda({G})) \xrightarrow{\sim}   \Hom(w_\Lambda m_\Lambda({F}), w_\Lambda({G})) \xrightarrow{\sim} \Gamma(\Lambda, \mhom(m_\Lambda(F), G)),$$
    we can conclude that the diagram above indeed commutes.
\end{proof}

\begin{proof}[Proof of Proposition \ref{prop:natural-trans}]
    We consider the following commutative diagram, where the horizontal morphisms are induced from the identity $w_\Lambda \circ m_\Lambda = \mathrm{Cofib}(T_{-\epsilon} \rightarrow T_\epsilon)$, and vertical morphisms are induced by positive isotopies
    \[\xymatrix@C=2.5mm{
    \Hom(T_\epsilon({F}), w_\Lambda({G})) \ar[r] \ar@{=}[d]& \Hom(w_\Lambda \circ m_\Lambda({F}), w_\Lambda({G})) \ar[d] \ar[r] & \Hom(T_\epsilon({F}), w_\Lambda({G})) \ar[d] \\
    \Hom(T_\epsilon({F}), w_\Lambda({G})) \ar[r] & \Hom(w_\Lambda \circ m_\Lambda({F}), \mathfrak{W}_\Lambda^+ \circ w_\Lambda({G})) \ar[r] & \Hom(T_\epsilon({F}), \mathfrak{W}_\Lambda^+ \circ w_\Lambda({G})).
    }\]
    Lemma \ref{lem:natural-trans1} and \ref{lem:natural-trans2} imply that the algebraically defined natural transformation of functors is an isomorphism if and only if the composition of morphisms in the second row is an isomorphism. 
    
    From the computation in Theorem \ref{thm:doubling}, compared with Theorem \ref{thm:doubling_ad}, we know that horizontal natural morphisms in the first row are isomorphisms
    $$\Hom(T_\epsilon({F}), w_\Lambda({G})) \xrightarrow{\sim} \Hom(w_\Lambda \circ m_\Lambda({F}), w_\Lambda({G})) \xrightarrow{\sim} \Hom(T_\epsilon({F}), w_\Lambda({G})).$$
    Therefore, the composition of horizontal morphisms in the second row is an isomorphism if and only if the geometrically defined vertical morphism in the last column is an isomorphism.
\end{proof}

\subsection{Criterion for spherical adjunction}\label{sec:sphere-crit}
    With the presence of the adjunction pairs, fiber sequences and natural transformations in the previous sections, in this section, we will study the spherical cotwist and its dual cotwist, and prove Condition~(2) \& (4) in Definition \ref{def:spherical} under geometric assumptions. As we will see in the proof, both Condition~(2) \& (4) in will rely on some full faithfulness of the wrapping functor $\mathfrak{W}_\Lambda^+$.

    In this section, we will use the word stop for a closed subanalytic Legendrian in $S^{*}M$ (meaning that the positive wrappings in $S^{*}M$ are stopped by the subanalytic Legendrian), which comes from the study of symplectic topology and wrapped Fukaya categories \cite{Sylvan,Ganatra-Pardon-Shende1}.

\subsubsection{Spherical adjunction from full stops} 
    We assume that $M$ is compact in this subsection. First, we introduce the notion of an algebraic full stop, which has been frequently used in wrapped Fukaya categories.

\begin{definition}
    Let $M$ be compact and $\Lambda \subseteq S^{*}M$ be a compact subanalytic Legendrian. Then $\Lambda$ is called a full stop if $ \Sh_\Lambda(M)$ is a proper category.
\end{definition}
\begin{remark}
    Recall that from Theorem \ref{thm:perfcompact} (\cite{Nadler-pants}*{Theorem 3.21} or \cite[Corollary 4.23]{Ganatra-Pardon-Shende3}), we know that
    $\Sh^b_\Lambda(M) = \Sh^{pp}_\Lambda(M).$
    From Proposition \ref{prop:smooth}, we know that in our case $\Sh^c_\Lambda(M)$ is smooth, which then implies Corollary \ref{cor:prop-in-perf} that
    $\Sh^b_\Lambda(M) \subseteq \Sh^c_\Lambda(M).$
    On the other hand, when $\Sh^c_\Lambda(M)$ is moreover proper, then we know that \cite[Lemma A.8]{Ganatra-Pardon-Shende3}
    $$\Sh^c_\Lambda(M) = \Sh^b_\Lambda(M).$$
    Conversely, when $M$ and $\Lambda$ are both compact, then if $\Sh^c_\Lambda(M) = \Sh^b_\Lambda(M)$, we can also tell that $\Sh^c_\Lambda(M)$ is proper using for example \cite[Lemma 4.2]{Ganatra-Pardon-Shende3}.
\end{remark}
 
\begin{example}\label{ex:triangle-full}
    Let $\mathcal{S} = \{X_\alpha\}_{\alpha \in I}$ be a subanalytic triangulation on $M$. Then the union of unit conormal bundles over all strata $\Lambda = N^*_\infty\mathcal{S} = \bigcup_{\alpha \in I}N^{*}_\infty{X_\alpha}$ defines a full stop \cite{Ganatra-Pardon-Shende3}*{Proposition 4.24}.
\end{example}

    We recall our notion of a geometric full stop in the introduction. For the definition of generalized linking disks, see for example \cite[Section 7.1]{Ganatra-Pardon-Shende3}.

\begin{definition}
    A closed subanalytic Legendrian $\Lambda \subseteq S^{*}M$ is called a geometric full stop if for a collection of generalized linking spheres $\Sigma \subseteq S^{*}M$ of $\Lambda$, there exists a compactly supported positive Hamiltonian on $S^{*}M \backslash \Lambda$ such that the Hamiltonian flow sends $D$ to an arbitrary small neighbourhood of $T_{-\epsilon}(\Lambda)$.
\end{definition}

    Following Ganatra-Pardon-Shende \cite{Ganatra-Pardon-Shende3}*{Proposition 6.7}, we prove that a geometric full stop is an algebraic full stop.

\begin{proposition}\label{prop:full-geo=>alg}
    Let $\Lambda \subseteq S^{*}M$ be a geometric full stop. Then $\Lambda$ is also an algebraic full stop.
\end{proposition}

    To prove the proposition, we recall the following definition by the first author in \cite{Kuo-wrapped-sheaves}. Let $M$ be compact, and $\widetilde{\wsh}_\Lambda(M)$ be the category of constructible sheaves with perfect stalks whose singular support is disjoint from $\Lambda$. Let $\mathcal{C}_\Lambda(M)$ be all continuation maps of positive isotopies supported away from $\Lambda$. Then the category of wrapped sheaves is \cite[Definition 4.1]{Kuo-wrapped-sheaves}
    $$\wsh_\Lambda(M) \coloneqq \widetilde{\wsh}_\Lambda(M)/\mathcal{C}_\Lambda(M).$$
    We have $\Hom_{\wsh_\Lambda(M)}(F, G) = \operatorname{colim}_{\varphi \in W^+(S^*M \backslash \Lambda)}\Hom(F, G^\varphi)$, and \cite[Theorem 1.3]{Kuo-wrapped-sheaves}
    $$\wrap_\Lambda^+: \wsh_\Lambda(M) \xrightarrow{\sim} \Sh^c_\Lambda(M).$$

\begin{proof}[Proof of Proposition \ref{prop:full-geo=>alg}] 
    We prove that $\wsh_\Lambda(M)$ is a proper category. Namely, for any ${F, G} \in \wsh_\Lambda(M)$,
    $$\Hom_{\wsh_\Lambda(M)}({F}, {G}) \in \cV_0.$$
    Indeed, note that $\ms^\infty({G}) \subseteq S^{*}M \backslash \Lambda$ is compact. Thus there exists a cofinal wrapping $\varphi_k \in W^+(S^{*}M \backslash \Lambda)$ such that $\varphi_{k}^1(\ms^\infty({G}))$ is contained in a neighbourhood of $\Lambda$, and is in particular away from $\ms^\infty({F})$. Therefore
    \[\begin{split}
    \Hom_{\wsh_\Lambda(M)}({F}, {G}) &= \clmi{\varphi \in W^+(S^{*}M \backslash \Lambda)}\Hom({F}, {G}^\varphi) \simeq \clmi{k \rightarrow \infty}\Hom({F}, {G}^{\varphi_k}) \\
    &\simeq \clmi{k \rightarrow \infty}\Hom({F}, {G}) = \Hom({F, G}) \in \cV_0,
    \end{split}\]
    which completes the proof for the first case of a geometric full stop.
\end{proof}

\begin{example}\label{ex:Lef-full}
    Let $\pi: T^*M \rightarrow \mathbb{C}$ be an exact symplectic Lefschetz fibration (whose existence is ensured by Giroux-Pardon \cite{GirouxPardon}), and $F = \pi^{-1}(\infty)$ a regular Weinstein fiber at infinity. Then the Lagrangian skeleton of the Weinstein manifold $\Lambda = \mathfrak{c}_F \subseteq S^{*}M$ defines a full stop \cite{Ganatra-Pardon-Shende2}*{Corollary 1.14} \& \cite{Ganatra-Pardon-Shende3}*{Proposition 6.7} under the equivalence between wrapped Fukaya categories and microlocal sheaf categories \cite{Ganatra-Pardon-Shende3}.

    Let $\Lambda = \Lambda_\Sigma^\infty \subseteq S^{*}T^n$ be the FLTZ skeleton 
    associated to the toric fan $\Sigma$ \cite{FLTZCCC,FLTZHMS,RSTZSkel}. Gammage-Shende (under an extra assumption) \cite{GammageShende} and Zhou (without extra assumptions) \cite{ZhouSkel} show that it is indeed the Lagrangian skeleton of a regular fiber of a symplectic fibration $\pi: T^*T^n \rightarrow \mathbb{C}$, which is expected to be a Lefschetz fibration when the mirror toric stack $\mathcal{X}_\Sigma$ is smooth. The fact that $\Lambda_\Sigma^\infty$ is a full stop (when $\mathcal{X}_\Sigma$ is smooth) is also independently proved by Kuwagaki using mirror symmetry \cite{KuwaCCC}.
\end{example}

    When $\Lambda \subseteq S^{*}M$ is a full Legendrian stop, we know that $\Sh_\Lambda(M) = \mathrm{Ind}(\Sh^b_\Lambda(M))$. Therefore we only focus on results on the small category $\Sh^b_\Lambda(M)$.
    
    To show Condition~(2) that $S_\Lambda^+$ and $S_\Lambda^-$ are equivalences, we appeal to Section \ref{sec:serre} where $S_\Lambda^-$ is shown to abide Serre duality (up to a twist) on $\Sh^b_\Lambda(M)$.

\begin{proposition}\label{prop:full-condition2}
    Let $\Lambda \subseteq S^{*}M$ be a compact full Legendrian stop where $M$ is a closed manifold. Then there is a pair of inverse autoequivalences
    $$S^+_\Lambda: \Sh^b_\Lambda(M) \rightleftharpoons \Sh^b_\Lambda(M): S^-_\Lambda.$$
\end{proposition} 
\begin{proof}
    For $F, G \in \Sh^b_\Lambda(M)$, we claim that
    $$\Hom(F, G) \rightarrow \Hom_{\wsh_\Lambda(M)}(T_\epsilon(F), T_\epsilon(G)).$$
    Indeed, consider a sequence of descending open neighbourhoods $\{\Omega_k\}_{k \in \NN}$ of $\Lambda \subset S^{*}M$ such that $\Omega_{k+1} \subseteq \overline{\Omega_k}$ and
    $\bigcap_{k \in \NN} \Omega_k = \Lambda.$
    Let the sequence cofinal wrapping be the Reeb flow $T_{1/t}$ on $S^*M \backslash \Omega_k$ and identity on $\Omega_{k+1}$. Then
    \begin{align*}
    \Hom_{\wsh_\Lambda(M)}(T_\epsilon(F), T_\epsilon(G))
    &= \clmi{\delta \rightarrow 0^+} \Hom(T_\delta(F), T_\epsilon(G)) \\
    &= \clmi{\delta \rightarrow 0^+} \Hom(F, T_{\epsilon - \delta}(G)).
    \end{align*}
    Then the right hand side is a constant since $\Hom(F, T_{\epsilon - \delta}(G)) \xleftarrow{\sim} \Hom(F,G)$ by the perturbation trick Proposition \ref{prop:hom_w_pm}.
    
    On the other hand, we also know that $\wrap_\Lambda^+: \wsh_\Lambda(M) \rightarrow \Sh_\Lambda^b(M)$ is an equivalence \cite[Theorem 1.3]{Kuo-wrapped-sheaves}. Therefore
    $$\Hom_w(T_\epsilon(F), T_\epsilon(G)) \simeq \Hom(\wrap_\Lambda^+ \circ T_\epsilon(F), \wrap_\Lambda^+ \circ T_\epsilon(G)) = \Hom(S_\Lambda^+(F), S_\Lambda^+(G)).$$
    Since $S_\Lambda^-$ is the right adjoint of $S_\Lambda^+$, we know that $S^-_\Lambda \circ S^+_\Lambda = \mathrm{id}_{\Sh^b_\Lambda(M)}$.
    
    Then consider $F, G \in \Sh^b_\Lambda(M)$. Sabloff-Serr{e} duality Proposition \ref{prop:sab-serre} implies that
    $$\Hom(S^-_\Lambda(F), S_\Lambda^-(G)) = \Hom(G \otimes \omega_M^{-1}, S^-_\Lambda(F))^\vee = \Hom(F, G).$$
    Then since $S_\Lambda^-$ is the right adjoint of $S_\Lambda^+$, we know that $S^-_\Lambda \circ S^+_\Lambda = \mathrm{id}_{\Sh^b_\Lambda(M)}$.
\end{proof}

    Next, we show Condition~(4) that there is a natural isomorphism of functors $m_\Lambda^r \xrightarrow{\sim} S^-_\Lambda \circ m_\Lambda^l[1]$, which again requires Serr{e} duality in Section \ref{sec:serre}.
    
\begin{proposition}\label{prop:full-condition4}
    Let $\Lambda \subset S^{*}M$ be a compact subanalytic full Legendrian stop. Then for any $F \in \Sh_\Lambda^b(M)$ and $G \in \msh^c_\Lambda(\Lambda)$ there is an isomorphism 
    $$\Hom(F, m_\Lambda^r(G)) \rightarrow \Hom(w_\Lambda \circ m_\Lambda(G), m_\Lambda^l(G)) \rightarrow \Hom(T_\epsilon(F), m_\Lambda^l(G))$$
\end{proposition}
\begin{proof}
    Let $\mu_i \in \msh_\Lambda^c(\Lambda)$ be the corepresentatives of microstalks at the point $p_i$ on the smooth stratum $\Lambda_i \subset \Lambda$, which (split) generate the category $\msh_\Lambda^c(\Lambda)$. By Proposition \ref{prop:natural-trans}, it suffices to show that for any $F \in \Sh_\Lambda^b(M)$, 
    $$\Hom(T_\epsilon(F), w_\Lambda(\mu_i)) \simeq \Hom(T_\epsilon(F), \wrap_\Lambda^+ \circ w_\Lambda(\mu_i)).$$
    Note that $m_\Lambda^l(\mu_i) \in \Sh^b_\Lambda(M)$. By Sabloff-Serr{e} duality Proposition \ref{prop:sab-serre}, we know that the right hand side is
    \begin{align*}
        \Hom(T_\epsilon(F), m_\Lambda^l(\mu_i)[1])^\vee &= \Hom(F, T_{-\epsilon} \circ m_\Lambda^l(\mu_i)[1])^\vee = \Hom(m_\Lambda^l(\mu_i)[1], F \otimes \omega_M).
    \end{align*}
    On the other hand, since $F$ is cohomologically constructible, by Theorem \ref{thm:doubling_ad} and Remark \ref{rem:DFotimesG}, the left hand side is
    \begin{align*}
        \Hom(T_\epsilon(F), w_\Lambda(\mu_i))^\vee &= p_*(\ND{M}F \otimes w_\Lambda(\mu_i))^\vee = \Hom(\ND{M}F \otimes w_\Lambda(\mu_i), p^!1_\cV) \\
        &= \Hom( w_\Lambda(\mu_i), \VD{M} \circ \ND{M} F) = \Hom(w_\Lambda(\mu_i), F \otimes \omega_M).
    \end{align*}
    Moreover, the continuation map is exactly induced by the continuation map $w_\Lambda(\mu_i) \rightarrow \wrap_\Lambda^+ \circ w_\Lambda(\mu_i)$. Therefore, the result immediately follows from Proposition \ref{w=ad}.
\end{proof}

By Proposition \ref{prop:full-condition2} and \ref{prop:full-condition4}, we can immediately finish the proof of Theorem \ref{thm:main-fun} and hence the full stop part in Theorem \ref{thm:main}.

\subsubsection{Spherical adjunction from swappable stops}
    Next, we define the notion of a swappable stop, which is introduced by Sylvan \cite{SylvanOrlov}, but is a priori weaker than his terminology.

\begin{definition}\label{def:swappable}
    Let $\Lambda \subset T^{*,\infty}M$ be a compact subanalytic Legendrian. Then $\Lambda$ is called a swappable Legendrian stop if there exists a positive wrapping fixing $\Lambda$ that sends $T_\epsilon(\Lambda)$ to an arbitrarily small neighbourhood of $T_{-\epsilon}(\Lambda)$.
\end{definition}
\begin{example}
    The Legendrian stops in Example \ref{ex:Lef-full} are swappable, and we conjecture that Legendrian stops in Example \ref{ex:triangle-full} are swappable as well. More generally, when $F \subset S^{*}M$ is a Weinstein page of a contact open book decomposition for $S^{*}M$ \cite{Giroux,HondaHuang}, then it is a swappable stop.

    However, swappable stops are not necessarily full stops. For instance, for the Landau-Ginzburg model $\pi: T^*M \rightarrow \mathbb{C}$ that are not Lefschetz fibrations, $F \subset S^{*}M$ will in general not be a full stop (one can consider $\pi: \mathbb{C}^n \rightarrow \mathbb{C}; \;\; \pi = z_1z_2\dots z_n$ \cite{Nadspherical,AbAurouxHMS}). Sylvan \cite{SylvanOrlov}*{Example 1.4} also explained that one can take any monodromy invariant subset for the Lagrangian skeleton of fiber of a Lefschetz fibration and get a swappable stop.
\end{example}

\begin{example}
    There is another way to get new swappable stops from old ones\footnote{The authors would like to thank Emmy Murphy who explains to us this construction.}. Consider $\Lambda \subset \partial_\infty X$ to be a swappable stop in $\partial_\infty X$ the contact boundary of some Liouville domain, and $\partial_\infty X'$ the contact boundary of some other Liouville domain. When we take the Liouville connected sum of $X$ and $X'$ along some subcritical Weinstein hypersurface $F \hookrightarrow \partial_\infty X$ and $F \hookrightarrow \partial_\infty X'$ \cite{Avdek}, as the skeleton $\mathfrak{c}_F$ is of subcritical dimension, the positive loop of $\Lambda$ will generically avoid $\mathfrak{c}_F$. Therefore, $\Lambda$ is also swappable in $\partial_\infty (X \#_F X')$.
    
    In particular, let $X = \mathbb{C}^n$ and $F = \mathbb{C}^{n-1}$, then $X \#_F X'$ is the 1-handle connected sum of $\mathbb{C}^n$ and $X'$, which is Liouville homotopy equivalent to $X'$. In particular, for any swappable stop $\Lambda \subset D^{2n-1} \subset S^{2n-1}$ (for example, the skeleton of any page of contact open book decomposition), putting it in a Darboux ball in $S^{*}M$, we will get a swappable stop in $S^{*}M$.
\end{example}

    When $\Lambda \subset S^{*}M$ is a swappable Legendrian stop, then there exists a cofinal positive (resp.~negative) wrapping that sends $T_{\epsilon}(\Lambda)$ to an arbitrary small neighbourhood of $T_{-\epsilon}(\Lambda)$ (resp.~sends $T_{-\epsilon}(\Lambda)$ to an arbitrary small neighbourhood of $T_{\epsilon}(\Lambda)$). So there exists a (cofinal sequence of) positive contact flow $\varphi_k^t$, $k \in \NN$, supported away from $\Lambda$ such that $\varphi_k^1(T_\epsilon(\Lambda))$ is contained in a small neighbourhood of $ T_{-\epsilon}(\Lambda)$ for $k \gg 0$. We will fix the cofinal sequence of positive flow and check condition (2) and (4) of Definition \ref{def:spherical} by considering this particular wrapping.
    
\begin{proposition}\label{prop:swap-condition2}
    Assume $\Lambda$ is swappable. The functors $S_\Lambda^+$ and $S_\Lambda^-$ are equivalences.
\end{proposition}
    
\begin{proof}
    It's sufficient to check that they are fully-faithful since $S_\Lambda^+ \vdash S_\Lambda^-$ is an adjunction pair. The computation is symmetric so we check that $\Hom(S_\Lambda^+ F, S_\Lambda^+ G) = \Hom(F,G)$, or equivalently, the canonical map 
    $$ \Hom(F,G) = \Hom(T_\epsilon (F), T_\epsilon (G)) \rightarrow \Hom(T_\epsilon (F), \wrap_\Lambda^+ \circ T_\epsilon(G)) $$
    is an isomorphism. First apply Proposition \ref{prop:hom_w_pm}, so that the map factorizes as
    $$\Hom(T_\epsilon (F), T_\epsilon (G)) \xrightarrow{\sim} \Hom(T_\epsilon (F), T_{\epsilon + \delta} (G)) \rightarrow \Hom(T_\epsilon (F), \wrap_\Lambda^+ \circ T_\epsilon(G)) $$
    for some $0 < \delta \ll \epsilon$. Then by the swappable assumption, for a sequence of descending open neighbourhoods $\{\Omega_k\}_{k \in \NN}$ of $\Lambda \subseteq S^{*}M$ such that $\Omega_{k+1} \subseteq \overline{\Omega_k}$ and
    $\bigcap_{k \in \NN} \Omega_k = \Lambda,$
    there exist (an increasing sequence of) positive Hamiltonian flows $\varphi_k^t$, $k \in \NN$, supported away from $\Omega_k$ such that 
    $$\varphi_k^1(T_\epsilon(\Lambda)) \subseteq T_{-1/k}(\Omega_k)$$
    for $k \gg 0$. Thus we are in the situation of Lemma \ref{lem:nearby_cycle}.
\end{proof}

\begin{proposition}\label{prop:swap-condition4}
When $\Lambda$ is swappable. The canonical map  $m_\Lambda^r \rightarrow m_\Lambda^r \circ m_\Lambda \circ m_\Lambda^l \rightarrow S_\Lambda^- \circ m_\Lambda^l [1]$ is an isomorphism. 
\end{proposition}

\begin{proof}
By Proposition \ref{prop:natural-trans}, it's sufficient to show that the map
$$\Hom(T_\epsilon({F}), w_\Lambda({G})) \rightarrow \Hom(T_{\epsilon}({F}), \wrap_\Lambda^+ \circ w_\Lambda({G}))$$ is an isomorphism.
Since $\msif (w_\Lambda(G)) \subseteq T_{-\epsilon} (\Lambda) \cup T_\epsilon (\Lambda)$, we can again flow it forward by the Reeb flow $T_\delta$ for some $0 < \delta \ll \epsilon$ by Proposition \ref{prop:hom_w_pm} such that the above map factorizes as
$$\Hom(T_\epsilon({F}), w_\Lambda({G})) = \Hom(T_\epsilon({F}), T_\delta \circ w_\Lambda({G})) \rightarrow \Hom(T_{\epsilon}({F}), \wrap_\Lambda^+ \circ w_\Lambda({G})).$$
But then the same proof from the last proposition applies. By the swappable assumption, for a sequence of descending open neighbourhoods $\{\Omega_k\}_{k \in \NN}$ of $\Lambda \subset S^{*}M$ such that $\Omega_{k+1} \subseteq \overline{\Omega_k}$ and
$\bigcap_{k \in \NN} \Omega_k = \Lambda,$
there exist (an increasing sequence of) positive Hamiltonian flows $\varphi_k^t$, $k \in \NN$, supported away from $\Omega_k$ such that 
$$\varphi_k^t(T_\epsilon(\Lambda) \cup T_{-\epsilon}(\Lambda)) \subseteq T_{-1/k}(\Omega_k) \cup T_{-1/2k}(\Lambda)$$
for $k \gg 0$. We can again apply Lemma \ref{lem:nearby_cycle} to conclude that the map is an isomorphism.
\end{proof}

    By Proposition \ref{prop:swap-condition2} and \ref{prop:swap-condition4}, we can immediately finish the proof of Theorem \ref{thm:main-fun} and hence the swappable stop part of Theorem \ref{thm:main}.

\begin{example}
Let $M = \RR^1$ and $\Lambda = 0_{\RR^1} \cup (\cup_{n \in \ZZ} T^*_{n,\leq} \RR^1)$. One can see $\Lambda$ is swappable in this case.
The constant sheaves $\{ 1_{(n,\infty)} \}_{n \in \ZZ}$ form a set of generators in $\Sh_\Lambda(\RR^1)$ and $S_\Lambda^+$ sends $1_{(n,\infty)}$ to $1_{(n+1,\infty)}$ for $n \in \ZZ$.
From the point of view of \cite[Theorem 1.1]{Kuo-Li-duality}, the functor $S_\Lambda^+$ and its inverse are given by sheaf kernels and they are, up to a shift by $[1]$, the constant sheaves supported in the shaded areas as in Figure \ref{fig:wraponce-kernel}. We mention that the case $M = S^1$ and $\Lambda = 0_{S^1} \cup T^*_{0,\leq} S^1$ can be described similarly by lifting to the universal cover.
\begin{figure}[h!]
\includegraphics[width=0.9\textwidth]{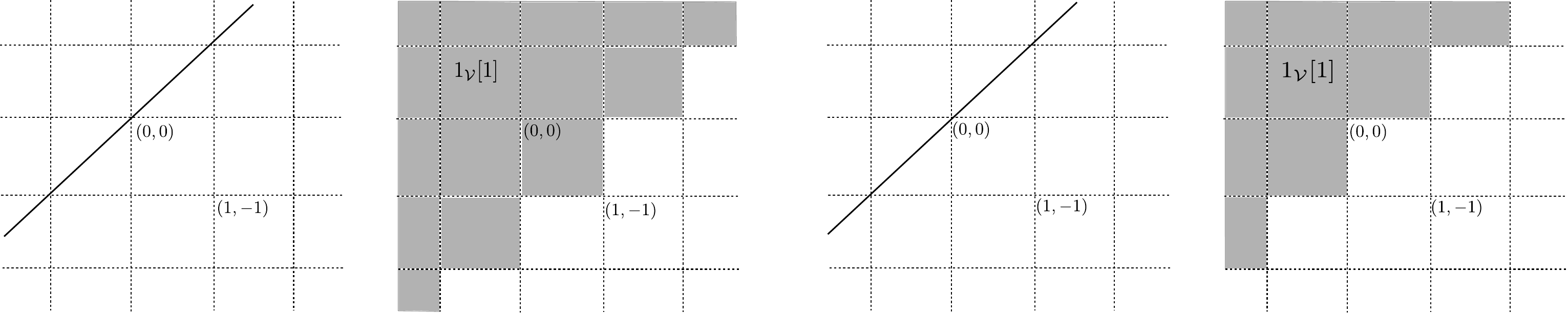}
\caption{The sheaf kernel of the positive wrap-once functor $S_\Lambda^+$ for $\Lambda = 0_{\RR^1} \cup (\cup_{n \in \ZZ} T^*_{n,\leq} \RR^1)$ (left) the kernel of the negative wrap-once functor $S_\Lambda^-$ (right). The stops are the dashed lines where conormal directions are pointing down and right.}\label{fig:wraponce-kernel}
\end{figure}
\end{example}

\subsection{Spherical adjunction on subcategories}\label{sec:prop-cpt}
    In this section, we will restrict to the subcategories of proper objects and compact objects of sheaves. Over the category of proper objects of sheaves (equivalently, sheaves with perfect stalks), we will show that the microlocalization
    $$m_\Lambda: \Sh^b_\Lambda(M) \rightarrow \msh_\Lambda^b(\Lambda)$$
    is a spherical functor, and over the category of compact objects of sheaves, we will show that the left adjoint of the microlocalization
    $$m_\Lambda^l: \msh^c_\Lambda(M) \rightarrow \Sh_\Lambda^c(M)$$
    is a spherical functor. We know that autoequivalences coming from twists and cotwists immediately restrict to these corresponding subcategories. As a result, once we know the corresponding functors restrict, they will be spherical. 

    First, we consider the subcategories of compact objects. For the spherical adjunction $m_\Lambda^l \dashv m_\Lambda$. We know that the left adjoint $m_\Lambda^l$ preserves compact objects, i.e.~we have
    $$m_\Lambda^l: \msh_\Lambda^c(\Lambda) \rightarrow \Sh_\Lambda^c(M).$$
    However, it is not clear whether microlocalization $m_\Lambda$ also preserves these objects.

\begin{lemma}
    Let $\Lambda \subset S^{*}M$ be a subanalytic Legendrian stop. When $m_\Lambda^r$ admits a right adjoint, the essential image of the microlocalization functor
    $$m_\Lambda:\, \Sh^c_\Lambda(M) \rightarrow \msh_\Lambda(\Lambda)$$
    is contained in $\msh^c_\Lambda(\Lambda)$
\end{lemma}
\begin{proof}
    We know that $m_\Lambda^r$ preserves colimits as it admits a right adjoint. Now since the right adjoint of $m_\Lambda$ preserves colimits, we can conclude that $m_\Lambda$ preserves compact objects.
\end{proof}

    Whenever $m_\Lambda \vdash m_\Lambda^l$ is a spherical adjunction, we know by Remark \ref{rem:sphere-ad} that $m_\Lambda^r$ admits a right adjoint. Therefore spherical adjunction can always be restricted to the subcategories of compact objects, as we have claimed in Corollary \ref{cor:main}.
    
    Then we consider the subcategories of proper objects. We know that the microlocalization functor preserve proper objects (or equivalently objects with perfect stalks), i.e.~we have
    $$m_\Lambda: \Sh_\Lambda^b(M) \rightarrow \msh_\Lambda^b(\Lambda).$$
    However, it is not clear whether the left adjoint $m_\Lambda^l$ and right adjoint $m_\Lambda^r$ preserves these objects.
    
\begin{lemma}\label{lem:sphere-proper}
    Let $\Lambda \subset S^{*}M$ be a subanalytic Legendrian stop. When $m_\Lambda^r$ admits a right adjoint, the essential image of the left adjoint of microlocalization functor
    $$m_\Lambda^r: \msh^b_\Lambda(\Lambda) \rightarrow  \Sh_\Lambda(M)$$
    is also contained in $\Sh^b_\Lambda(M)$.
\end{lemma}
\begin{proof}
    We recall Theorem \ref{thm:perfcompact} that 
    $$\Sh^b_\Lambda(M) = \Fun^{ex}(\Sh^c_\Lambda(M)^{op}, \cV_0), \;\; \msh^b_\Lambda(\Lambda) = \Fun^{ex}(\msh^c_\Lambda(\Lambda)^{op}, \cV_0),$$
    where the isomorphism is given by the $\Hom(-, -)$ pairing on $\Sh^c_\Lambda(M)^{op} \times \Sh^b_\Lambda(M)$ and respectively on $\msh^c_\Lambda(\Lambda)^{op} \times \msh^b_\Lambda(\Lambda)$. Then since we know that microlocalization preserves compact objects
    $$m_\Lambda: \Sh^c_\Lambda(M) \rightarrow \msh^c_\Lambda(\Lambda),$$
    the right adjoint $m_\Lambda^r$ clearly preserves proper objects.
\end{proof}

    Whenever $m_\Lambda \vdash m_\Lambda^l$ is a spherical adjunction, we know by Remark \ref{rem:sphere-ad} that $m_\Lambda^r$ admits a right adjoint, so the adjunction $m_\Lambda^r \vdash m_\Lambda$ can be restricted to the subcategories of proper objects.
    
    Therefore by Remark \ref{rem:sphere-ad} the spherical adjunction $m_\Lambda \vdash m_\Lambda^l$ can always be restricted to the subcategories of proper objects, as we have claimed in Corollary \ref{cor:main}.

    As we have seen, the candidate right adjoint of $m_\Lambda^l$ will simply be the microlocalization functor
    $$m_\Lambda:\, \Sh_\Lambda(M) \rightarrow \msh_\Lambda(\Lambda).$$
    The candidate left adjoint of $m_\Lambda^l$, by Remark \ref{rem:sphere-ad}, is the functor
    $$m_\Lambda \circ {S}_\Lambda^-[1]: \, \Sh_\Lambda(M) \rightarrow \msh_\Lambda(\Lambda).$$
    
    The readers may be confused about the non-symmetry as $m_\Lambda^l$ is supposed to be the cup functor on wrapped Fukaya categories and there is no non-symmetry there. We can provide an explanation as follows. Consider the category of wrapped sheaves in \cite{Kuo-wrapped-sheaves}, we have a preferred equivalence $\mathfrak{W}_\Lambda^+: \wsh_\Lambda(M) \rightarrow \Sh^c_\Lambda(M)$. If we try to instead replace the domain $\Sh^c_\Lambda(M)$ by $\wsh_\Lambda(M)$, then $m_\Lambda^l$ can be replaced by the doubling functor $w_\Lambda[-1]$, and one can easily see that
    $$m_\Lambda^r = m_\Lambda \circ \mathfrak{W}_\Lambda^+, \,\,\, m_\Lambda^l = m_\Lambda \circ \mathfrak{W}_\Lambda^-[1].$$
    Then $m_\Lambda^r$ (resp.~$m_\Lambda^l$) is indeed the cap functor by wrapping positively (resp.~negatively) into the Legendrian $\Lambda$ and then take microlocalization, i.e.~the sheaf theoretic restriction.

\subsection{Serre functor on proper subcategory}\label{sec:serre-proper}
    In this section, we finally prove that $S^-_\Lambda$ is in fact the Serr{e} functor on $\Sh^b_\Lambda(M)$, when $\Lambda$ is a full stop or swappable stop. Thus we will prove the folklore conjecture on partially wrapped Fukaya categories associated to Lefschetz fibrations formulated by Seidel \cite{SeidelSH=HH} (who attributes the conjecture to Kontsevich) with partial results in \cite{SeidelFukI,SeidelFukII,SeidelFukIV1/2}, in the case of partially wrapped Fukaya categories on cotangent bundles.

    Let $\sA$ be compactly generated by a small idempotent complete stable category $A$. Recall that $\sA$ is proper category if for any $X, Y \in A$,
    $$\Hom_\sA(X, Y) \in \cV_0.$$
    When $\sA$ is a proper category, by the above lemma, we are always able to define the (right) dualizing bi-module $\sA^\vee$.

\begin{definition}
    For a proper stable category $\sA$, the (right) dualizing bi-module $\sA^\vee$ is defined by
    $$\sA^\vee(X, Y) = \Hom_\sA(X, Y)^\vee = \Hom_\cV(\Hom_\sA(X, Y), 1_\cV).$$
\end{definition}
\begin{definition}
    For a proper stable category $\sA$, a Serr{e} functor $S_\sA$ is the functor that represents the right dualizing bimodule $\sA^\vee$, i.e.
    $$\Hom_\sA(-, -)^\vee \simeq \Hom_\sA(-, S_\sA(-)).$$
\end{definition}

\begin{proposition}\label{prop:serre}
    Let $\Lambda \subset S^{*}M$ be a full or swappable compact subanalytic Legendrian stop. Then $S_\Lambda^- \otimes \omega_M$ is the Serr{e} functor on $\Sh^b_\Lambda(M)_0$ of sheaves with perfect stalks and compact supports. In particular, when $M$ is orientable, $S_\Lambda^-[-n]$ is the Serr{e} functor on $\Sh^b_\Lambda(M)_0$.    
\end{proposition}
\begin{proof}
    First, by Lemma \ref{lem:sphere-proper}, we know that $S_\Lambda^-: \Sh^b_\Lambda(M) \rightarrow \Sh^b_\Lambda(M)$ preserves perfect stalks. Moreover, when $M$ is noncompact and $\Lambda$ is a swappable stop, we argue that $S_\Lambda^-$ also preserves compact supports. By the swappable assumption, for a sequence of descending open neighbourhoods $\{\Omega_k\}_{k \in \NN}$ of $\Lambda \subset S^{*}M$ such that $\Omega_{k+1} \subseteq \overline{\Omega_k}$ and
    $\bigcap_{k \in \NN} \Omega_k = \Lambda,$
    there exist (an increasing sequence of) positive Hamiltonian flows $\varphi_k^t$, $k \in \NN$, supported away from $\Omega_k$ such that 
    $$\varphi_k^{-1}(T_{-\epsilon}(\Lambda)) \subset T_{1/k}(\Omega_k)$$
    for $k \gg 0$. Without loss of generality, we can even assume that $\varphi_k^t$ are all supported in some common compact subset. Since $M$ is noncompact, consider any unbounded region $U$ such that $\Lambda \cap S^*U = \varnothing$. Then there exists an open subset $U' \subseteq U \subseteq M$, $\varphi_k^t(T_\epsilon(\Lambda)) \cap S^*U' = \varnothing$ for $k \gg 0$. Then 
    $$\Gamma(U', S_\Lambda^-(F)) = \Gamma\Big(U', \lmi{k \rightarrow \infty}\,K(\varphi_k^{-1}) \circ T_{-\epsilon}(F)\Big) = 0.$$
    Since $\ms^\infty(S_\Lambda^-(F)) \subseteq \Lambda$, we get $\Gamma(U, S_\Lambda^-(F)) = 0$, which implies that $S^-_\Lambda(F)$ has compact support.
    
    Since we have concluded that $S_\Lambda^-: \Sh^b_\Lambda(M)_0 \rightarrow \Sh^b_\Lambda(M)_0$ preserves perfect stalks and compact supports, the proposition then immediately follows from Theorem \ref{w=ad} and Proposition \ref{prop:sab-serre}
    and the fact that $\wrap_\Lambda^-(- \otimes \omega_M) = (\wrap_\Lambda^-(-))\otimes \omega_M$ since $\omega_M$ is a local system.
\end{proof}

    We should remark that, even though Proposition \ref{prop:sab-serre} is true in general, the above statement is not without the assumption on $\Lambda \subseteq S^*M$. For example, in Section \ref{sec:example} we will see an example where $S_\Lambda^-$ fails to be an equivalence on $\Sh_\Lambda^b(M)$.

    Finally, we explain the implication of the above result in partially wrapped Fukaya categories. Ganatra-Pardon-Shende \cite[Proposition 7.24]{Ganatra-Pardon-Shende3} have proved that there is a commuative diagram intertwining the cup functor and the left adjoint of microlocalization functor
    \[\xymatrix{
    \mathcal{W}(F) \ar[r]^{\sim\hspace{10pt}} \ar[d]_{\cup_F} & \mu Sh^c_{\mathfrak{c}_F}(\mathfrak{c}_F) \ar[d]^{m_{\mathfrak{c}_F}^*}\\
    \mathcal{W}(T^*M, F) \ar[r]^{\sim} & Sh^c_{\mathfrak{c}_F}(M).
    }\]
    Sylvan has shown that the spherical twist associated to the cup functor is the wrap-once functor \cite{SylvanOrlov}, and hence it intertwines with the wrap-once functor in sheaf categories. Consequently, we have proven that the negative wrap-once functor
    $$\cS_\Lambda^-: \mathrm{Prop}\,\mathcal{W}(T^*M, F) \rightarrow \mathrm{Prop}\,\mathcal{W}(T^*M, F)$$
    is indeed the Serr{e} functor on $\mathrm{Prop}\,\mathcal{W}(T^*M, F)$. 

    In particular, let $\pi: T^*M \rightarrow \mathbb{C}$ be a symplectic Lefschetz fibration and $F = \pi^{-1}(\infty)$ be the Weinstein fiber. Let $\mathfrak{c}_F$ be the Lagrangian skeleton of $F$. Then by Ganatra-Pardon-Shende \cite{Ganatra-Pardon-Shende2,Ganatra-Pardon-Shende3} we know that 
    $$\mathrm{Perf}\,\mathcal{W}(T^*M, F) = \mathrm{Prop}\,\mathcal{W}(T^*M, F)$$
    is a proper subcategory. Therefore, $S_\Lambda^-$ is the Serr{e} functor on partially wrapped Fukaya category associated to Lefschetz fibrations. 
    
\begin{remark}\label{rem:relativeCY}
    Finally, we remark that according to the result of Katzarkov-Pandit-Spaide \cite{KPSsphericalCY}, existence of spherical adjunction in the next section together with a compatible Serr{e} functor will imply existence of the weak relative proper Calabi-Yau structure introduced in \cite{relativeCY} of the pair
    $$m_\Lambda: \Sh^b_\Lambda(M) \rightarrow \msh_\Lambda^b(\Lambda)$$
    (even though we do not explicitly show compatibility of the Serr{e} functor, we believe that it is basically clear from the definition in \cite{KPSsphericalCY}).
    
    However, we will see in the last section that spherical adjunction does not hold in all these pairs that are expected to be relative Calabi-Yau (on the contrary, as explained in Sylvan \cite{SylvanOrlov} or Remark \ref{rem:multi-component-ad} and \ref{rem:multi-component-double}, when we consider microlocalization along a single component of a Legendrian stop with multiple components
    $$m_{\Lambda_i}: \Sh_\Lambda(M) \rightarrow \msh_{\Lambda}(\Lambda) \rightarrow \msh_{\Lambda_i}(\Lambda_i)$$
    we will still get spherical adjunctions, but the pair is unlikely to be relative Calabi-Yau). We will investigate relative Calabi-Yau structures separately (and hopefully, in full generality) in future works.
\end{remark}

\section{Spherical pairs and perverse sch\"{o}bers}

    In this section, we provide immediate corollaries of the main theorem and discuss how they give rise to spherical pairs and semi-orthogonal decompositions, and prove Proposition \ref{prop:sphere-pair-tri} and Theorem \ref{thm:sphere-pair-var}. Using Proposition \ref{prop:sphere-pair-tri}, we will also give an explicit characterization of the spherical twists and dual twists.

    The description of spherical pairs comes from the relation between spherical functors and perverse sheaves of categories (called perverse sch\"{o}bers) on a disk with one singularity \cite{KapraSchober}. For a perverse sch\"{o}ber on $\mathbb{D}^2$ with singularity at $0$ associated to the spherical functor
    $$F: {\sA} \rightarrow {\sB}$$
    consider a single cut $[0, 1] \subset \mathbb{D}^2$. Then the nearby category at $0$ is ${\sA}$ while the vanishing category on $(0, 1]$ is ${\sB}$, and the spherical twist is determined by monodromy around $\mathbb{D}^2 \backslash \{0\}$. Kapranov-Schechtman realized a symmetric description of the perverse sch\"{o}ber determined by the diagram
    $${\sB}_- \xleftarrow{F_-} {\sC} \xrightarrow{F_+} {\sB}_+,$$
    by considering a double cut on the disk $[-1, 1] \subset \mathbb{D}^2$. The nearby category at $0$ is ${\sC}$ while the vanishing category on $[-1, 0)$ (resp.~$(0, 1]$) is ${\sB}_-$ (resp.~${\sB}_+$). The nearby category ${\sC}$ will carry a 4-periodic semi-orthogonal decomposition. Such a viewpoint will provide new information of the microlocal sheaf categories.

    We will provide precise definitions of the terminologies in this section and then show the results in the introduction. Note that none of the arguments in this section essentially relies on microlocal sheaf theory, and therefore they can all be rewritten using Lagrangian Floer theory.

\subsection{Semi-orthogonal decomposition}\label{sec:semi-ortho}
    Firstly, we explain how spherical adjunctions give rise to 4-periodic semi-orthogonal decompositions and spherical twists are given by iterated mutations Halpern--Laistner-Shipman \cite{SphericalGIT} in the case of dg categories and Dykerhoff-Kapranov-Schechtman-Soibelman \cite{SphericalInfty} in general (this is how Sylvan proved that the Orlov cup functor is spherical \cite{SylvanOrlov}).

\begin{theorem}[Halpern--Laistner-Shipman \cite{SphericalGIT}, \cite{SphericalInfty}]\label{thm:4periodic}
    Let $F: {\sA} \rightarrow {\sB}$ be an $\infty$-functor, and ${\sC}$ be the semi-orthogonal gluing of ${\sA}$ and ${\sB}$ along the graphical bi-module $\Gamma(F)$. Then $F$ is spherical if and only if ${\sC}$ fits into a 4-periodic semi-orthogonal decomposition such that ${\sA}^{\perp\perp\perp\perp} = {\sA}$. The dual twist is the iterated mutation $T_{\sA} = R_{\sA} \circ R_{{\sA}^{\perp\perp}}$, and the dual cotwist is $S_{\sB} = L_{\sA} \circ L_{{\sA}^{\perp\perp}}$.
\end{theorem}
\begin{remark}
    Given a pair of semi-orthogonal decompositions ${\sC} = \left<{\sA}_+, {\sB} \right> \simeq \left<{\sB}, {\sA}_-\right>$, the right mutation functor is the equivalence $R_{{\sA}}: {\sA}_+ \rightarrow {\sA}_-$ defined by the composition of embedding and projection.
\end{remark}

    Therefore, restricting to our setting of sheaf categories, our main theorem is equivalent to the following statement.

\begin{proposition}\label{prop:semi-ortho}
    Let $\Lambda \subseteq S^{*}M$ be a compact subanalytic Legendrian stop. Then under the fully faithful embedding $w_\Lambda: \msh_\Lambda(\Lambda) \rightarrow \Sh_{T_{-\epsilon}(\Lambda) \cup T_\epsilon(\Lambda)}(M)$, there is a semi-orthogonal decomposition
    $$\Sh_{T_{-\epsilon}(\Lambda) \cup T_\epsilon(\Lambda)}(M) \simeq \langle  \msh_\Lambda(\Lambda), \Sh_{T_{\epsilon}(\Lambda)}(M) \rangle.$$
\end{proposition}
\begin{proof}
    We can check that $\Sh_{T_{-\epsilon}(\Lambda) \cup T_\epsilon(\Lambda)}(M)$ is the semi-orthogonal gluing (i.e.~Grothendieck construction) of $\msh_\Lambda(\Lambda)$ and $\Sh_{\Lambda}(M)$ along the graphical bi-module $\Gamma(m_\Lambda)$. First, we show that
    $$\langle  \msh_\Lambda(\Lambda), \Sh_{T_{\epsilon}(\Lambda)}(M) \rangle \hookrightarrow \Sh_{T_{-\epsilon}(\Lambda) \cup T_\epsilon(\Lambda)}(M), $$
    By Theorem \ref{thm:doubling_ad} and Remark \ref{rem:doubling_ad}, we know that
    $$\Hom(w_\Lambda({F}), T_{\epsilon}({G})) \simeq 0, \;\; \Hom(T_{\epsilon}({G}),  w_\Lambda({F})) \simeq \Hom(m_\Lambda(G), {F} ).$$
    This proves full faithfulness. 
    
    For essential surjectivity, consider ${F} \in \Sh_{T_{-\epsilon}(\Lambda) \cup T_\epsilon(\Lambda)}(M)$. Then Theorem \ref{thm:doubling_ad} and Remark \ref{rem:multi-component-double} implies the following fiber sequence
    $$\wrap_{T_{-\epsilon}(\Lambda) \cup T_\epsilon(\Lambda)}^+ w_{T_{-\epsilon}(\Lambda)}  m_{T_{\epsilon}(\Lambda)}({F})  \rightarrow {F} \rightarrow \wrap_{T_{\epsilon}(\Lambda)}^+(F) .$$
    Since $\wrap_{T_{-\epsilon}(\Lambda) \cup T_\epsilon(\Lambda)}^+ w_{T_{-\epsilon}(\Lambda)}  m_{T_{-\epsilon}(\Lambda)}({F}) = w_{\Lambda}  m_{T_{-\epsilon}(\Lambda)}({F})$, we can conclude that
    $$w_{\Lambda}  m_{T_{-\epsilon}(\Lambda)}({F})  \rightarrow {F} \rightarrow \wrap_{T_{\epsilon}(\Lambda)}^+(F) ,$$
    where $\wrap_{T_{-\epsilon}(\Lambda)}^+(F) \in \Sh_{T_{\epsilon}(\Lambda)}(M)$ and $m_{T_{-\epsilon}(\Lambda)}({F}) \in \msh_\Lambda(\Lambda)$. This shows the essential surjectivity and thus shows the semi-orthogonal decomposition.
\end{proof}

To see that the semi-orthogonal decomposition restricts to compact objects, we need the following lemma.

\begin{lemma}[{\cite[Corollary 4.23]{Ganatra-Pardon-Shende3}}]\label{prop:stopremoval}
Let $\Lambda$ be a subanalytic (singular) isotropic in $S^* M$.
The category $\Sh_\Lambda(M)$ is compactly generated.
If $\Lambda \subseteq \Lambda^\prime$ is an inclusion of subanalytic isotropics,
then the left adjoint of $\Sh_\Lambda(M) \hookrightarrow \Sh_{\Lambda^\prime}(M)$ sends compact objects to compact objects, i.e.,
$\Sh_{\Lambda^\prime}^c(M) \twoheadrightarrow \Sh_\Lambda^c(M)$.
\end{lemma}

\begin{corollary}\label{cor:semi-ortho}
    The semi-orthogonal decomposition can be restricted to compact objects, namely there is a semi-orthogonal decomposition
    $$\Sh_{T_{-\epsilon}(\Lambda) \cup T_\epsilon(\Lambda)}^c(M) \simeq \langle \Sh_{T_{\epsilon}(\Lambda)}^c(M), \msh_\Lambda^c(\Lambda) \rangle.$$
\end{corollary}
\begin{proof}
    Consider the fiber sequence of categories
    $$\msh_\Lambda(\Lambda) \hookrightarrow \Sh_{T_{-\epsilon}(\Lambda) \cup T_\epsilon(\Lambda)}(M) \twoheadrightarrow \Sh_{T_\epsilon(\Lambda)}(M).$$
    For $F \in \Sh_{T_{-\epsilon}(\Lambda) \cup T_\epsilon(\Lambda)}^c(M)$, we know that $\wrap_{T_\epsilon(\Lambda)}^+(F) \in \Sh_{T_\epsilon(\Lambda)}^c(M)$ since by $\wrap_{T_\epsilon(\Lambda)}^+$ is the stop removal functor as explained in Proposition \ref{prop:stopremoval} and Remark \ref{rem:stoprem-cpt}. 
    
    Moreover, by Proposition \ref{prop:stopremoval} we know that the fiber of the stop removal functor is compactly generated by the corepresentatives of microstalks at $T_{-\epsilon}(\Lambda)$ in the category $\Sh_{T_{-\epsilon}(\Lambda) \cup T_\epsilon(\Lambda)}(M)$, which by definition are
    $$m_{T_{-\epsilon}(\Lambda)}^l(\mu_i) = \wrap_{T_{-\epsilon}(\Lambda) \cup T_\epsilon(\Lambda)}^+ w_{T_{-\epsilon}(\Lambda)}(\mu_i) = w_\Lambda(\mu_i),$$ 
    where $\mu_i \in \msh_\Lambda(\Lambda)$ are corepresentatives of the microstalks in the category $\msh_\Lambda(\Lambda)$. Then when restricting to compact objects, the fiber is $\msh_\Lambda^c(\Lambda)$ (split) generated by $\mu_i \in \msh_\Lambda^c(\Lambda)$.
\end{proof}

\begin{remark}
    We explain how this is related to Sylvan's proof of the spherical adjunction \cite{SylvanOrlov}*{Section 4}. Sylvan considered a sectorial gluing of the original Weinstein sector $(X, F)$ and the $A_2$-sector $F\langle 2\rangle = (\mathbb{C}, \{e^{2\pi \sqrt{-1} j/3}\infty\}_{0\leq j\leq 2}) \times F$, and showed semi-orthogonal decompositions for the ambient sector $(X, F) \cup_F F\langle 2 \rangle$. However, the sector $(X, F) \cup_F F\langle 2 \rangle$ is exactly $(X, T_{-\epsilon}(F) \cup T_\epsilon(F))$.
\end{remark}

    Going back to semi-orthogonal decompositions and spherical adjunctions, combining Theorem \ref{thm:4periodic} and Proposition \ref{prop:semi-ortho}, we immediately get the following corollary from the spherical adjunction we have proved.

\begin{corollary}\label{cor:4-periodic-shv}
    Let $\Lambda \subseteq S^{*}M$ be a swappable Legendrian stop or full Legendrian stop. Then under the fully faithful embedding $w_\Lambda: \msh_\Lambda(\Lambda) \rightarrow \Sh_{T_{-\epsilon}(\Lambda) \cup T_\epsilon(\Lambda)}(M)$, there are semi-orthogonal decompositions
    $$\Sh_{T_{-\epsilon}(\Lambda) \cup T_\epsilon(\Lambda)}(M) \simeq \langle \msh_\Lambda(\Lambda), \Sh_{T_{\epsilon}(\Lambda)}(M)\rangle \simeq \langle \Sh_{T_{-\epsilon}(\Lambda)}(M), \msh_\Lambda(\Lambda) \rangle.$$
\end{corollary}

    Consider $\msh_{\Lambda}(\Lambda)$ and $\Sh_{\widehat{\Lambda}_{\cup, \epsilon}}(M)$ where 
    $$\widehat{\Lambda}_{\cup, \epsilon} = (\Lambda \times \bR_+) \cup (\Lambda_\epsilon \times \bR_+) \cup \bigcup\nolimits_{0 \leq s \leq \epsilon} \pi(\Lambda_s).$$
    Consider the fully faithful embedding given by the doubling functor. We can give an explicit characterization of its essential image.

\begin{proposition}\label{prop:microsheaf-sheaf}
    Let $\Lambda \subseteq S^*M$ be a compact subanalytic Legendrian. Then for $\epsilon > 0$ sufficiently small, the doubling functor gives an equivalence 
    $$w_\Lambda: \msh_{\Lambda}(\Lambda) \xrightarrow{\sim} \Sh_{\widehat\Lambda_{\cup,\epsilon}}(M).$$
\end{proposition}
\begin{proof}
    Consider the complement of the front projection $M \setminus \pi(\Lambda)$. We fix a collection of points $\{x_j \mid j \in \pi_0(M \setminus \pi(\Lambda))\}$ for all the connected component. Then for $\epsilon > 0$ sufficiently small, 
    $$\{x_j \mid j \in \pi_0(M \setminus \pi(\Lambda))\} \cap \bigcup\nolimits_{0 \leq s\leq \epsilon}\pi(\Lambda_s) = \varnothing.$$
    Note that there exists a deformation retract from $\bigcup_{0 \leq s\leq \epsilon}\pi(\Lambda_s)$ to $\pi(\Lambda)$. Thus, for $F \in \Sh_{\Lambda \cup \Lambda_\epsilon}(M)$, when $F_{x_j} = 0$, we can conclude that $F \in \Sh_{\widehat\Lambda_{\cup,\epsilon}}(M)$.
    
    Consider the functor which induces the semi-orthogonal decomposition $\wrap_\Lambda^-: \Sh_{\Lambda \cup \Lambda_\epsilon}(M) \rightarrow \Sh_\Lambda(M)$. The functor can be realized by a cofinal sequence of negative wrappings. Consider a family of cut-off functions $H_k$ that are equal to $1$ on a neighbourhood $\bigcup_{\epsilon/k \leq s \leq \epsilon}\pi(\Lambda_s)$ and $0$ outside a neighbourhood of $\bigcup_{\epsilon/2k \leq s \leq \epsilon}\pi(\Lambda_s)$. Let $\varphi_{k}: S^*M \rightarrow S^*M$ be the decreasing sequence of negative Hamiltonian flow such that
    $$\varphi_{k}(\Lambda) = \Lambda, \; \varphi_{k}(\Lambda_\epsilon) = \Lambda_{\epsilon/k}, \; \varphi_{k}(S^*_{x_j} M) = S^*_{x_j}M.$$
    Then $\wrap_\Lambda^-(F)$ can be realized by the nearby cycle of $F$ with respect to the decreasing family of Hamiltonian flows. By the locality property of the Hamiltonian action, we know that 
    $$\wrap_\Lambda^-(F)_{x_j} = \lim_{k \to \infty} \varphi_k(F)_{x_j} = F_{x_j} = F_{x_j}.$$
    When $\wrap_\Lambda^-(F)_{x_j} = 0$, we know $F_{x_j} = 0$ as well, which implies that $F \in \Sh_{\widehat\Lambda_{\cup,\epsilon}}(M)$. Therefore, the kernel of $\wrap_\Lambda^-: \Sh_{\Lambda \cup \Lambda_\epsilon}(M) \to \Sh_\Lambda(M)$ is given by $\Sh_{\widehat\Lambda_{\cup,\epsilon}}(M)$.
\end{proof}

    In fact, the semi-orthogonal decompositions
    $$\Sh_{T_{-\epsilon}(\Lambda) \cup T_\epsilon(\Lambda)}(M) = \langle \msh_\Lambda(\Lambda), \Sh_{T_\epsilon(\Lambda)}(M) \rangle = \langle \Sh_{T_{-\epsilon}(\Lambda)}(M), \msh_\Lambda(\Lambda) \rangle$$ also provide (trivial) examples of spherical pairs, which we now introduce.

    For a diagram of $\infty$-functors over stable $\infty$-categories
    $${\sB}_- \xleftarrow{F_-} {\sC} \xrightarrow{F_+} {\sB}_+$$
    where $F_\pm$ admit fully faithful left and adjoints $F_\pm^{l,r}$, we can write
    $${\sA}_- = \ker(F_+) = ^\perp(F_+^r{\sB}_+) = (F_+^l{\sB}_+)^\perp, \;\; {\sA}_+ = \ker(F_-) = ^\perp(F_-^r{\sB}_-) = (F_-^l{\sB}_-)^\perp,$$
    and write $\iota_\pm: {\sA}_\pm \rightarrow {\sC}$ (which admits left and right adjoints $\iota_\pm^*$ and $\iota_\pm^!$). The following definition is essentially a reinterpretation of the conditions in Theorem \ref{thm:4periodic}.

\begin{definition}
    A diagram of $\infty$-functors over stable $\infty$-categories
    $${\sB}_- \xleftarrow{F_-} {\sC} \xrightarrow{F_+} {\sB}_+$$
    is called a spherical pair if $F_\pm$ admit fully faithful left and right adjoints $F_\pm^{l,r}$ such that
    \begin{enumerate}
      \item the compositions $F^l_+ \circ F_-: {\sB}_+ \rightarrow {\sB}_-, \; F^l_- \circ F_+: {\sB}_- \rightarrow {\sB}_+$ are equivalences;
      \item the compositions $\iota^!_+ \circ \iota_-: {\sA}_+ \rightarrow {\sA}_-, \; \iota^!_- \circ \iota_+: {\sA}_- \rightarrow {\sA}_+$ are equivalences.
    \end{enumerate}
\end{definition}

    Then Corollary \ref{cor:4-periodic-shv} immediately implies the following proposition.
    
\begin{proposition}\label{prop:sphere-pair-tri}
    Let $\Lambda \subset S^{*}M$ be a swappable Legendrian stop or full Legendrian stop. Then there exists spherical pairs of the form
    $$\msh_{T_{-\epsilon}(\Lambda)}(T_{-\epsilon}(\Lambda)) \leftarrow \Sh_{T_{-\epsilon}(\Lambda) \cup T_\epsilon(\Lambda)}(M) \rightarrow \msh_{T_\epsilon(\Lambda)}(T_{-\epsilon}(\Lambda)).$$
    Meanwhile, there is also a spherical pair
    $$\Sh_{T_{-\epsilon}(\Lambda)}(M) \rightarrow \Sh_{T_{-\epsilon}(\Lambda) \cup T_\epsilon(\Lambda)}(M) \leftarrow \Sh_{T_\epsilon(\Lambda)}(M).$$
\end{proposition}
\begin{remark}
    We can restrict the spherical pairs to the subcategories of compact or proper objects as explained in Section \ref{sec:prop-cpt} and Corollary \ref{cor:semi-ortho}.
\end{remark}

    Moreover, from the description in Theorem \ref{thm:4periodic}, we can show that the spherical twists (resp.~dual twists) are simply the positive (resp.~negative) monodromy functor, under the inclusion by the doubling functor.
    
\begin{corollary}\label{cor:twist}
    Let $\Lambda \subseteq S^{*}M$ be a swappable Legendrian stop or full Legendrian stop. Under the inclusion by the doubling functor $w_\Lambda: \msh_\Lambda(\Lambda) \hookrightarrow \Sh_{T_{-\epsilon}(\Lambda) \cup T_\epsilon(\Lambda)}(M)$, the spherical dual twist is computed by
    $$S_{T_{-\epsilon}(\Lambda) \cup T_\epsilon(\Lambda)}^- \circ S_{T_{-\epsilon}(\Lambda) \cup T_\epsilon(\Lambda)}^-|_{w_\Lambda(\msh_\Lambda(\Lambda))}[2].$$
    Similarly, the spherical dual cotwist is computed by $$S_{T_{-\epsilon}(\Lambda) \cup T_\epsilon(\Lambda)}^+ \circ S_{T_{-\epsilon}(\Lambda) \cup T_\epsilon(\Lambda)}^+|_{\Sh_{T_\epsilon(\Lambda)}(M)} = S_{T_\epsilon(\Lambda)}^+.$$
\end{corollary}
\begin{proof}
    By Proposition \ref{prop:sphere-pair-tri} and Theorem \ref{thm:4periodic}, it suffices to show that the right mutation functor
    $$R_{\msh_\Lambda(\Lambda)}: w_\Lambda(\msh_\Lambda(\Lambda)) \hookrightarrow \Sh_{T_{-\epsilon}(\Lambda) \cup T_\epsilon(\Lambda)}(M) \rightarrow w_\Lambda(\msh_\Lambda(\Lambda))$$
    is the functor $S_{T_{-\epsilon}(\Lambda) \cup T_\epsilon(\Lambda)}^-|_{w_\Lambda(\msh_\Lambda(\Lambda))}[1]$. Consider the pair of semi-orthogonal decompositions
    $$\Sh_{T_{-\epsilon}(\Lambda) \cup T_\epsilon(\Lambda)}(M) = \left< \msh_\Lambda(\Lambda), \Sh_{T_\epsilon(\Lambda)}(M) \right> = \left< \Sh_{T_{-\epsilon}(\Lambda)}(M), \msh_\Lambda(\Lambda) \right>.$$
    The first semi-orthogonal decomposition is realized in Proposition \ref{prop:semi-ortho} by $m_{T_{-\epsilon}(\Lambda)}^l m_{T_{-\epsilon}(\Lambda)}(F) \rightarrow F \rightarrow \mathfrak{W}_{T_\epsilon(\Lambda)}^+ F$. Following a similar argument, the second semi-orthogonal decomposition is realized by the fiber sequence
    $$\wrap_{T_{-\epsilon}(\Lambda)}^- F \rightarrow F \rightarrow m_{T_{-\epsilon}(\Lambda)}^r m_{T_{-\epsilon}(\Lambda)}(F).$$
    Therefore, one can show that the right mutation functor associated to the pair of semi-orthogonal decompositions using Theorem \ref{thm:doubling_ad} and Remark \ref{rem:multi-component-double}
    \begin{align*}
        R_{\msh_\Lambda(\Lambda)}&(w_{\Lambda} m_{T_{-\epsilon}(\Lambda)}(F)) = m_{T_{-\epsilon}(\Lambda)}^r m_{T_{-\epsilon}(\Lambda)} \circ m_{T_{-\epsilon}(\Lambda)}^l m_{T_{-\epsilon}(\Lambda)}(F)[1] \\
        &= m_{T_{-\epsilon}(\Lambda)}^r m_{T_{-\epsilon}(\Lambda)} \circ m_{T_{-\epsilon}(\Lambda)}^l m_{T_{-\epsilon}(\Lambda)}(F)[1] = m_{T_{-\epsilon}(\Lambda)}^r m_{T_{-\epsilon}(\Lambda)}(F)[1] \\
        &= \wrap_{T_{-\epsilon}(\Lambda) \cup T_\epsilon(\Lambda)}^- w_{T_{-\epsilon}(\Lambda)} m_{T_{-\epsilon}(\Lambda)}(F)[1] = S_{T_{-\epsilon}(\Lambda) \cup T_\epsilon(\Lambda)}^- w_{\Lambda} m_{T_{-\epsilon}(\Lambda)}(F)[1].
    \end{align*}
    One can also compute $R_{\msh_\Lambda(\Lambda)^{\perp\perp}}$ in the same way, which implies the result on spherical twists.
    
    For spherical cotwists, it suffices to show that the left mutation functor is $S_{T_{-\epsilon}(\Lambda) \cup T_\epsilon(\Lambda)}^+|_{\Sh_{T_\epsilon(\Lambda)}(M)} $. Then using the above semi-orthogonal decompositions, we have
    \begin{align*}
        L_{\Sh_{T_\epsilon(\Lambda)}(M)}(\wrap_{T_\epsilon(\Lambda)}^-F) = \wrap_{T_{-\epsilon}(\Lambda)}^+ \circ \wrap_{T_\epsilon(\Lambda)}^-F = S_{T_{-\epsilon}(\Lambda) \cup T_\epsilon(\Lambda)}^+ \circ \wrap_{T_\epsilon(\Lambda)}^-F.
    \end{align*}
    One also can compute $L_{\Sh_{T_{-\epsilon}(\Lambda)}(\Lambda)}$ in the same way. This implies the result on spherical dual cotwists. Finally, we note that from the computation
    $$S_{T_{-\epsilon}(\Lambda) \cup T_\epsilon(\Lambda)}^+ \circ S_{T_{-\epsilon}(\Lambda) \cup T_\epsilon(\Lambda)}^+(\wrap_{T_\epsilon(\Lambda)}^-F) = \wrap_{T_{-\epsilon}(\Lambda)}^+ \circ \wrap_{T_\epsilon(\Lambda)}^+(\wrap_{T_\epsilon(\Lambda)}^-F) = S_{T_\epsilon(\Lambda)}^+(\wrap_{T_\epsilon(\Lambda)}^-F),$$
    which confirms the last assertion.
\end{proof}
\begin{remark}
    Note that the functor $S_{T_{-\epsilon}(\Lambda) \cup T_\epsilon(\Lambda)}^+$ intertwines $T_{-\epsilon}(\Lambda)$ and $T_\epsilon(\Lambda)$ by a positive isotopy. Hence applying the functor twice has the effect of the monodromy functor.
\end{remark}
    
    These seem to be trivial examples of spherical pairs. In the following section, we will provide some examples that are less trivial, namely pairs of different subanalytic Legendrians.

\subsection{Spherical pairs from variation of skeleta}
    In general, there are subanalytic Legendrian stops that give rise to spherical pairs which are not necessarily homeomorphic. For example, Donovan-Kuwagaki \cite{DonovanKuwa} have considered two specific examples from homological mirror symmetry of toric stacks. They presented equivalences between sheaf categories
    $$\Sh^c_{\Lambda_+}(T^n) \xrightarrow{\sim} \Sh^c_{\Lambda_-}(T^n)$$
    that are mirror to certain flop-flop equivalences of the mirror toric stacks defined by GIT quotients \cite{Flopsphere} (more generally, it is discussed in \cite{ZhouVarI} how different GIT quotients are related by semi-orthogonal decompositions).

\begin{remark}
    Unlike in algebraic geometry, where the equivalence is between derived categories of varieties related by flops that are only birational equivalent, in symplectic geometry, we expect that the Weinstein sectors $(T^*T^n, \Lambda_-)$ and $(T^*T^n, \Lambda_+)$ are Weinstein homotopic (see \cite{CE,WeinRevisit} for the definition), even though the Lagrangian skeleta and associated Landau-Ginzburg potentials \cite{RSTZSkel,ZhouSkel,GammageShende} are different. This reflects the flexibility on the symplectic side.
\end{remark}

    Here we provide a general criterion for this type of equivalences between microlocal sheaf categories. In known examples of such equivalences between microlocal sheaf categories, the Legendrian stops are required to be related by non-characteristic deformations \cite{DonovanKuwa,ZhouVarI}. However, we provide a criterion that does not a priori require existence of non-characteristic deformations thanks to the equivalences from spherical adjunctions (though it often turns out that the Legendrian stops are related by such deformation).

    Note that any two Legendrian stops are generically disjoint after a small contact perturbation.

\begin{definition}
    Let $\Lambda_\pm \subseteq S^{*}M$ be two disjoint closed subanalytic Legendrian stops. Suppose there exists both a positive and a negative Hamiltonian flow that sends $\Lambda_+$ to an arbitrary small neighbourhood of $\Lambda_-$, and there also exists a positive and a negative Hamiltonian flow that sends $\Lambda_-$ to an arbitrary small neighbourhood of $\Lambda_+$. Then $(\Lambda_-, \Lambda_+)$ is called a swappable pair.
\end{definition}

    Figure \ref{fig:swap-pair} explains the example of a swappable pair $\Lambda_\pm \subseteq S^*T^2$ which is considered in \cite{DonovanKuwa,ZhouVarI}, though from a different perspective.

\begin{figure}[h!]
    \centering
    \includegraphics[width=0.7\textwidth]{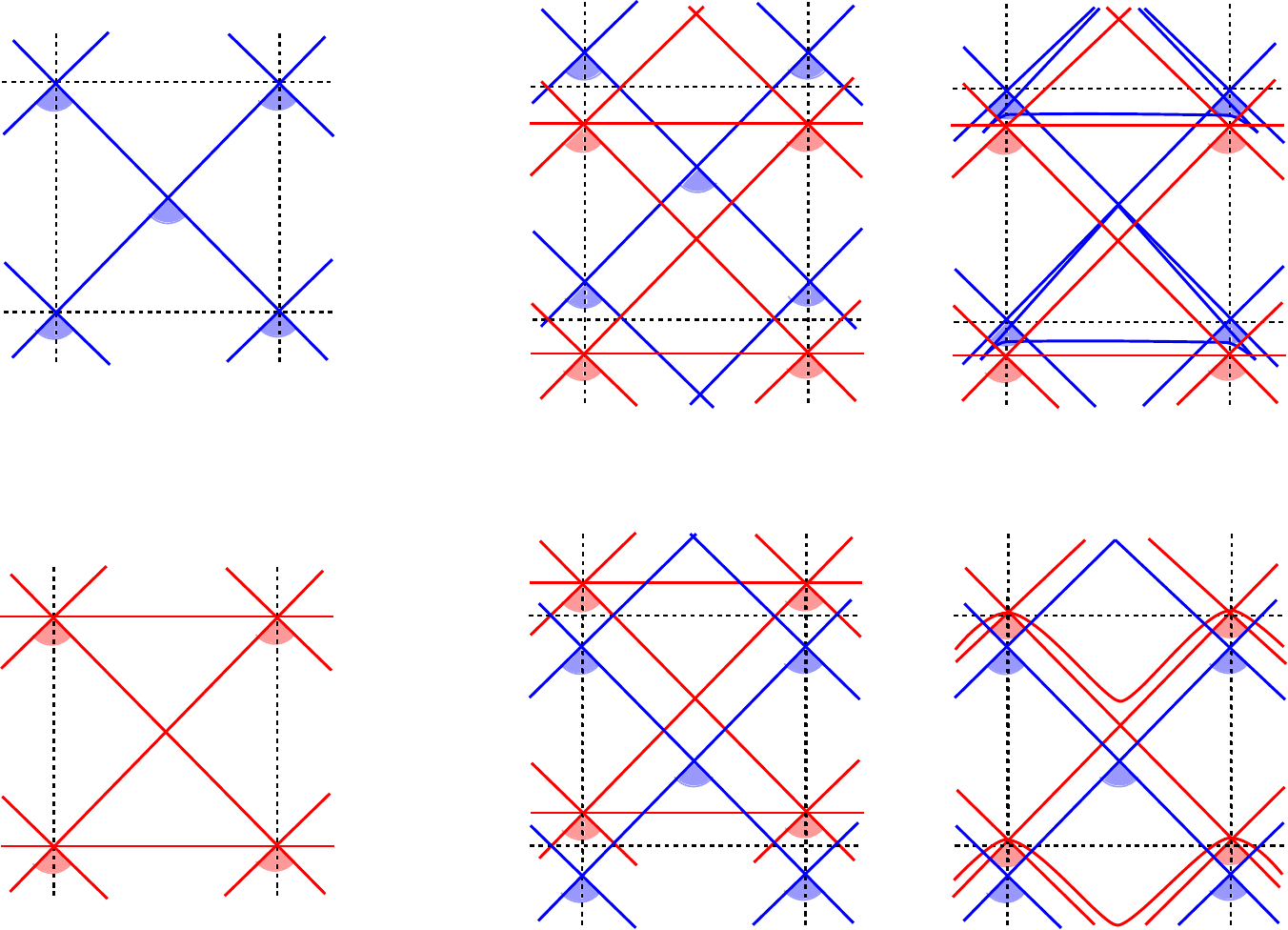}
    \caption{The figure on the left illustrates the swappable pair $\Lambda_\pm \subset S^*T^2$ mirror to the flops associated to $X_0 = \mathbb{C}^2/\mathbb{Z}_2$, where all the covectors are pointing downward. The figure on the right illustrates a cofinal wrapping that sends $T_{-\epsilon}(\Lambda)$ to a neighbourhood of $T_\epsilon(\Lambda)$ and one that sends $T_\epsilon(\Lambda_+)$ to a neighbourhood of $T_{-\epsilon}(\Lambda)$.}\label{fig:swap-pair}
\end{figure}

The main theorem of the section is the following statement, that swappable pairs of Legendrian stops induce spherical pairs of sheaf categories.

\begin{theorem}\label{thm:sphere-pair-var}
    Let $\Lambda_\pm \subseteq S^{*}M$ be a swappable pair of closed Legendrian stops. Then $\Sh_{\Lambda_-}(M) \simeq \Sh_{\Lambda_+}(M), \msh_{\Lambda_-}(\Lambda_-) \simeq \msh_{\Lambda_+}(\Lambda_+)$, and there is a spherical pair
    $$\Sh_{\Lambda_-}(M) \leftarrow \Sh_{\Lambda_+ \cup \Lambda_-}(M) \rightarrow \Sh_{\Lambda_+}(M).$$
\end{theorem}
\begin{proof}
    We notice that since $(\Lambda_-, \Lambda_+)$ is a swappable pair, this implies that $\Lambda_\pm$ are independently swappable in $S^{*}M$: in fact, we can wrap $\Lambda_-$ into a small neighbourhood of $\Lambda_+$, and then follow the wrapping which sends the neighbourhood of $\Lambda_+$ back into a small neighbourhood of $\Lambda_-$. Therefore, the left adjoints of microlocalization
    $$m_{\Lambda_\pm}^l: \msh_{\Lambda_\pm}(\Lambda_\pm) \rightarrow \Sh_{\Lambda_\pm}(M)$$
    are spherical functors whose spherical twists are $S_{\Lambda_\pm}^+$. On the other hand, for any ${F} \in \Sh_{\Lambda_-}(M)$ and ${G} \in \Sh_{\Lambda_+}(M)$, we can define the swapping functors
    $${R}_{\Lambda_-,\Lambda_+}^+({F}) = S_{\Lambda_- \cup \Lambda_+}^+ F = \mathfrak{W}_{\Lambda_+}^+ {F}, \;\; {R}_{\Lambda_+,\Lambda_-}^+({G}) = S_{\Lambda_- \cup \Lambda_+}^+ G = \mathfrak{W}_{\Lambda_-}^+ {G}.$$
    From Corollary \ref{cor:twist}, the compositions of swapping functor give the spherical twists
    $${S}_{\Lambda_-}^+ = {R}_{\Lambda_+,\Lambda_-}^+ \circ {R}_{\Lambda_-,\Lambda_+}^+, \; {S}_{\Lambda_+}^+ = {R}_{\Lambda_-,\Lambda_+}^+ \circ {R}_{\Lambda_+,\Lambda_-}^+.$$
    As a result, we know that ${R}_{\Lambda_-,\Lambda_+}^+, {R}_{\Lambda_+,\Lambda_-}^+$ are equivalences. Similarly we can also consider spherical cotwists and show that the corresponding co-swapping functors are equivalences. This implies that
    $$\Sh_{\Lambda_-}(M) \simeq \Sh_{\Lambda_+}(M), \;\; \msh_{\Lambda_-}(\Lambda_-) \simeq \msh_{\Lambda_+}(\Lambda_+).$$

    Then consider two new pairs of Legendrian stops $(\Lambda_-, T_\epsilon(\Lambda_-))$ and $(T_{-\epsilon}(\Lambda_+), \Lambda_+)$, obtained by a sufficiently small Reeb push-off. We show that there are equivalences
    $$\Sh_{\Lambda_- \cup T_\epsilon(\Lambda_-)}(M) \simeq \Sh_{\Lambda_- \cup \Lambda_+}(M) \simeq \Sh_{T_{-\epsilon}(\Lambda_+) \cup \Lambda_+}(M),$$
    such that the corresponding restrictions are the swapping functor
    $${R}_{\Lambda_-,\Lambda_+}^+: \, \Sh_{\Lambda_-}(M) \rightarrow \Sh_{T_{-\epsilon}(\Lambda_+)}(M), \;\; \Sh_{T_\epsilon(\Lambda_-)}(M) \rightarrow \Sh_{\Lambda_+}(M)$$
    Then since $\Sh_{\Lambda_- \cup T_\epsilon(\Lambda_-)}(M)$ and $\Sh_{T_{-\epsilon}(\Lambda_+) \cup \Lambda_+}(M)$ are both endowed with semi-orthogonal decompositions by Proposition \ref{prop:semi-ortho}, this will complete the proof.
    
    Consider the non-nagative wrapping that fixes $\Lambda_-$ while sending $T_\epsilon(\Lambda_-)$ into $\Lambda_+$, and another non-negative wrapping that fixes $\Lambda_-$ while sending $\Lambda_+$ into $T_{\epsilon}(\Lambda_-)$. Viewing $\Lambda_- \cup T_\epsilon(\Lambda_-)$ and $\Lambda_- \cup \Lambda_+$ independently as two Legendrian stops, the above observation shows that they form a swappable pair as well. Hence we have
    \begin{align*}
    {R}_{\Lambda_- \cup T_\epsilon(\Lambda_-), \Lambda_- \cup \Lambda_+}^+ : \Sh_{\Lambda_- \cup T_\epsilon(\Lambda_-)}(M) \rightarrow \Sh_{\Lambda_- \cup \Lambda_+}(M),\\
    {R}_{\Lambda_- \cup \Lambda_+, \Lambda_- \cup T_\epsilon(\Lambda_-)}^- :  \Sh_{\Lambda_- \cup \Lambda_+}(M) \rightarrow \Sh_{\Lambda_- \cup T_\epsilon(\Lambda_-)}(M),
    \end{align*}
    where ${R}_{\Lambda_- \cup T_\epsilon(\Lambda_-), \Lambda_- \cup \Lambda_+}^+$ and ${R}_{\Lambda_- \cup \Lambda_+, \Lambda_- \cup T_\epsilon(\Lambda_-)}^-$ are inverse equivalences by the definition above. Hence we have shown the equivalence
    $$\Sh_{\Lambda_- \cup \Lambda_+}(M) \simeq \Sh_{\Lambda_- \cup T_\epsilon(\Lambda_-)}(M)$$
    which realizes $\Sh_{\Lambda_- \cup \Lambda_+}(M)$ as semi-orthogonal decompositions. 
    
    Finally, to identify the projection functors with the stop removal functors, i.e.~positive wrapping functors, we only need to notice that the following diagram commutes
    \[\xymatrix{
    \Sh_{\Lambda_-}(M) \ar@{=}[d] & \Sh_{\Lambda_- \cup \Lambda_+}(M) \ar[l]_{\wrap_{\Lambda_-}^+\hspace{10pt}}  \ar@{=}[r] & \Sh_{\Lambda_- \cup \Lambda_+}(M) \ar[r]^{\hspace{10pt}\wrap_{\Lambda_+}^+} \ar[d]_{\rotatebox{90}{$\sim$}}^{\wrap^+_{T_{-\epsilon}(\Lambda_+) \cup \Lambda_+}} & \Sh_{\Lambda_+}(M) \ar@{=}[d] \\
    \Sh_{\Lambda_-}(M) & \Sh_{\Lambda_- \cup T_\epsilon(\Lambda_-)}(M) \ar[u]_{\rotatebox{90}{$\sim$}}^{\wrap^+_{\Lambda_- \cup \Lambda_+}} \ar[l]^{\wrap_{\Lambda_-}^+\hspace{10pt}} \ar[r]^{\sim} & \Sh_{T_{-\epsilon}(\Lambda_+) \cup \Lambda_+}(M) \ar[r]_{\hspace{14pt}\wrap_{\Lambda_+}^+} & \Sh_{\Lambda_+}(M).
    }\]
    This completes the proof of the theorem.
\end{proof}

\section{Example that wrap-once is not equivalence}\label{sec:example}

    When introducing the notion of a swappable stop, Sylvan has already noticed the strong constraint that swappability puts on the stop \cite{SylvanOrlov}. Now we have proved that full stops also implies sphericality. However, it is not know whether being a full or swappable stop is a necessary condition, or even whether any condition is really needed to show spherical adjunction.
    
    Here we provide an example where $S_\Lambda^+: \Sh_\Lambda^c(M) \rightarrow \Sh_\Lambda^c(M)$ is not an equivalence. Moreover, our computation implies that in this case, $m_\Lambda$ does not preserve compact objects and $m_\Lambda^l$ does not preserve proper modules (or equivalently, sheaves with perfect stalks). 

    
\begin{lemma}[{\cite[Lemma 3.29]{Ganatra-Pardon-Shende1}}]
    Let $\Sigma \subseteq S^*M \backslash \Lambda$ be a subanalytic Legendrian, and $\varphi_k$ be an increasing sequence of contact flows on $S^*M \backslash \Lambda$. Suppose there exists a contact form $\alpha$ on $S^*M \backslash \Lambda$ such that
    $$\lim_{k \rightarrow \infty} \int_0^1 \min_{\varphi_k^t(\Sigma)} \alpha \big(\partial_t \varphi_k|_{\varphi_k^t(\Sigma)}\big)\, dt = \infty.$$
    Then $\{\varphi_k\}_{k \in \NN}$ is a cofinal sequence of wrappings in the category of positive wrappings of $\Sigma$ in $S^*M \backslash \Lambda$.
\end{lemma}

\begin{proposition}\label{prop:T2example}
    Let $M = T^n = \RR^n/\ZZ^n$, $\Lambda = S^{*}_{T^m}T^n \subseteq S^{*}T^n\,(n - m \geq 2)$, and $\ol{B_\epsilon(0)}$ be a closed ball of radius $\epsilon$ around $0$. Then $S^-_\Lambda(1_0) = \wrap_\Lambda^- \circ 1_{\ol{B_\epsilon(0)}} \notin \Sh^b_\Lambda(T^n)$. In particular, $S_\Lambda^-$ does not induce an equivalence on the proper subcategory $\Sh^b_\Lambda(T^n)$.
\end{proposition}
\begin{proof}
    Let $\epsilon > 0$ be a small positive number and $\eta_k: T^n \rightarrow [0, 1]$ be a smooth cut-off function such that for small neighbourhoods $\overline{B_{\epsilon/2k^2}(0)} \subset \overline{B_{\epsilon/k^2}(0)}$ around $0$ we have
    $$\eta_k|_{\overline{B_{\epsilon/2k^2}(0)}} = 0, \;\; \eta|_{T^n \backslash \overline{B_{\epsilon/k^2}(0)}} = 1.$$
    Let the decreasing sequence of negative Hamiltonians be $H_k(x, \xi) = -k\,\eta_k(x)|\xi|^2$ and the contact flow be $\varphi_k$. Rescale the contact form on $S^*T^n \backslash S^*_0T^n$ by a function $\delta(x) \rightarrow \infty, \, x \rightarrow 0$. One can easily check by the above lemma that $\varphi_k$ is a cofinal sequence of positive wrappings on $S^{*}T^n \backslash \Lambda$. We will show that $\operatorname{lim}_{k \rightarrow \infty} K(\varphi_{k}) \circ 1_{\overline{B_\epsilon(0)}}$ does not have perfect stalk at $T^n \backslash \{0\}$.
    
    For simplicity, we now assume that $n = 2$. Consider the universal cover $\pi: \RR^n \rightarrow \RR^n /\ZZ^n \cong T^n$. Let the lifting of the Hamiltonian be $\ol{H}_k(x, \xi) = -k\,\eta_k(\pi(x))|\xi|^2$ and the lifting of the flow be $\ol{\varphi}_k$. Then
    $$K({\varphi}_{k}) \circ 1_{\overline{B_\epsilon(0)}} = \pi_*\big(K(\ol{\varphi}_{k}) \circ 1_{\overline{B_\epsilon(0)}}\big).$$
    We then show that in each region of the form $\square_m = [m+\epsilon, m+1-\epsilon] \times [\epsilon, 1-\epsilon]$ where $m \geq 0$, when $k \geq m+1$ we have
    $$1_\cV \hookrightarrow \Gamma\big(\square_m, K(\varphi_{k}) \circ 1_{\overline{B_\epsilon(0)}}\big).$$
    In fact, for the outward unit conormal bundle of $\overline{B_\epsilon(0)}$, under the Hamiltonian $H_k$, the boundary arc of the sector $S_k$ in between the rays $\theta = \arcsin(\epsilon/k^2)$ and $\theta = \arctan(1/k) - \arcsin(\epsilon/k^2(1 + k^2)^{1/2})$ will follow the inverse geodesic flow $\overline{H}_k = -k|\xi|^2$ determined by $\partial_t \varphi_k = k\partial/\partial r$. Therefore there is an injection 
    $$1_\cV \hookrightarrow \Gamma\big(S_k, K(\varphi_{k}) \circ 1_{\overline{B_\epsilon(0)}}\big) \hookrightarrow \Gamma\big(\square_m, K(\ol{\varphi}_{k}) \circ 1_{\overline{B_\epsilon(0)}}\big).$$
    Then under the projection map $\pi: \RR^n \rightarrow \RR^n /\ZZ^n \cong T^n$, write $\square =  [\epsilon, 1-\epsilon] \times [\epsilon, 1-\epsilon] \subset T^n$. We know that when $k \geq m+1$,
    $$1_\cV^{\oplus m} \hookrightarrow \Gamma\big(\square, K(\varphi_{k}) \circ 1_{\overline{B_\epsilon(0)}}\big).$$
    Therefore, $\Gamma\big(\square, \lmi{k \rightarrow \infty} K(\varphi_{k}) \circ 1_{\overline{B_\epsilon(0)}}\big) \notin \cV_0$. Since $\lmi{k \rightarrow \infty} K(\varphi_{k}) \circ 1_{\overline{B_\epsilon(0)}} \in \Sh_{S_0^*T^n}(T^n)$ is constructible, we can conclude that the stalk at $x' \neq x$ is isomorphic to the sections on $\square$, and hence is not perfect.
\end{proof}

\begin{figure}
    \centering
    \includegraphics[width=1.0\textwidth]{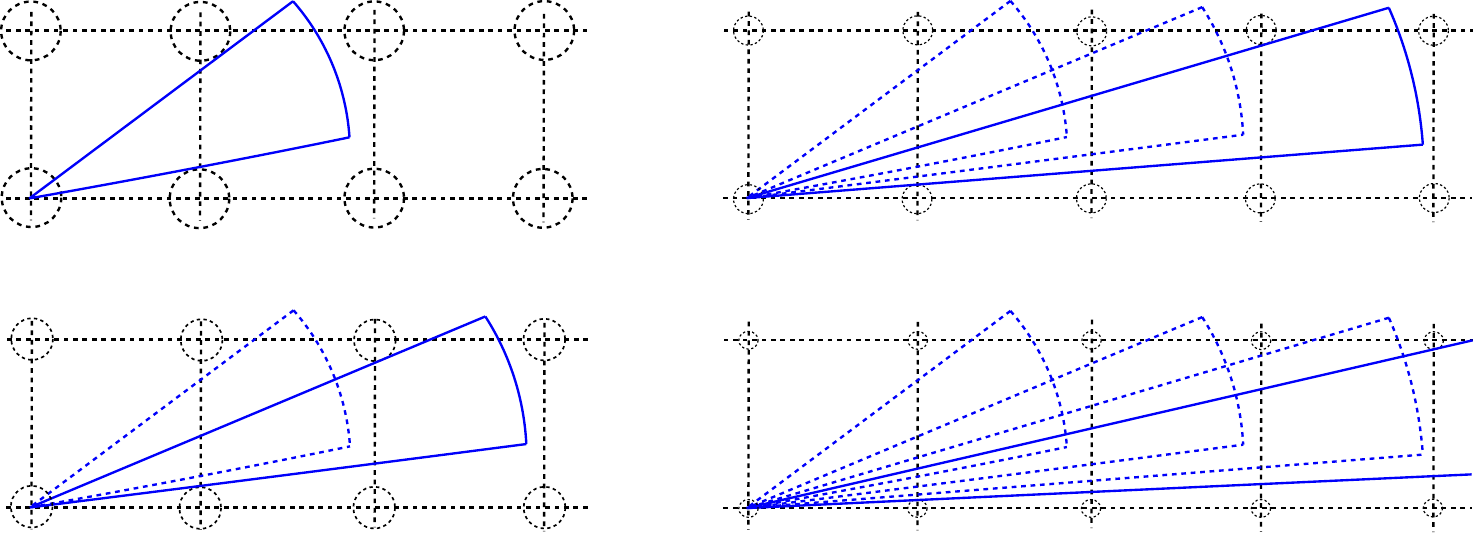}
    \caption{The black circles are the boundary of the regions where the Hamiltonian $H_k$ are cut off by the function $\eta_k$. The blue sectors are the sectors which completely follow the inverse geodesic flow as they do not intersect the black circles. Since radii of the the black circles decreases (and converges to 0), the slope of the lower edge of the sectors that follow the geodesic flow also decreases (and converges to 0). Even though the slope of the upper edge of the sectors are decreasing as more and more black circles appear on the top right pat of the plane, the sequence of sectors will not shrink to nothing and can go arbitrary far away.}\label{fig:my_label}
\end{figure}

\begin{corollary}
    Let $M = T^n$ and $\Lambda = S^{*}_{T^m}T^n \,(n - m \geq 2)$. $S_\Lambda^+$ is not an equivalence on $\Sh^c_\Lambda(T^n)$.
\end{corollary}
\begin{proof}
    Assume that $S_\Lambda^+$ is an equivalence on $\Sh^c_\Lambda(T^n)$. Consider the equivalence $\Sh^b_\Lambda(T^n) \simeq \mathrm{Fun}^{ex}(\Sh^c_\Lambda(T^n)^{op}, \cV_0)$ given by the homomorphism pairing as stated in Theorem \ref{thm:perfcompact}. For $F \in \Sh_\Lambda^c(T^n), G \in \Sh_\Lambda^b(T^n)$, since
    $$Hom(S_\Lambda^+(F), G) = Hom(F, S_\Lambda^-(G)),$$
    we know that $S_\Lambda^-$ has to be an equivalence on $\Sh^b_\Lambda(T^n)$. This contradicts the proposition.
\end{proof}
\begin{corollary}
    Let $M = T^n$ and $\Lambda = S^{*}_{T^m}T^n\,(n - m \geq 2)$. Then
    $m_\Lambda$ does not preserve compact objects.
\end{corollary}
\begin{proof}
    This follows immediately from the fiber sequence $m^l_\Lambda \circ m_\Lambda \rightarrow \mathrm{id}_{\Sh_\Lambda(T^n)} \rightarrow S^+_\Lambda$.
\end{proof}
\begin{remark}
    We believe one can also show that $m_\Lambda^l$ does not perserve proper modules (or objects with perfect stalks) using a similar argument.
\end{remark}

    We can then deduce the following geometric result which shows that the Weinstein stop is not a swappable stop.

\begin{corollary}
    Let $M = T^n$ and $\Lambda = S^{*}_{T^k}T^n\,(n-k \geq 2)$. Then the Weinstein hypersurface $F_\Lambda$ defined as the ribbon of $\Lambda \subseteq S^{*}T^n$ is not a swappable hypersurface.
\end{corollary}

    We compare our result with the result of Dahinden \cite{BS-LegIsotopy}. Dahinden's theorem states that in this setting, there does not exist a positive Legendrian loop $\Lambda_t \subseteq S^*T^n$ such that $\Lambda_0 = \Lambda_1 = S_0^*T^n$ and $\Lambda_t \cap S_0^*T^n = \varnothing$ for $0 \leq t \leq 1$. However, it does not imply that the Weinstein ribbon $F_\Lambda$ of $\Lambda \subseteq S^{*}M$ is a swappable hypersurface, as it is in general unknown whether the exact symplectomorphism $F_\Lambda$ defined by the positive loop sends the zero section to itself. Therefore, our corollary is at least a priori stronger than the theorem in the case of $M = T^n$.


\bibliographystyle{amsplain}
\bibliography{ref_KuoLi}


\end{document}